\documentclass[11pt, twoside, leqno]{article}

\usepackage{amsmath}
\usepackage{amssymb}
\usepackage{titletoc}
\usepackage{mathrsfs}
\usepackage{amsthm}
\usepackage{color}
\usepackage{latexsym,amssymb}

\usepackage{indentfirst}
\usepackage{color}
\usepackage{txfonts}

\allowdisplaybreaks

\def\bint{{\ifinner\rlap{\bf\kern.30em--}
\int\else\rlap{\bf\kern.35em--}\int\fi}\ignorespaces}

\def\sbint{{\ifinner\rlap{\bf\kern.32em--}
\hspace{0.078cm}\int\else\rlap{\bf\kern.45em--}\int\fi}\ignorespaces}

\def\red{\color{red}}

\def\rr{{\mathbb R}}
\def\rn{{\mathbb{R}^n}}
\def\cc{{\mathbb C}}

\def\nn{{\mathbb N}}
\def\zz{{\mathbb Z}}

\def\fz{\infty }

\def\lz{\lambda}

\def\lf{\left}
\def\r{\right}
\def\ls{\lesssim}
\def\noz{\nonumber}
\def\wz{\widetilde}

\def\loc{{\mathrm{loc}}}
\DeclareMathOperator{\supp}{supp}

\def\XXint#1#2#3{{\setbox0=\hbox{$#1{#2#3}{\int}$ }
\vcenter{\hbox{$#2#3$ }}\kern-.6\wd0}}

\def\f{\frac}

\def\lz{{\lambda}}

\newtheorem{theorem}{Theorem}[section]
\newtheorem{lemma}[theorem]{Lemma}

\newtheorem{proposition}[theorem]{Proposition}

\newtheorem{assumption}[theorem]{Assumption}
\theoremstyle{definition}
\newtheorem{remark}[theorem]{Remark}
\newtheorem{definition}[theorem]{Definition}
\renewcommand{\appendix}{\par
\setcounter{section}{0}%
\setcounter{subsection}{0}%
\setcounter{subsubsection}{0}%
\gdef\thesection{\@Alph\c@section}%
\gdef\thesubsection{\@Alph\c@section.\@arabic\c@subsection}%
\gdef\theHsection{\@Alph\c@section.}%
\gdef\theHsubsection{\@Alph\c@section.\@arabic\c@subsection}%
\csname appendixmore\endcsname
}

\numberwithin{equation}{section}

\begin{document}

\arraycolsep=1pt

\title{\bf\Large
Boundedness of Fractional Integrals on
Ball Campanato-Type Function Spaces
\footnotetext{\hspace{-0.35cm} 2020 {\it
Mathematics Subject Classification}. Primary 47G40;
Secondary 42B20, 47A30, 42B30, 46E35, 42B25, 42B35.
\endgraf {\it Key words and phrases.}
fractional integral, ball quasi-Banach function space, Campanato-type space,
Hardy-type space.
\endgraf This project is partially supported by
the National Natural Science Foundation of China (Grant Nos.\
11971058 and 12071197) and the National
Key Research and Development Program of China
(Grant No.\ 2020YFA0712900).}}
\date{}
\author{Yiqun Chen, Hongchao Jia and Dachun Yang\footnote{Corresponding author,
E-mail: \texttt{dcyang@bnu.edu.cn}/{\red June 13, 2022}/Final version.}}
\maketitle

\vspace{-0.8cm}

\begin{center}
\begin{minipage}{13cm}
{\small {\bf Abstract}\quad
Let $X$ be a ball quasi-Banach function space on ${\mathbb R}^n$
satisfying some mild assumptions
and let $\alpha\in(0,n)$ and $\beta\in(1,\infty)$.
In this article, when $\alpha\in(0,1)$, the authors first find a reasonable version
$\widetilde{I}_{\alpha}$ of the fractional integral $I_{\alpha}$
on the ball Campanato-type function space $\mathcal{L}_{X,q,s,d}(\mathbb{R}^n)$
with $q\in[1,\infty)$, $s\in\mathbb{Z}_+^n$, and $d\in(0,\infty)$.
Then the authors prove that
$\widetilde{I}_{\alpha}$ is bounded
from $\mathcal{L}_{X^{\beta},q,s,d}(\mathbb{R}^n)$ to $\mathcal{L}_{X,q,s,d}(\mathbb{R}^n)$
if and only if there exists a positive constant $C$ such that,
for any ball $B\subset \mathbb{R}^n$,
$|B|^{\frac{\alpha}{n}}\leq C
\|\mathbf{1}_B\|_X^{\frac{\beta-1}{\beta}}$,
where $X^{\beta}$ denotes the $\beta$-convexification of $X$.
Furthermore, the authors extend the range $\alpha\in(0,1)$ in $\wz I_\alpha$ to
the range $\alpha\in(0,n)$ and also obtain the
corresponding boundedness in this case.
Moreover, $\wz I_\alpha$ is proved to be the adjoint operator of $I_\alpha$.
All these results have a wide range of applications.
Particularly, even when they are
applied, respectively, to
Morrey spaces, mixed-norm Lebesgue spaces,
local generalized Herz spaces, and mixed-norm Herz spaces, all
the obtained results are new. The proofs of these results strongly
depend on the dual theorem on $\mathcal{L}_{X,q,s,d}(\mathbb{R}^n)$
and also on the special atomic decomposition of molecules of
$H_X(\mathbb{R}^n)$ (the Hardy-type
space associated with $X$) which proves the predual space of $\mathcal{L}_{X,q,s,d}(\mathbb{R}^n)$.
}
\end{minipage}
\end{center}

\vspace{0.2cm}



\section{Introduction\label{s-intro}}

Recall that John and Nirenberg \cite{JN} introduced the well-known
space $\mathop{\mathrm{BMO}\,}(\rn)$ of functions with bounded mean oscillation,
which proves the dual space of the Hardy space $H^1(\rn)$ by Fefferman and Stein in their seminal artical \cite{FS72}.
Later, Taibleson and Weiss \cite{tw80} gave a more general result that, for any
$p\in(0,1]$, $q\in(1,\fz]$, and $s$ being a nonnegative integer not smaller than $n(\frac{1}{p}-1)$,
the dual space of the Hardy space $H^p(\rn)$ is $\mathcal{C}_{\frac{1}{p}-1,q,s}(\rn)$,
the Campanato space introduced in \cite{C}, which coincides with
$\mathop{\mathrm{BMO}\,}(\rn)$ when $p=1$.

It is well known that the real-variable theory,
including the boundedness of fractional integrals
on Hardy-type spaces, plays a fundamental role in harmonic analysis and partial
differential equations (see, for instance,
\cite{ans2021,chy2021,Ho2022,N17,EMS1970,EMS2}).
Moreover, in recent years,
the boundedness of fractional integrals on
Campanato-type spaces has attracted more and more attention.
For instance, Nakai \cite{N10} studied fractional integrals
on Campanato-type spaces with variable growth conditions
over spaces of homogeneous type in the sense of Coifman and Weiss.
In addition,
Arai and Nakai \cite{an2019} studied the
compact commutators of generalized fractional integrals,
with a function in generalized Campanato
spaces, on generalized Morrey spaces.
Recently, Ho \cite{ho2019} studied the mapping properties of several kind of integral operators
on both the space $\mathop{\mathrm{BMO}\,}(\rn)$ and the Campanato space.
We refer the reader also to \cite{HSS2016,ho2019,NS2012,yn2021,yn2022}
for more studies on this.

On the other hand,
Sawano et al. \cite{SHYY}
introduced the ball quasi-Banach function space $X$ and
the related Hardy space $H_X(\rn)$.
Compared with quasi-Banach function spaces,
ball quasi-Banach function spaces contain more function spaces, for
instance, the weighted Lebesgue space, the Morrey space, the mixed-norm Lebesgue space,
the Orlicz--slice space, and the Musielak--Orlicz space are
ball quasi-Banach function space, but they may not be quasi-Banach function spaces (see \cite{SHYY,zyyw,zhyy2022} for the details),
which implies that the results obtained in \cite{SHYY} have more wide applications.
We refer the reader to
\cite{CWYZ2020,dgpyyz,ho2021,is2017,tyyz,wyy,wyyz,zyyw}
for more studies on both $H_X(\rn)$ and $WH_X(\rn)$ (the weak
Hardy spaces associated with $X$)
and to \cite{syy,syy2,yhyy,yhyy1} for
(weak) Hardy spaces associated with $X$ on
spaces of homogeneous type in
the sense of Coifman and Weiss.

Moreover, to obtain a counterpart of the dual theorem
$(H^p(\rn))^*=\mathcal{C}_{\frac{1}{p}-1,q,s}(\rn)$ for the Hardy space $H_X(\rn)$,
Yan et al. \cite{yyy20} introduced the Campanato-type function space
$\mathcal{L}_{X,q,s}(\rn)$ associated with $X$,
which proves the dual space of $H_X(\rn)$ when $X$ is concave.
Very recently, Zhang et al. \cite{zhyy2022} introduced
the ball Campanato-type function space $\mathcal{L}_{X,q,s,d}(\rn)$
associated with $X$ and proved that $(H_X(\rn))^*=\mathcal{L}_{X,q,s,d}(\rn)$.
This result improves the dual theorem in \cite{yyy20}
via removing the assumption that $X$ is concave.
Moreover, the Carleson measure characterization of $\mathcal{L}_{X,q,s,d}(\rn)$
was also obtained in \cite{zhyy2022}.
However, the boundedness of fractional integrals on
$\mathcal{L}_{X,q,s,d}(\rn)$ is still unknown
even when $X$ is some concrete function spaces, for instance,
the Morrey space, the mixed-norm Lebesgue space,
the local generalized Herz space, and the mixed-norm Herz space.

In this article, we study the boundedness of fractional integrals
on $\mathcal{L}_{X,q,s,d}(\rn)$.
To be precise, assume that
the (powered) Hardy--Littlewood maximal operator satisfies some Fefferman--Stein
vector-valued maximal inequality on $X$ and is bounded on the associate space of $X$.
We first find a reasonable version
$\widetilde{I}_{\alpha}$ of the fractional integral $I_{\alpha}$, with $\alpha\in(0,1)$,
on the ball Campanato-type function space
$\mathcal{L}_{X,q,s,d}(\mathbb{R}^n)$ (see Definition \ref{I-w} below),
which is a generalization of $\wz I_\alpha$ studied by Nakai
\cite{N01} (see also \cite{gv} for a variant on spaces of homogeneous type).
Let $\beta\in(1,\fz)$. Then
we prove that $\widetilde{I}_{\alpha}$ is bounded
from $\mathcal{L}_{X^{\beta},q,s,\beta d}(\mathbb{R}^n)$ to $\mathcal{L}_{X,q,s,d}(\mathbb{R}^n)$
if and only if there exists a positive constant $C$ such that,
for any ball $B\subset \mathbb{R}^n$,
$|B|^{\frac{\alpha}{n}}\leq C
\|\mathbf{1}_B\|_X^{\frac{\beta-1}{\beta}}$,
where $X^{\beta}$ denotes the $\beta$-convexification of $X$
(see Theorem \ref{thm-main-I} below).
Furthermore, we extend the range
$\alpha\in(0,1)$ in $\wz I_\alpha$ to
the range $\alpha\in(0,n)$ and also obtain the
corresponding boundedness in this case
(see Theorem \ref{main-corollary} below).
Moreover, $\widetilde{I}_{\alpha}$ is proved
to be the adjoint operator of $I_{\alpha}$
(see Theorem \ref{dual-I-wI} below).
All these results have a wide range of applications.
Particularly, even when they are
applied, respectively, to
Morrey spaces, mixed-norm Lebesgue spaces,
local generalized Herz spaces, and mixed-norm Herz spaces, all
the obtained results are new. Obviously, due
to the generality and the flexibility, more applications
of these main results of this article are predictable.
Moreover, we refer the reader to \cite{cjy-01}
for studies on fractional integrals on $H_{X}(\rn)$.
To limit the length of this article, the boundedness of Calder\'on--Zygmund
operators on $\mathcal{L}_{X,q,s,d}(\rn)$
will be studied in the forthcoming article
\cite{cjy2022cz}.

One of the main contribution of this work is to prove
that $\widetilde{I}_{\alpha}$ is the adjoint operators of $I_{\alpha}$.
This further strengthens the rationality of the definition of
$\widetilde{I}_{\alpha}$. One key step
is to show that, for any $g\in (H_X(\mathbb{R}^n))^*$ and any
molecule $M$ of $H_X(\mathbb{R}^n)$,
$\langle g,M\rangle=\int_{\rn}g(x)M(x)\,dx$
(see Lemma \ref{thm-dual-T} below). To this end, we skillfully use
the dual theorem on $\mathcal{L}_{X,q,s,d}(\mathbb{R}^n)$ obtained in \cite{cjy-01},
and the special atomic decomposition of molecules of $H_X(\mathbb{R}^n)$
(see Proposition \ref{HK-mole} below) to obtain the above desired conclusion.

To be precise, the remainder of this article is organized as follows.

In Section \ref{s1}, we recall some basic concepts
on the quasi-Banach function space $X$
and the ball Campanato-type space $\mathcal{L}_{X,q,s,d}(\rn)$.
We also present two mild assumptions that the
(powered) Hardy--Littlewood maximal operator satisfies the Fefferman--Stein
vector-valued maximal inequality on $X$ and is bounded on the associate space of $X$.

The main aim of Section \ref{sec-frac-B}
is to study the mapping properties of fractional integrals on
the ball Campanato-type space $\mathcal{L}_{X,q,s,d}(\rn)$.
To be precise, in Subsection \ref{sec-I}, we first give a reasonable version
of the fractional integral $I_{\alpha}$,
with $\alpha\in (0,1)$, on the Campanato space $\mathcal{L}_{X,q,s,d}(\rn)$,
denoted by $\widetilde{I}_{\alpha}$, and  establish a
sufficient and necessary condition on the boundedness of
$\widetilde{I}_{\alpha}$ from $\mathcal{L}_{X^{\beta},q,s,\beta d}(\rn)$
to $\mathcal{L}_{X,q,s,d}(\rn)$ with some $\beta\in(1,\fz)$
(see Theorem \ref{thm-main-I} below).
Then we extend the range of $\alpha\in(0,1)$ in $\wz I_\alpha$ to
the range $\alpha\in(0,n)$ and also obtain the corresponding
boundedness in this case
(see Theorem \ref{main-corollary} below).
In Subsection \ref{sec-I-2}, we prove that $\widetilde{I}_{\alpha}$ is
the adjoint operator of $I_{\alpha}$ (see Theorem \ref{dual-I-wI} below).

In Section \ref{Appli}, we apply all the above main results to four concrete examples
of ball quasi-Banach function spaces,
namely, the Morrey space $M_r^p(\rn)$,
the mixed-norm Lebesgue space $L^{\vec{p}}(\rn)$,
the local generalized Herz space
$\dot{\mathcal{K}}_{\omega,\mathbf{0}}^{p,r}(\mathbb{R}^{n})$,
and the mixed-norm Herz space $\dot{E}^{\vec{\alpha},\vec{p}}_{\vec{q}}(\rn)$,
respectively. Therefore, the boundedness of $\wz I_{\alpha}$,
respectively, on $\mathcal{L}_{M_r^p(\rn),q,s,d}(\rn)$, $\mathcal{L}_{L^{\vec{p}}(\rn),q,s,d}(\rn)$,
$\mathcal{L}_{\dot{\mathcal{K}}_{\omega,\mathbf{0}}^{p,r}(\mathbb{R}^{n}),q,s,d}(\rn)$,
and $\mathcal{L}_{\dot{E}^{\vec{\alpha},\vec{p}}_{\vec{q}}(\rn),q,s,d}(\rn)$
is obtained (see, respectively, Theorems \ref{ap-M}, \ref{apply2},
\ref{apply6}, and \ref{apply8} below).

Finally, we make some conventions on notation. Let
$\nn:=\{1,2,\ldots\}$, $\zz_+:=\nn\cup\{0\}$, and $\mathbb{Z}^n_+:=(\mathbb{Z}_+)^n$.
We always denote by $C$ a \emph{positive constant}
which is independent of the main parameters,
but it may vary from line to line.
The symbol $f\lesssim g$ means that $f\le Cg$.
If $f\lesssim g$ and $g\lesssim f$, we then write $f\sim g$.
If $f\le Cg$ and $g=h$ or $g\le h$,
we then write $f\ls g\sim h$
or $f\ls g\ls h$, \emph{rather than} $f\ls g=h$
or $f\ls g\le h$. For any $s\in\zz_+$, we use $\mathcal{P}_s(\rn)$
to denote the set of all polynomials on $\rn$
with total degree not greater than $s$. We use $\mathbf{0}$ to
denote the \emph{origin} of $\rn$.
For any measurable subset $E$ of $\rn$, we denote by $\mathbf{1}_E$ its
characteristic function. Moreover, for any $x\in\rn$ and $r\in(0,\fz)$,
let $B(x,r):=\{y\in\rn:\ |y-x|<r\}$.
Furthermore, for any $\lambda\in(0,\infty)$
and any ball $B(x,r)\subset\rn$ with $x\in\rn$ and
$r\in(0,\fz)$, let $\lambda B(x,r):=B(x,\lambda r)$.
Finally, for any $q\in[1,\infty]$,
we denote by $q'$ its \emph{conjugate exponent},
namely, $\frac{1}{q}+\frac{1}{q'}=1$.

\section{Preliminaries\label{s1}}

In this section, we recall the definitions of ball quasi-Banach function spaces and their related
ball Campanato-type function spaces.
In what follows, we use $\mathscr M(\rn)$ to denote the set of
all measurable functions on $\rn$.
For any $x\in\rn$ and $r\in(0,\infty)$, let $B(x,r):=\{y\in\rn:\ |x-y|<r\}$ and
\begin{equation}\label{Eqball}
\mathbb{B}(\rn):=\lf\{B(x,r):\ x\in\rn \text{ and } r\in(0,\infty)\r\}.
\end{equation}
The following concept of ball quasi-Banach function
spaces on $\rn$
is just \cite[Definition 2.2]{SHYY}.

\begin{definition}\label{Debqfs}
Let $X\subset\mathscr{M}(\rn)$ be a quasi-normed linear space
equipped with a quasi-norm $\|\cdot\|_X$ which makes sense for all
measurable functions on $\rn$.
Then $X$ is called a \emph{ball quasi-Banach
function space} on $\rn$ if it satisfies:
\begin{enumerate}
\item[$\mathrm{(i)}$] if $f\in\mathscr{M}(\rn)$, then $\|f\|_{X}=0$ implies
that $f=0$ almost everywhere;
\item[$\mathrm{(ii)}$] if $f, g\in\mathscr{M}(\rn)$, then $|g|\le |f|$ almost
everywhere implies that $\|g\|_X\le\|f\|_X$;
\item[$\mathrm{(iii)}$] if $\{f_m\}_{m\in\nn}\subset\mathscr{M}(\rn)$ and $f\in\mathscr{M}(\rn)$,
then $0\le f_m\uparrow f$ almost everywhere as $m\to\infty$
implies that $\|f_m\|_X\uparrow\|f\|_X$ as $m\to\infty$;
\item[$\mathrm{(iv)}$] $B\in\mathbb{B}(\rn)$ implies
that $\mathbf{1}_B\in X$,
where $\mathbb{B}(\rn)$ is the same as in \eqref{Eqball}.
\end{enumerate}

Moreover, a ball quasi-Banach function space $X$
is called a
\emph{ball Banach function space} if it satisfies:
\begin{enumerate}
\item[$\mathrm{(v)}$] for any $f,g\in X$,
\begin{equation*}
\|f+g\|_X\le \|f\|_X+\|g\|_X;
\end{equation*}
\item[$\mathrm{(vi)}$] for any ball $B\in \mathbb{B}(\rn)$,
there exists a positive constant $C_{(B)}$,
depending on $B$, such that, for any $f\in X$,
\begin{equation*}
\int_B|f(x)|\,dx\le C_{(B)}\|f\|_X.
\end{equation*}
\end{enumerate}
\end{definition}

\begin{remark}\label{rem-ball-B}
\begin{enumerate}
\item[$\mathrm{(i)}$] Let $X$ be a ball quasi-Banach
function space on $\rn$. By \cite[Remark 2.6(i)]{yhyy1},
we conclude that, for any $f\in\mathscr{M}(\rn)$, $\|f\|_{X}=0$ if and only if $f=0$
almost everywhere.

\item[$\mathrm{(ii)}$] As was mentioned in
\cite[Remark 2.6(ii)]{yhyy1}, we obtain an
equivalent formulation of Definition \ref{Debqfs}
via replacing any ball $B$ by any
bounded measurable set $E$ therein.

\item[$\mathrm{(iii)}$] We should point out that,
in Definition \ref{Debqfs}, if we
replace any ball $B$ by any measurable set $E$ with
finite measure, we obtain the
definition of (quasi-)Banach function spaces which were originally
introduced in \cite[Definitions 1.1 and 1.3]{BS88}. Thus,
a (quasi-)Banach function space
is also a ball (quasi-)Banach function
space and the converse is not necessary to be true.
\item[$\mathrm{(iv)}$] By \cite[Theorem 2]{dfmn2021},
we conclude that both (ii) and (iii) of
Definition \ref{Debqfs} imply that any ball quasi-Banach
function space is complete.
\end{enumerate}
\end{remark}

The associate space $X'$ of any given ball
Banach function space $X$ is defined as follows
(see \cite[Chapter 1, Section 2]{BS88} or \cite[p.\,9]{SHYY}).

\begin{definition}\label{de-X'}
For any given ball quasi-Banach function space $X$, its \emph{associate space}
 (also called the
\emph{K\"othe dual space}) $X'$ is defined by setting
\begin{equation*}
X':=\lf\{f\in\mathscr M(\rn):\ \|f\|_{X'}<\infty\r\},
\end{equation*}
where, for any $f\in X'$,
$$\|f\|_{X'}:=\sup\lf\{\lf\|fg\r\|_{L^1(\rn)}:\ g\in X,\ \|g\|_X=1\r\},$$
and $\|\cdot\|_{X'}$ is called the \emph{associate norm} of $\|\cdot\|_X$.
\end{definition}

\begin{remark}\label{bbf}
From \cite[Proposition 2.3]{SHYY}, we deduce that, if $X$ is a ball
Banach function space, then its associate space $X'$ is also a ball
Banach function space.
\end{remark}

We also recall the concepts of both the convexity and
the concavity of ball quasi-Banach function spaces,
which are a part of \cite[Definition 2.6]{SHYY}.

\begin{definition}\label{Debf}
Let $X$ be a ball quasi-Banach function space and $p\in(0,\infty)$.
\begin{enumerate}
\item[(i)] The \emph{$p$-convexification} $X^p$ of $X$
is defined by setting
$$X^p:=\lf\{f\in\mathscr M(\rn):\ |f|^p\in X\r\}$$
equipped with the \emph{quasi-norm} $\|f\|_{X^p}:=\|\,|f|^p\,\|_X^{1/p}$ for any $f\in X^p$.

\item[(ii)] The space $X$ is said to be
\emph{concave} if there exists a positive constant
$C$ such that, for any $\{f_k\}_{k\in{\mathbb N}}\subset \mathscr M(\rn)$,
$$\sum_{k=1}^{{\infty}}\|f_k\|_{X}
\le C\left\|\sum_{k=1}^{{\infty}}|f_k|\right\|_{X}.$$
In particular, when $C=1$, $X$ is said to be
\emph{strictly concave}.
\end{enumerate}
\end{definition}

\begin{remark}
It is easy to show that, for any ball quasi-Banach function space $X$ and any $p\in(0,\infty)$, the
$p$-convexification $X^p$ of $X$ is also a ball quasi-Banach function space.
\end{remark}

Next, we present the concept of ball Campanato-type function spaces
associated with ball quasi-Banach function spaces
introduced in \cite[Definition 3.2]{zhyy2022} (see also \cite[Definition 1.11]{yyy20}).
Recall that the \emph{Lebesgue space} $L^q(\rn)$ with
$q\in(0,\infty]$
is defined to be the set of all the measurable functions $f$ on $\rn$
such that
$$\|f\|_{L^q(\rn)}:=
\begin{cases}
\displaystyle
\lf[\int_{\rn}|f(x)|^q\, dx\r]^{\frac{1}{q}}
&\text{if}\quad q\in(0,\fz),\\
\displaystyle
\mathop{\mathrm{ess\,sup}}_{x\in\rn}\,|f(x)|
&\text{if}\quad q=\fz
\end{cases}$$
is finite. For any given $q\in(0,\infty]$,
$L^q_{\mathrm{loc}}(\rn)$ denotes
the set of all the measurable functions $f$ such that
$f\mathbf{1}_E\in L^q(\rn)$
for any bounded measurable set $E\subset \rn$
and, for any $f\in L_{\loc}^1(\rn)$
and any finite measurable subset $E\subset\rn$, let
\begin{equation}\label{fe}
f_E:=\fint_Ef(x)\,dx:=\frac{1}{|E|}\int_E f(x)\,dx.
\end{equation}
Moreover, for any $s\in\zz_+$, we use $\mathcal{P}_s(\rn)$
to denote the set of all the polynomials on $\rn$
with total degree not greater than $s$; for any ball
$B\in\mathbb{B}(\rn)$ and
any locally integrable function $g$ on $\rn$,
$P^{(s)}_B(g)$ denotes the \emph{minimizing polynomial} of
$g$ with total degree not greater than $s$, which means
that $P^{(s)}_B(g)$ is the unique polynomial $f\in\mathcal{P}_s(\rn)$
such that, for any $P\in\mathcal{P}_s(\rn)$,
\begin{align*}
\int_{B}[g(x)-f(x)]P(x)\,dx=0.
\end{align*}

Recall that, in \cite[Definition 1.11]{yyy20}, Yan et al. introduced the following
Campanato space $\mathcal{L}_{X,q,s}(\rn)$ associated with $X$.

\begin{definition}\label{cqb}
Let $X$ be a ball quasi-Banach function space, $q\in[1,\fz)$,
and $s\in\zz_+$. The \emph{Campanato space}
$\mathcal{L}_{X,q,s}(\rn)$, associated with $X$, is defined to be
the set of all the $f\in L^q_{\mathrm{loc}}(\rn)$ such that
$$\|f\|_{\mathcal{L}_{X,q,s}(\rn)}
:=\sup_{B\subset\rn}\frac{|B|}{\|\mathbf{1}_B\|_X}\lf\{
\fint_{B}\lf|f(x)-P^{(s)}_B(f)(x)\r|^q\,dx\r\}^{1/q}<\infty,$$
where the supremum is taken over all balls $B\in\mathbb{B}(\rn)$.
\end{definition}

In what follows, by abuse of notation, we identify
$f\in\mathcal{L}_{X,q,s}(\rn)$ with $f+\mathcal{P}_s(\rn)$.

\begin{remark}
Let $q\in[1,\infty)$, $s\in\zz_+$, and $\alpha\in[0,\fz)$.
Recall that Campanato \cite{C} introduced the \emph{Campanato space}
$\mathcal{C}_{\alpha,q,s}(\rn)$ which is defined to be the
set of all the $f\in L^q_\loc(\rn)$ such that
\begin{align}\label{campanato}
\|f\|_{\mathcal{C}_{\alpha,q,s}(\rn)} \
:=\sup |B|^{-\alpha}\left[\fint_{B}\left|f(x)
-P^{(s)}_B(f)(x)\right|^{q}\right]^{\frac{1}{q}}<\infty,
\end{align}
where the supremum is taken over all open balls (or cubes) $B\subset\rn$.
It is well known that $\mathcal{C}_{\alpha,q,s}(\rn)$ when $\alpha=0$
coincides with the space $\mathrm{BMO}\,(\rn)$.
Moreover, when $X:=L^{\frac{1}{\alpha+1}}(\rn)$, then
$\mathcal{L}_{X,q,s}(\rn)=\mathcal{C}_{\alpha,q,s}(\rn)$.
\end{remark}

\begin{remark}\label{rem-ball-B2}
Let $f\in L^{1}_{\mathrm{loc}}(\rn)$, $s\in\zz_+$, and
$B\in \mathbb{B}(\rn)$. By a claim in \cite[p.\,83]{tw80},
we find that there exists a positive
constant $C$, independent of both $f$ and $B$, such that
\begin{align*}
\left\|P_{B}^{(s)}(f)\right\|_{L^{\fz}(B)}
\le C\fint_B|f(x)|\,dx.
\end{align*}
\end{remark}

Very recently, in \cite[Definition 3.2]{zhyy2022}, Zhang et al.
introduced the following
ball Campanato-type function space $\mathcal{L}_{X,q,s,d}(\rn)$ associated with $X$.

\begin{definition}\label{2d2}
Let $X$ be a ball quasi-Banach function space, $q\in[1,{\infty})$,
$d\in(0,\infty)$,
and $s\in\zz_+$. Then the \emph{ball Campanato-type function space}
$\mathcal{L}_{X,q,s,d}(\rn)$, associated with $X$, is defined to be
the set of all the $f\in L^q_{\mathrm{loc}}({{\rr}^n})$ such that
\begin{align*}
\|f\|_{\mathcal{L}_{X,q,s,d}(\rn)}
:&=\sup
\lf\|\lf\{\sum_{i=1}^m
\lf(\frac{{\lambda}_i}{\|{\mathbf{1}}_{B_i}\|_X}\r)^d
{\mathbf{1}}_{B_i}\r\}^{\frac1d}\r\|_{X}^{-1}\\
&\quad\times\sum_{j=1}^m\frac{{\lambda}_j|B_j|}{\|{\mathbf{1}}_{B_j}
\|_{X}}
\lf[\fint_{B_j}\lf|f(x)-P^{(s)}_{B_j}(f)(x)\r|^q \,dx\r]^\frac1q
\end{align*}
is finite, where the supremum is taken over all
$m\in\nn$, $\{B_j\}_{j=1}^m\subset \mathbb{B}(\rn)$, and
$\{\lambda_j\}_{j=1}^m\subset[0,\infty)$ with
$\sum_{j=1}^m\lambda_j\neq0$.
\end{definition}

In what follows, by abuse of notation, we identify
$f\in\mathcal{L}_{X,q,s,d}(\rn)$ with $f+\mathcal{P}_s(\rn)$.

\begin{remark}\label{rem-ball-B3}
Let $q\in[1,\fz)$, $s\in\zz_+$, $d\in(0,\fz)$, and
$X$ be a ball quasi-Banach function space. It
is easy to show that
\begin{align}\label{sub-01}
\mathcal{L}_{X,q,s,d}({{\rr}^n})\subset\mathcal{L}_{X,q,s}(\rn).
\end{align}
Moreover, when $X$ is concave and $d\in(0,1]$, it was proved in
\cite[Proposition 3.7]{zhyy2022} that
$$\mathcal{L}_{X,q,s,d}({{\rr}^n})=\mathcal{L}_{X,q,s}(\rn)$$
with equivalent quasi-norms.
\end{remark}

In this article, we need the following two mild assumptions
about the
boundedness of the (powered) Hardy--Littlewood maximal operator
on ball quasi-Banach function spaces.
To this end, recall that the \emph{Hardy-Littlewood maximal operator}
$\mathcal{M}$ is defined by setting, for any $f\in
L_{\mathrm{loc}}^1(\rn)$ and
$x\in\rn$,
\begin{align*}
\mathcal{M}(f)(x):=\sup_{B\ni x}\frac{1}{|B|}\int_{B}
|f(y)|dy,
\end{align*}
where the supremum is taken over all the balls $B\in\mathbb{B}(\rn)$
containing $x$.

\begin{assumption}\label{assump1}
Let $X$ be a ball quasi-Banach function space. Assume that there
exists a $p_-\in(0,\infty)$ such that,
for any given $p\in(0,p_-)$ and $u\in(1,\infty)$, there exists a
positive constant $C$ such that,
for any $\{f_j\}_{j=1}^\infty\subset\mathscr M(\rn)$,
\begin{align*}
\lf\|\lf\{\sum_{j\in\nn}\lf[\mathcal{M}(f_j)\r]^u\r\}^{\frac{1}{u}}
\r\|_{X^{1/p}}
\le C\lf\|\lf(\sum_{j\in\nn}|f_j|^u\r)^{\frac{1}{u}}\r\|_{X^{1/p}}.
\end{align*}
\end{assumption}

To present the second assumption, we also need the concept of the power Hardy--Littlewood
maximal operator.
In what follows, for any $\theta\in(0,\infty)$, the \emph{powered Hardy--Littlewood
maximal operator} $\mathcal{M}^{(\theta)}$ is defined by setting,
for any $f\in L_{\loc}^1(\rn)$ and $x\in\rn$,
\begin{align*}
\mathcal{M}^{(\theta)}(f)(x):=\lf\{\mathcal{M}\lf(|f|^\theta\r)(x)\r\}^{\frac{1}{\theta}}.
\end{align*}

\begin{assumption}\label{assump2}
Let $X$ be a ball quasi-Banach function space.
Assume that there exists an $r_0\in(0,\infty)$ and a
$p_0\in(r_0,\infty)$
such that $X^{1/r_0}$ is a ball Banach function space and there exists a positive constant $C$ such that,
for any $f\in(X^{1/r_0})'$,
\begin{align*}
\lf\|\mathcal{M}^{((p_0/r_0)')}(f)\r\|_{(X^{1/r_0})'}\le
C\lf\|f\r\|_{(X^{1/r_0})'}.
\end{align*}
\end{assumption}

\begin{remark}\label{main-remark}
Let both $X$ and $p_-$ satisfy Assumption \ref{assump1}, and
let $d\in(0,\fz)$.
Notice that, for any ball $B\in\mathbb{B}(\rn)$ and any
$\beta\in[1,\fz)$,
$\mathbf{1}_{\beta B}\leq (\beta+1)^{\frac{dn}{r}}
[\mathcal{M}(\mathbf{1}_B)]^{\frac{d}{r}}$
with $r\in(0,\min\{d,p_-\})$. By this and Assumption \ref{assump1},
we easily conclude that, for any $r\in(0,\min\{d,p_-\})$, any
$\beta\in[1,\fz)$,
any sequence $\{B_j\}_{j\in\nn}\subset \mathbb{B}(\rn)$,
and any $\{\lambda_j\}_{j\in\nn}\subset [0,\fz)$,
\begin{align*}
\lf\|\lf(\sum_{j\in\nn}\lambda_j^d\mathbf{1}_{\beta
B_j}\r)^{\frac{1}{d}}\r\|_{X}
\leq(2\beta)^{\frac{n}{r}}\lf\|\lf(\sum_{j\in\nn}
\lambda_j^d\mathbf{1}_{B_j}\r)^{\frac{1}{d}}\r\|_{X}.
\end{align*}
Particularly, for any $r\in(0,p_-)$ and any ball $B\in
\mathbb{B}(\rn)$, we have
\begin{align}\label{key-c}
\lf\|\mathbf{1}_{\beta B}\r\|_{X}
\leq (2\beta)^{\frac{n}{r}}\|\mathbf{1}_{B}\|_{X}.
\end{align}
\end{remark}

\section{Fractional Integrals on
$\mathcal{L}_{X,q,s,d}(\rn)$}\label{sec-frac-B}

This section is divided into two subsections.
In Subsection \ref{sec-I},
we first give a reasonable version $\widetilde{I}_{\alpha}$
of the fractional integral $I_{\alpha}$
on $\mathcal{L}_{X,q,s,d}(\rn)$.
Then we give a sufficient and necessary condition of the boundedness of
$\widetilde{I}_{\alpha}$ from $\mathcal{L}_{X^{\beta},q,s,d}(\rn)$
to $\mathcal{L}_{X,q,s,d}(\rn)$.
In Subsection \ref{sec-I-2}, we prove that $\wz{I}_\alpha$ is
just the adjoint operator of $I_\alpha$, which further
strengthens the rationality of the definition of
$\widetilde{I}_{\alpha}$.

\subsection{Boundedness of Fractional Integrals on $\mathcal{L}_{X,q,s,d}(\rn)$\label{sec-I}}

In this subsection, we study the boundedness of fractional integrals on $\mathcal{L}_{X,q,s,d}(\rn)$.
Recall the concept of the following fractional integral $I_{\alpha}$ with $\alpha\in(0,n)$,
which was first introduced by
Hardy and Littlewood \cite{hl1928} (see also \cite[p.\,117]{EMS1970}).

\begin{definition}\label{Def-I}
Let $\alpha\in(0,n)$ and $p\in[1,\frac{n}{\alpha})$. The fractional integral
$I_{\alpha}$ is defined by setting, for any $f\in L^{p}(\rn)$ and almost every $x\in\rn$,
\begin{align}\label{cla-I}
I_{\alpha}(f)(x)
:=\int_{\rn}\frac{f(y)}{|x-y|^{n-\alpha}}\,dy.
\end{align}
\end{definition}

We now recall the definition of $\wz I_{\alpha}$
which was first introduced in \cite{jtyyz2} to establish the boundedness
of fractional integrals
on special John--Nirenberg--Campanato spaces in \cite{jtyyz1}.
In what follows, for any given $s\in\mathbb{Z}_+$ and
$\lambda\in(s,\infty)$, we always need to consider a function
$f\in L^1_{\loc}(\rn)$ satisfying that,
for any ball $B(x,r)\in\mathbb{B}(\rn)$ with $x\in\rn$ and $r\in(0,\fz)$,
\begin{align}\label{suita}
\int_{\rn\setminus B(x,r)}\frac{|f(y)
-P_{B(x,r)}^{(s)}(f)(y)|}{|x-y|^{n+\lambda}}\,dy<\fz.
\end{align}

\begin{definition}\label{I-w}
Let $s\in\zz_+$, $\alpha\in(0,1)$,
and $B_0:=B(x_0,r_0)\in\mathbb{B}(\rn)$
with $x_0\in\rn$ and $r_0\in (0,\fz)$.
The \emph{modified fractional integral
$\widetilde{I}_{\alpha,B_0}$}
is defined by setting, for any $f\in L^1_{\mathrm{loc}}(\rn)$
satisfying \eqref{suita} with $\lambda:=s+1-\alpha$, and for almost every
$x\in \rn$,
\begin{align}\label{D-Iw}
&\widetilde{I}_{\alpha,B_0}(f)(x)\\
&\quad:=\int_{\rn}\lf[\frac{1}{|x-y|^{n-\alpha}}
-\sum_{\{\gamma\in\zz_+^n:\ |\gamma|\leq s\}}
\frac{\partial_{x}^{\gamma}(\frac{1}{|x-y|^{n-\alpha}})|_{x=x_0}}
{\gamma!}(x-x_0)^{\gamma}
\mathbf{1}_{\rn\setminus B_0}(y)\r]f(y)\,dy.\noz
\end{align}
\end{definition}

\begin{remark}
When $s:=0$, Definition \ref{I-w} coincides with the
$\wz I_\alpha$ studied by Nakai
\cite{N01} (see also \cite{gv} for a variant on spaces of homogeneous type).
\end{remark}

To show that $\widetilde{I}_{\alpha,B_0}(f)$ in \eqref{D-Iw}
is well defined almost everywhere on $\rn$, we need the following
technical lemma which is just a corollary of \cite[Remark 2.5 and
Proposition 2.16]{jtyyz2}; we omit the details here.

\begin{lemma}\label{I-a-x}
Let $s\in\zz_+$, $\alpha\in(0,1)$,
$B_0:=B(x_0,r_0)\in\mathbb{B}(\rn)$ with $x_0\in\rn$
and $r_0\in(0,\fz)$, and
$\widetilde{I}_{\alpha,B_0}$ be the same
as in \eqref{D-Iw}. Then, for any $\gamma\in\zz_+^n$ with
$|\gamma|\leq s$,
$\widetilde{I}_{\alpha,B_0}(y^{\gamma})\in \mathcal{P}_s(\rn)$
after changing values on a set of measure zero.
\end{lemma}

We also need the following well-known Hardy--Littlewood--Sobolev theorem
which was first established by Hardy and Littlewood \cite{hl1928}
and Sobolev \cite{s1938} (see also \cite[p.\,119]{EMS1970}).

\begin{lemma}\label{fractional}
Let $\alpha\in(0,n)$ and $I_{\alpha}$ be the same as in \eqref{cla-I}.
Let $p\in[1,\frac{n}{\alpha})$ and $q\in(1,\infty)$ with
$\frac{1}{q}:=\frac{1}{p}-\frac{\alpha}{n}$.
\begin{enumerate}
\item[\rm (i)]
For any $f\in L^p(\rn)$, $I_\alpha(f)(x)$ is well defined for
almost every $x\in\rn$.

\item[\rm (ii)]
If $p\in(1,\frac{n}{\alpha})$, then $I_\alpha$ is bounded from
$L^p(\rn)$
to $L^{q}(\rn)$, namely, there exists a
positive constant $C$ such that, for any $f\in L^p(\rn)$,
$$\lf\|I_\alpha(f)\r\|_{L^{q}(\rn)}\leq
C\|f\|_{L^p(\rn)}.$$
\end{enumerate}
\end{lemma}

Now, we show that $\widetilde{I}_{\alpha,B_0}(f)$ in \eqref{D-Iw} is
well defined; since its proof is quite similar to that of
\cite[Proposition 2.19]{jtyyz2}, we only sketch some important steps.

\begin{proposition}\label{I-w-h}
Let all the symbols be the same as in Definition \ref{I-w}.
Then, for any $f\in L^1_{\mathrm{loc}}(\rn)$
satisfying \eqref{suita} with $\lambda:=s+1-\alpha$,
$\widetilde{I}_{\alpha,B_0}(f)$ in \eqref{D-Iw}
is well defined almost everywhere on $\rn$.
\end{proposition}

\begin{proof}
Let all the symbols be the same as in the present proposition
and let $B:=B(z,r)\in\mathbb{B}(\rn)$ with $z\in\rn$ and $r\in(0,\fz)$.
By the definition of $\widetilde{I}_{\alpha,B_0}$, we write that,
for any $f\in L^1_{\mathrm{loc}}(\rn)$
satisfying \eqref{suita} with $\lambda:=s+1-\alpha$,
and for almost every $x\in B$,
\begin{align}\label{I-1}
\widetilde{I}_{\alpha,B_0}(f)(x)
=F_{B}^{(1)}(x)+F_{B}^{(2)}(x)
+P_{B,B_0}^{(1)}(x)+P_{B,B_0}^{(2)}(x),
\end{align}
where
\begin{align}\label{2.22x}
F_{B}^{(1)}(x):=\int_{2B}\frac{f(y)-P_{2B}^{(s)}(f)(y)}
{|x-y|^{n-\alpha}}\,dy,
\end{align}

\begin{align}\label{2.22y}
F_{B}^{(2)}(x):&=\int_{\rn\setminus 2B}
\lf[\frac{1}{|x-y|^{n-\alpha}}
-\sum_{\{\gamma\in\zz_+^n:\ |\gamma|\leq s\}}
\frac{\partial_{x}^{\gamma}(\frac{1}{|x-y|^{n-\alpha}})|_{x=z}}
{\gamma!}(x-z)^{\gamma}\r]\\
&\quad\times\lf[f(y)-P_{2B}^{(s)}(f)(y)\r]\,dy,\noz
\end{align}

\begin{align}\label{2.22z}
P_{B,B_0}^{(1)}(x):=\int_{\rn}
K_{B,B_0}(x,y)\lf[f(y)-P_{2B}^{(s)}(f)(y)\r]\,dy
\end{align}
with
\begin{align*}
K_{B,B_0}(x,y)
:&=\sum_{\{\gamma\in\zz_+^n:\ |\gamma|\leq s\}}
\frac{\partial_{x}^{\gamma}(\frac{1}{|x-y|^{n-\alpha}})|_{x=z}}
{\gamma!}(x-z)^{\gamma}\mathbf{1}_{\rn\setminus 2B}(y)\\
&\quad-\sum_{\{\gamma\in\zz_+^n:\ |\gamma|\leq s\}}
\frac{\partial_{x}^{\gamma}(\frac{1}{|x-y|^{n-\alpha}})|
_{x=x_0}}{\gamma!}(x-x_0)^{\gamma}
\mathbf1_{\rn\setminus B_0}(y),
\end{align*}
and
\begin{align}\label{2.22w}
P_{B,B_0}^{(2)}(x):&=\int_{\rn}
\lf[\frac{1}{|x-y|^{n-\alpha}}
-\sum_{\{\gamma\in\zz_+^n:\ |\gamma|\leq s\}}
\frac{\partial_{x}^{\gamma}(\frac{1}{|x-y|^{n-\alpha}})|
_{x=x_0}}{\gamma!}(x-x_0)^{\gamma}
\mathbf1_{\rn\setminus B_0}(y)\r]\\
&\quad\times P_{2B}^{(s)}(f)(y)\,dy.\noz
\end{align}
Thus, to prove that $\widetilde{I}_{\alpha,B_0}(f)$ is well defined
almost everywhere on $B$, it suffices to show that
$F_{B}^{(1)}$, $F_{B}^{(2)}$, $P_{B,B_0}^{(1)}$,
and $P_{B,B_0}^{(2)}$ are all well defined
almost everywhere on $B$.

We first consider $P_B^{(1)}$.
Indeed, it is easy to show that
$[f-P_{2B}^{(s)}(f)]\mathbf1_{2B}\in L^1(\rn)$.
From this and Lemma \ref{fractional}(i), we deduce that
$F_{B}^{(1)}(x)$ is well defined
for almost every $x\in B$. In addition, by \cite[(2.24)]{jtyyz2}
and \eqref{suita} with $\lambda:=s+1-\alpha$, we find that, for any $x\in B$,
\begin{align}\label{I-E-1}
\lf|F_{B}^{(2)}(x)\r|
\lesssim r^{s+1}\int_{\rn\setminus 2B}
\frac{|f(y)-P_{2B}^{(s)}(f)(y)|}{|y-z|^{n+s+1-\alpha}}\,dy<\fz.
\end{align}
Moreover, from \eqref{suita} and an argument similar to that used in the estimation of
\cite[(2.26)]{jtyyz2}, we infer that
$P_{B,B_0}^{(1)}(x)$ for any $x\in B$ is well defined and
\begin{align}\label{2.27x}
P_{B,B_0}^{(1)}\in \mathcal{P}_{s}(B).
\end{align}
Finally, Lemma \ref{I-a-x} implies that $P_{B,B_0}^{(2)}(x)$ for almost every $x\in B$
is well defined and
\begin{align}\label{2.27xx}
P_{B,B_0}^{(2)}\in \mathcal{P}_{s}(B)
\end{align}
after changing values on a set of measure zero.
Then we conclude that $\widetilde{I}_{\alpha,B_0}(f)$ is well defined almost
everywhere on $B$.
This, together with the arbitrariness
of $B$, then finishes the proof of Proposition \ref{I-w-h}.
\end{proof}

\begin{remark}\label{rem-I-B}
Let all the symbols be the same as in Definition \ref{I-w}.
\begin{enumerate}
\item[\rm (i)]
By \cite[Remark 2.20]{jtyyz2}, we conclude that, for any given ball $B_1\in\mathbb{B}(\rn)$ and
for any $f$ satisfying \eqref{suita} with $\lambda:=s+1-\alpha$,
$$
\widetilde{I}_{\alpha,B_0}(f)-\widetilde{I}_{\alpha,B_1}(f)\in
\mathcal{P}_s(\rn)$$
after changing values on a set of measure zero.
Based on this, in what follows, we write
$\widetilde{I}_{\alpha}$
instead of $\widetilde{I}_{\alpha,B_0}$ if there exists no confusion.
Moreover, for any given $\alpha\in[1,n)$ and for any
suitable measurable function $g$ on $\rn$, we define
\begin{align}\label{de-wsIn}
\wz I_\alpha(g):=\lf(\wz I_{\frac{\alpha}{n}}\r)^n(g)
:=\lf(\wz I_{\frac{\alpha}{n}}\circ\cdots\circ
\wz I_{\frac{\alpha}{n}}\r)(g).
\end{align}

\item[\rm (ii)]
Let $q\in[1,\frac{n}{\alpha})$ and $f\in L^q(\rn)$. By
\cite[Proposition 2.12]{jtyyz2}, we know that
$\wz I_\alpha(f)-I_{\alpha}(f)\in \mathcal{P}_s(\rn)$
after changing values on a set of measure zero.
\end{enumerate}
\end{remark}

Next, we present the main result of this subsection.

\begin{theorem}\label{thm-main-I}
Let both $X$ and $p_-$ satisfy Assumption \ref{assump1}.
Assume that $X$, $r_0\in(0,\min\{\frac{1}{\beta},p_-\})$
with $\beta\in(1,\fz)$, and $p_0\in(r_0,\infty)$
satisfy Assumption \ref{assump2}.
Let $\alpha\in(0,1)$, $d:=r_0$, integer
\begin{align*}
s\geq\max\lf\{\lf\lfloor\frac{n}
{\min\{1,p_-\}}-n\r\rfloor,
\lf\lfloor\frac{n}{\min\{1,\beta p_-\}}-1\r\rfloor,
\lf\lfloor\frac{n}{\min\{1,d\}}-n+\alpha\r\rfloor\r\},
\end{align*}
$I_\alpha$ be the same
as in \eqref{cla-I}, and $\wz{I}_\alpha$ the same as in
Remark \ref{rem-I-B}(i).
Further assume that $X$ has an absolutely continuous quasi-norm.
Then $\widetilde{I}_{\alpha}$ is bounded from
$\mathcal{L}_{X^{\beta},1,s,{\beta d}}(\rn)$
to $\mathcal{L}_{X,1,s,d}(\rn)$, namely,
there exists a positive constant $C$ such that, for any
$f\in \mathcal{L}_{X^{\beta},1,s,{\beta d}}(\rn)$,
\begin{align*}
\lf\|\widetilde{I}_{\alpha}(f)
\r\|_{\mathcal{L}_{X,1,s,d}(\rn)}
\leq C \|f\|_{\mathcal{L}_{X^{\beta},1,s,{\beta d}}(\rn)}
\end{align*}
if and only if there exists a positive constant $\wz C$
such that, for any ball $B\in\mathbb{B}(\rn)$,
\begin{align}\label{3.12x}
|B|^{\frac{\alpha}{n}}\leq \wz C
\|\mathbf{1}_B\|_X^{\frac{\beta-1}{\beta}}.
\end{align}
\end{theorem}

\begin{remark}\label{sgjs}
\begin{enumerate}
\item[\rm (i)]
Let both $X$ and $p_-$ satisfy Assumption \ref{assump1}. Assume that
$r_0\in(0,\min\{1,p_-\})$ and $p_0\in(r_0,\infty)$
satisfy Assumption \ref{assump2}.
Let $s\in[\lfloor n(\frac{1}{\min\{1,p_-\}}-1)\rfloor,\infty)\cap\zz_+$,
$d\in(0,r_0]$, $q\in(\max\{1,p_0\},\infty]$,
and $X$ have an absolutely continuous quasi-norm.
In \cite[Corollary 3.15]{zhyy2022}, Zhang et al. proved that
$\mathcal{L}_{X,1,s_0,d_0}(\rn)=\mathcal{L}_{X,q',s,d}(\rn)$
with equivalent quasi-norms,
where $s_0:=\max\{0,\lfloor n(\frac{1}{{\min\{1,p_-\}}}-1)\rfloor\}$
and $d_0:=r_0$.
\item[\rm (ii)]
Let all the symbols be the same as in Theorem \ref{thm-main-I}.
Using (i) of the present remark, we conclude that
$\mathcal{L}_{X,1,s,d}(\rn)=\mathcal{L}_{X,q',s, d}(\rn)$ and
$\mathcal{L}_{X^{\beta},1,s,{\beta d}}(\rn)=\mathcal{L}_{X^{\beta},q',s,{\beta d}}(\rn)$
for any $q\in(\max\{1,\beta p_0\},\infty]$, and hence
Theorem \ref{thm-main-I} still holds true
with both $\mathcal{L}_{X,1,s,d}(\rn)$ and $\mathcal{L}_{X^{\beta},1,s,{\beta d}}(\rn)$ replaced,
respectively, by $\mathcal{L}_{X,q',s,d}(\rn)$
and $\mathcal{L}_{X^{\beta},q',s,{\beta d}}(\rn)$.

\item[\rm (iii)]
Let all the symbols be the same as in Theorem \ref{thm-main-I},
and $I_{\alpha}$ the same as in \eqref{cla-I}.
In \cite[Theorem 3.7]{cjy-01}, we proved that $I_{\alpha}$ is bounded from
$H_{X}(\rn)$ to $H_{X^{\beta}}(\rn)$ if and only if \eqref{3.12x} holds true.
In addition, using \cite[Theorem 3.14]{zhyy2022}, we find both
$(H_{X}(\rn))^*=\mathcal{L}_{X,1,s,d}(\rn)$ and
$(H_{X^{\beta}}(\rn))^*=\mathcal{L}_{X^{\beta},1,s,\beta d}(\rn)$.
Thus, by duality, the formal desired corresponding
result of \cite[Theorem 3.7]{cjy-01}
on the boundedness of fractional integrals on ball
Campanato-type function spaces should be
that $\wz I_{\alpha}$ is bounded from $\mathcal{L}_{X^{\beta},1,s,\beta d}(\rn)$
to $\mathcal{L}_{X,1,s,d}(\rn)$ [rather than $\wz I_{\alpha}$ is bounded
from $\mathcal{L}_{X,1,s, d}(\rn)$
to $\mathcal{L}_{X^{1/\beta},1,s,d/\beta}(\rn)$]
if and only if \eqref{3.12x} holds true, which is just
Theorem \ref{thm-main-I}. This explains the rationality of
the symbols used in Theorem \ref{thm-main-I}.
\end{enumerate}
\end{remark}

We prove the sufficiency and the necessity of Theorem \ref{thm-main-I}, respectively.
We first consider the sufficiency;
indeed, we have the following more general conclusion.

\begin{theorem}\label{main-Theorem2}
Let both $X$ and $p_-$ satisfy Assumption \ref{assump1},
$\beta\in(1,\fz)$, $s\in\zz_+$, $d\in(0,\fz)$,
$\alpha\in (0,1)$, and $\widetilde{I}_{\alpha}$ be the same as
in Remark \ref{rem-I-B}(i).
Let $q_1=1$ and $q_2\in[1,\frac{n}{n-\alpha})$,
or $q_1\in(1,\frac{n}{\alpha})$
and $q_2\in[1,\frac{nq_1}{n-\alpha q_1}]$,
or $q_1\in[\frac{n}{\alpha},\fz)$
and $q_2\in [1,\fz)$.
Further assume that there exists a positive constant $\wz C$
such that, for any ball $B\in\mathbb{B}(\rn)$,
\begin{align}\label{assume000}
|B|^{\frac{\alpha}{n}}\leq \wz C
\|\mathbf{1}_B\|_X^{\frac{\beta-1}{\beta}}.
\end{align}
\begin{enumerate}
\item[\rm(i)]	
If $s\in(\frac{n}{\beta p_-}-n-1+\alpha,\fz)\cap\zz_+$, then
$\widetilde{I}_{\alpha}$
is well defined on $\mathcal{L}_{X^{\beta},q_1,s}(\rn)$ and
bounded from $\mathcal{L}_{X^{\beta},q_1,s}(\rn)$
to $\mathcal{L}_{X,q_2,s}(\rn)$, namely,
there exists a positive constant $C$ such that,
for any $f\in \mathcal{L}_{X^{\beta},q_1,s}(\rn)$,
\begin{align*}
\lf\|\widetilde{I}_{\alpha}(f)\r\|_{\mathcal{L}_{X,q_2,s}(\rn)}
\leq C\|f\|_{\mathcal{L}_{X^{\beta},q_1,s}(\rn)};
\end{align*}

\item[\rm(ii)]
if $s\in(\frac{n}{\beta\min\{d,p_-\}}-n-1+\alpha,\fz)\cap\zz_+$,
then $\widetilde{I}_{\alpha}$
is well defined on $\mathcal{L}_{X^{\beta},q_1,s,\beta d}(\rn)$ and
bounded from
$\mathcal{L}_{X^{\beta},q_1,s,{\beta d}}(\rn)$
to $\mathcal{L}_{X,q_2,s,d}(\rn)$, namely,
there exists a positive constant $C$ such that, for any
$f\in \mathcal{L}_{X^{\beta},q_1,s,{\beta d}}(\rn)$,
\begin{align}\label{I-01}
\lf\|\widetilde{I}_{\alpha}(f)
\r\|_{\mathcal{L}_{X,q_2,s,d}(\rn)}
\leq C \|f\|_{\mathcal{L}_{X^{\beta},q_1,s,{\beta d}}(\rn)}.
\end{align}
\end{enumerate}
\end{theorem}

To show this theorem, we
first establish two technical lemmas. Similarly to the proof of
\cite[Lemma 2.21]{jtyyz3}, we have the following conclusion; we omit the details here.

\begin{lemma}\label{I-JN}
Let $q\in[1,\fz)$, $s\in\zz_+$, and $\lambda\in(s,\fz)$.
Then there exists a positive constant $C$ such that,
for any $f\in L^1_{\loc}(\rn)$ and
any ball $B(x,r)$ with $x\in\rn$ and $r\in(0,\fz)$,
\begin{align}\label{JN-I-E}
&\int_{\rn\setminus B(x,r)}\frac{|f(y)
-P_{B(x,r)}^{(s)}(f)(y)|}{|x-y|^{n+\lz}}\,dy\\
&\quad\leq C\sum_{k\in\nn}\lf(2^kr\r)^{-\lambda}
\lf[\fint_{2^{k}B(x,r)}\lf|f(y)
-P_{2^{k}B(x,r)}^{(s)}(f)(y)\r|^q\,dy\r]^{\frac{1}{q}}.\noz
\end{align}
\end{lemma}

The following lemma shows that
functions of both $\mathcal{L}_{X,q,s}(\rn)$
and $\mathcal{L}_{X,q,s,d}(\rn)$ satisfy \eqref{suita}.

\begin{lemma}\label{lem-suit}
Let $s\in\zz_+$, $q\in[1,\fz)$, $d\in(0,\fz)$, and
$\lambda\in(s,\fz)$.
Let both $X$ and $p_-\in(\frac{n}{n+\lambda},\fz)$ satisfy Assumption
\ref{assump1}.
Then, for any $f\in \mathcal{L}_{X,q,s}(\rn)$ [resp., $f\in
\mathcal{L}_{X,q,s,d}(\rn)$],
$f$ satisfies \eqref{suita}.
\end{lemma}

\begin{proof}
Let all the symbols be the same as in the present lemma.
We only need to show the case $f\in \mathcal{L}_{X,q,s}(\rn)$ due to
\eqref{sub-01}.
In this case, by \eqref{JN-I-E}, Remark \ref{main-remark},
$p_-\in(\frac{n}{n+\lambda},\fz)$, and the definition of
$\|\cdot\|_{\mathcal{L}_{X,q,s}(\rn)}$,
we find that,
for any given $r\in(\frac{n}{n+\lambda},p_-)$ and
for any ball $B:=B(x_0,r_0)\in\mathbb{B}(\rn)$
with $x_0\in\rn$ and $r_0\in(0,\fz)$,
\begin{align*}
&\int_{\rn\setminus B}\frac{|f(y)
-P_{B}^{(s)}(f)(y)|}{|x_0-y|^{n+\lz}}\,dy\\
&\quad\ls\sum_{k\in\nn}\lf(2^kr_0\r)^{-\lambda}
\lf[\fint_{2^{k}B}\lf|f(y)
-P_{2^{k}B}^{(s)}(f)(y)\r|^q\,dy\r]^{\frac{1}{q}}\\
&\quad\sim\sum_{k\in\nn}\lf(2^kr_0\r)^{-\lambda}
\frac{\|\mathbf{1}_{2^{k+1}B}\|_{X}}{|2^{k+1}B|}
\frac{|2^{k+1}B|}{\|\mathbf{1}_{2^{k+1}B}\|_{X}}
\lf[\fint_{2^{k}B}\lf|f(y)
-P_{2^{k}B}^{(s)}(f)(y)\r|^q\,dy\r]^{\frac{1}{q}}\\
&\quad\ls\sum_{k\in\nn}2^{k(-\lambda-n+\frac{n}{r})}
\frac{r_0^{-\lambda}\|\mathbf{1}_{B}\|_{X}}{|B|}
\|f\|_{\mathcal{L}_{X,q,s}(\rn)}
\ls\|f\|_{\mathcal{L}_{X,q,s}(\rn)}<\fz.
\end{align*}
Thus, $f$ satisfies \eqref{suita}.
This finishes the proof of the case $f\in \mathcal{L}_{X,q,s}(\rn)$,
and hence of Lemma \ref{lem-suit}.
\end{proof}

Now, we prove Theorem \ref{main-Theorem2}.

\begin{proof}[Proof of Theorem \ref{main-Theorem2}]
We only prove (ii) because the proof of (i) is similar.	
Let all the symbols be the same as in the present theorem.
Notice that $\beta p_-\in(\frac{n}{n+s+1-\alpha},\fz)$.
From this, Proposition \ref{I-w-h}, and Lemma \ref{lem-suit},
we deduce that $\widetilde{I}_{\alpha}$
is well defined on $\mathcal{L}_{X^{\beta},q_1,s,{\beta d}}(\rn)$.

Now, we show \eqref{I-01}.
To this end, let $\{B_j\}_{j=1}^m:=\{B(z_j,r_j)\}_{j=1}^m\subset
\mathbb{B}(\rn)$
with both $\{z_j\}_{j=1}^m\subset\rn$ and $\{r_j\}_{j=1}^m\subset(0,\fz)$, and
$\{\lambda_j\}_{j=1}^m\subset[0,\infty)$
with $\sum_{j=1}^m\lambda_j\neq0$.
Using Remark \ref{rem-ball-B2}, we easily find that,
for any $g\in L^1_{\mathrm{loc}}(\rn)$ and any ball
$B\in\mathbb{B}(\rn)$,
\begin{align}\label{int-P}
\lf[\fint_B\lf|g(x)-P_B^{(s)}(g)(x)
\r|^q\,dx\r]^{\frac{1}{q}}
\sim\inf_{P\in \mathcal{P}_s(B)}
\lf[\fint_B|g(x)-P(x)|^q\,dx\r]^{\frac{1}{q}};
\end{align}
see, for instance, \cite[(2.12)]{jtyyz1}.
From this, \eqref{I-1}, \eqref{2.27x}, \eqref{2.27xx},
and the fact that $P_{B}^{(s)}(P)=P$ for any
$P\in \mathcal{P}_s(\rn)$ and $B\in \mathbb{B}(\rn)$, we deduce that,
for any
$f\in \mathcal{L}_{X^{\beta},q_1,s,d}(\rn)$,
\begin{align}\label{Iw-B-01}
&\lf\|\lf\{\sum_{i=1}^m
\lf[\frac{{\lambda}_i}{\|{\mathbf{1}}_{B_i}\|_X}\r]^d
{\mathbf{1}}_{B_i}\r\}^{\frac1d}\r\|_{X}^{-1}
\sum_{j=1}^m\frac{{\lambda}_j|B_j|}{\|{\mathbf{1}}_{B_j}\|_{X}}
\lf[\fint_{B_j}\lf|\wz I_{\alpha}(f)(x)
-P^{(s)}_{B_j}\lf(\wz I_{\alpha}(f)\r)(x)\r|^q \,dx\r]^\frac1q\\
&\quad\lesssim\lf\|\lf\{\sum_{i=1}^m
\lf[\frac{{\lambda}_i}{\|{\mathbf{1}}_{B_i}\|_X}\r]^d
{\mathbf{1}}_{B_i}\r\}^{\frac1d}\r\|_{X}^{-1}
\sum_{j=1}^m\frac{\lambda_{j}|B_{j}|}{\|\mathbf{1}_{B_{j}}\|_{X}}\noz\\
&\quad\quad\times
\lf[\fint_{B_j}
\lf|\widetilde{I}_\alpha (f)(x)-P_{B_j,B_0}^{(1)}(x)
-P_{B_j,B_0}^{(2)}(x)\r|^{q_2}\,dx\r]^{\frac{1}{q_2}}\noz\\
&\quad\sim\lf\|\lf\{\sum_{i=1}^m
\lf[\frac{{\lambda}_i}{\|{\mathbf{1}}_{B_i}\|_X}\r]^d
{\mathbf{1}}_{B_i}\r\}^{\frac1d}\r\|_{X}^{-1}
\sum_{j=1}^m\frac{\lambda_{j}|B_{j}|}{\|\mathbf{1}_{B_{j}}\|_{X}}
\lf[\fint_{B_j}
\lf|F_{B_j}^{(1)}(x)+F_{B_j}^{(2)}(x)\r|^{q_2}\,dx
\r]^{\frac{1}{q_2}}\noz\\
&\quad\ls\lf\|\lf\{\sum_{i=1}^m
\lf[\frac{{\lambda}_i}{\|{\mathbf{1}}_{B_i}\|_X}\r]^d
{\mathbf{1}}_{B_i}\r\}^{\frac1d}\r\|_{X}^{-1}
\lf\{\sum_{j=1}^m\frac{\lambda_j|B_j|}{\|\mathbf{1}_{B_j}\|_{X}}
\lf[\fint_{B_j}
\lf|F_{B_j}^{(1)}(x)\r|^{q_2}\,dx\r]^{\frac{1}{q_2}}\r.\noz\\
&\qquad
\lf.+\sum_{j=1}^m\frac{\lambda_j|B_j|}{\|\mathbf{1}_{B_j}\|_{X}}
\lf[\fint_{B_j}
\lf|F_{B_j}^{(2)}(x)\r|^{q_2}\,dx\r]^{\frac{1}{q_2}}\r\}\noz\\
&\quad=:\mathrm{I}_1+\mathrm{I}_2,\noz
\end{align}
where $F_{B_j}^{(1)}$, $F_{B_j}^{(2)}$, $P_{B_j,B_0}^{(1)}$,
and $P_{B_j,B_0}^{(2)}$ are, respectively, as in \eqref{2.22x}
\eqref{2.22y}, \eqref{2.22z}, and \eqref{2.22w}
with $B$ therein replaced by $B_j$, and
the implicit positive constants are independent of $f$, $m$,
$\{\lambda_j\}_{j=1}^m$, and $\{B_j\}_{j=1}^m$.

Next, we prove that
\begin{align}\label{I-ball-01'}
\mathrm{I}_1\ls\|f\|_{\mathcal{L}_{X^{\beta},q_1,s,{\beta d}}(\rn)}
\end{align}
with the implicit positive constants independent of $f$, $m$,
$\{\lambda_j\}_{j=1}^m$, and $\{B_j\}_{j=1}^m$.
Indeed, by both the definition of $\|\cdot\|_{X^{\beta}}$
and $\beta\in(1,\fz)$, we find that
\begin{align}\label{I-B-01}
&\lf\|\lf\{\sum_{i=1}^m
\lf[\frac{\lambda_{i}}{\|\mathbf{1}_{B_i}\|_{X}}\r]^{d}
\mathbf{1}_{B_i}\r\}^{\frac{1}{d}}\r\|_{X}^{-1}\\
&\quad=\lf\|\lf\{\sum_{i=1}^m\lf[\frac{\lambda_{i}^{\frac{1}{\beta}}}
{\|\mathbf{1}_{B_i}\|_{X^{\beta}}}\r]^{\beta d}
\mathbf{1}_{B_i}\r\}^{\frac{1}{\beta
d}}\r\|_{X^{\beta}}^{-\beta}\noz\\
&\quad=\lf\|\lf\{\sum_{i=1}^m\lf[\frac{\lambda_{i}^{\frac{1}{\beta}}}
{\|\mathbf{1}_{B_i}\|_{X^{\beta}}}\r]^{\beta d}
\mathbf{1}_{B_i}\r\}^{\frac{1}{\beta d}}\r\|_{X^{\beta}}^{-1}
\lf\|\lf\{\sum_{i=1}^m\lf[\frac{\lambda_{i}^{\frac{1}{\beta}}}
{\|\mathbf{1}_{B_i}\|_{X^{\beta}}}\r]^{\beta d}
\mathbf{1}_{B_i}\r\}^{\frac{1}{\beta
d}}\r\|_{X^{\beta}}^{1-\beta}\noz\\
&\quad\leq\lf\|\lf\{\sum_{i=1}^m\lf[\frac{\lambda_{i}^{\frac{1}{\beta}}}
{\|\mathbf{1}_{B_i}\|_{X^{\beta}}}\r]^{\beta d}
\mathbf{1}_{B_i}\r\}^{\frac{1}{\beta d}}\r\|_{X^{\beta}}^{-1}
\inf\lf\{\lambda_i^{\frac{1-\beta}{\beta}}:\ i\in\{1,\ldots,m\}\r\}\noz\\
&\quad\leq\lf\|\lf\{\sum_{i=1}^m\lf[\frac{\lambda_{i}}
{\|\mathbf{1}_{B_i}\|_{X^{\beta}}}\r]^{\beta d}
\mathbf{1}_{B_i}\r\}^{\frac{1}{\beta d}}
\r\|_{X^{\beta}}^{-1}.\noz
\end{align}
In addition, using \cite[(2.29)]{jtyyz2}, we obtain, for any
$j\in\{1,\ldots,m\}$,
\begin{align}\label{A1-1}
&\lf|B_j\r|^{-\frac{\alpha}{n}}
\lf[\fint_{B_j}\lf|F_{B_j}^{(1)}(x)
\r|^{q_2}\,dx\r]^{\frac{1}{q_2}}
\lesssim\lf[\fint_{2B_j}\lf|f(x)
-P_{2B_j}^{(s)}(f)(x)\r|^{q_1}\,dx\r]^{\frac{1}{q_1}}.
\end{align}
Moreover, from the definition of $\|\cdot\|_{X^{\beta}}$,
\eqref{assume000}, and $\beta\in(1,\fz)$, we deduce that, for any
$j\in\{1,\ldots,m\}$,
\begin{align}\label{cjy4.18x}
\lf\|\mathbf{1}_{B_j}\r\|_{X^{\beta}}
=\lf\|\mathbf{1}_{B_j}\r\|_{X}^{\frac{1}{\beta}}
=\lf\|\mathbf{1}_{B_j}\r\|_{X}^{\frac{1-\beta}{\beta}}
\lf\|\mathbf{1}_{B_j}\r\|_{X}
\ls |B_j|^{-\frac{\alpha}{n}}\lf\|\mathbf{1}_{B_j}\r\|_{X},
\end{align}
which, together with \eqref{key-c}, further implies that
\begin{align}\label{I-ball-03}
\frac{|B_j|}{\|\mathbf{1}_{B_j}\|_{X}}
\ls\frac{|B_j|^{1-\frac{\alpha}{n}}}
{\|\mathbf{1}_{B_j}\|_{X^{\beta}}}
\sim \frac{|2B_j|^{1-\frac{\alpha}{n}}}
{\|\mathbf{1}_{2B_j}\|_{X^{\beta}}}.
\end{align}
Using \eqref{I-B-01}, \eqref{A1-1}, \eqref{I-ball-03},
Remark \ref{main-remark}, and the definition of
$\|\cdot\|_{\mathcal{L}_{X^{\beta},q_1,s,{\beta d}}(\rn)}$, we
conclude that
\begin{align*}
\mathrm{I}_1
&\ls\lf\|\lf\{\sum_{i=1}^m\lf[\frac{\lambda_{i}}
{\|\mathbf{1}_{B_i}\|_{X^{\beta}}}\r]^{{\beta d}}
\mathbf{1}_{B_i}\r\}^{\frac{1}{\beta d}}\r\|_{X^{\beta}}^{-1}
\sum_{j=1}^m\frac{\lambda_{j}|B_j|}
{\|\mathbf{1}_{B_j}\|_{X}}\lf[\fint_{B_j}
\lf|F_{B_j}^{(1)}(x)\r|^{q_2}\,dx\r]^{\frac{1}{q_2}}\\
&\ls\lf\|\lf\{\sum_{i=1}^m\lf[\frac{\lambda_{i}}
{\|\mathbf{1}_{2B_i}\|_{X^{\beta}}}\r]^{\beta d}
\mathbf{1}_{2B_i}\r\}^{\frac{1}{\beta d}}\r\|_{X^{\beta}}^{-1}
\sum_{j=1}^m\frac{\lambda_j|2B_j|}{\|\mathbf{1}_{2B_j}\|_{X^{\beta}}}
\lf[\fint_{2B_j}\lf|f(x)-P_{2B_j}^{(s)}(f)(x)\r|^{q_1}\,dx\r]
^{\frac{1}{q_1}}\\
&\ls\|f\|_{\mathcal{L}_{X^{\beta},q_1,s,\beta d}(\rn)}
\end{align*}
and hence \eqref{I-B-01} holds true.

Now, we show that
\begin{align*}
\mathrm{I}_2\ls\|f\|_{\mathcal{L}_{X^{\beta},q_1,s,\beta d}(\rn)}
\end{align*}
with the implicit positive constants independent of $f$, $m$,
$\{\lambda_j\}_{j=1}^m$, and $\{B_j\}_{j=1}^m$.
To this end, for any $j\in\{1,\ldots,m\}$ and $k\in\nn$, let
$\lambda_{j,k}:=\frac{\lambda_j\|\mathbf{1}_{2^kB_j}
\|_{X^{\beta}}}{\|\mathbf{1}_{B_j}\|_{X^{\beta}}}$.
Then, by the definition of $\|\cdot\|_{X^{\beta}}$,
$\beta\in(1,\fz)$, Remark \ref{main-remark}, and the definition of
$\lambda_{i,k}$
with both $i\in\{1,\ldots,m\}$ and $k\in\nn$, we find that, for any given
$r\in(0,\min\{d,p_-\})$ and for any $k\in\nn$,
\begin{align*}
&\lf\|\lf\{\sum_{i=1}^m\lf[\frac{\lambda_{i,k}}
{\|\mathbf{1}_{2^kB_i}\|_{X^{\beta}}}\r]^{\beta d}
\mathbf{1}_{2^kB_i}\r\}^{\frac{1}{\beta d}}\r\|_{X^{\beta}}\\
&\quad=\lf\|\lf\{\sum_{i=1}^m\lf[\frac{\lambda_{i,k}^{\beta}}
{\|\mathbf{1}_{2^kB_i}\|_{X}}\r]^{d}
\mathbf{1}_{2^kB_i}\r\}^{\frac{1}{d}}\r\|_{X}^{\frac{1}{\beta}}
\ls2^{\frac{kn}{r\beta}}\lf\|\lf\{\sum_{i=1}^m
\lf[\frac{\lambda_{i,k}^{\beta}}{\|\mathbf{1}_{2^kB_i}\|_{X}}\r]^d
\mathbf{1}_{B_i}\r\}^{\frac{1}{d}}\r\|_{X}^\frac{1}{\beta}\\
&\quad\sim2^{\frac{kn}{r\beta}}
\lf\|\lf\{\sum_{i=1}^m\lf[\frac{\lambda_{i,k}}
{\|\mathbf{1}_{2^kB_i}\|_{X^{\beta}}}\r]^{\beta d}
\mathbf{1}_{B_i}\r\}^{\frac{1}{\beta d}}\r\|_{X^{\beta}}\\
&\quad\sim2^{\frac{kn}{r\beta}}
\lf\|\lf\{\sum_{i=1}^m\lf[\frac{\lambda_{i}}
{\|\mathbf{1}_{B_i}\|_{X^{\beta}}}\r]^{\beta d}
\mathbf{1}_{B_i}\r\}^{\frac{1}{\beta d}}\r\|_{X^{\beta}}.
\end{align*}
Then, from this, \eqref{I-B-01}, \eqref{I-E-1}, \eqref{JN-I-E}, \eqref{cjy4.18x},
$\beta\min\{d,p_-\}\in(\frac{n}{n+s+1-\alpha},\fz)$,
and the definition of
$\|\cdot\|_{\mathcal{L}_{X^{\beta},q_1,s,d\beta}(\rn)}$,
we deduce that, for any given
$r\in(\frac{n}{\beta(n+s+1-\alpha)},\min\{d,p_-\})$,
\begin{align*}
\mathrm{I}_2&\ls\lf\|\lf\{\sum_{j=1}^m\lf[\frac{\lambda_{i}}
{\|\mathbf{1}_{B_i}\|_{X^{\beta}}}\r]^{\beta d}
\mathbf{1}_{B_i}\r\}^{\frac{1}{d\beta}}\r\|_{X^{\beta}}^{-1}
\sum_{j=1}^m\frac{\lambda_{j}|B_j|}
{\|\mathbf{1}_{B_j}\|_{X}}\lf[\fint_{B_j}
\lf|F_{B_j}^{(2)}(x)\r|^{q_2}\,dx\r]^{\frac{1}{q_2}}\\
&\ls\lf\|\lf\{\sum_{j=1}^m\lf[\frac{\lambda_{i}}
{\|\mathbf{1}_{B_i}\|_{X^{\beta}}}\r]^{\beta d}
\mathbf{1}_{B_i}\r\}^{\frac{1}{\beta d}}\r\|_{X^{\beta}}^{-1}
\sum_{j=1}^m\frac{\lambda_{j}|B_j|}
{\|\mathbf{1}_{B_j}\|_{X}}r_j^{s+1}\int_{\rn\setminus 2B_j}
\frac{|f(y)-P_{2B_j}^{(s)}(f)(y)|}{|y-z_j|^{n+s+1-\alpha}}\,dy\\
&\ls\sum_{k=1}^{\fz}\lf\|\lf\{\sum_{j=1}^m\lf[\frac{\lambda_{i}}
{\|\mathbf{1}_{B_i}\|_{X^{\beta}}}\r]^{\beta d}
\mathbf{1}_{B_i}\r\}^{\frac{1}{\beta d}}\r\|_{X^{\beta}}^{-1}\\
&\quad\times\sum_{j=1}^m\lf(2^k r_j\r)^{-s-1+\alpha}r_j^{s+1}
\frac{\lambda_{j}|B_j|}{\|\mathbf{1}_{B_j}\|_{X}}\lf[\fint_{2^{k+1}B_j}
\lf|f(y)-P_{2^{k+1}B_j}^{(s)}(f)(y)\r|^{q_1}
\,dy\r]^{\frac{1}{q_1}}\\
&\ls\sum_{k=1}^{\fz}2^{k(-s-1+\alpha)}\lf\|\lf\{\sum_{i=1}^m
\lf[\frac{\lambda_{i}}
{\|\mathbf{1}_{B_i}\|_{X^{\beta}}}\r]^{\beta d}
\mathbf{1}_{B_i}\r\}^{\frac{1}{\beta d}}\r\|_{X^{\beta}}^{-1}\\
&\quad\times\sum_{j=1}^m\frac{\lambda_{j}|B_j|}
{\|\mathbf{1}_{B_j}\|_{X^{\beta}}}
\lf[\fint_{2^{k+1}B_j}\lf|f(y)-P_{2^{k+1}B_j}^{(s)}(f)(y)
\r|^{q_1}\,dy\r]^{\frac{1}{q_1}}\\
&\ls\sum_{k=1}^{\fz}2^{k(-n-s-1+\alpha+\frac{n}{r\beta})}
\lf\|\lf\{\sum_{i=1}^m\lf[\frac{\lambda_{i,k}}
{\|\mathbf{1}_{2^{k+1}B_i}\|_{X^{\beta}}}\r]^{\beta d}
\mathbf{1}_{2^{k+1}B_i}\r\}^{\frac{1}{\beta
d}}\r\|_{X^{\beta}}^{-1}\\
&\quad\times\sum_{j=1}^m\frac{\lambda_{j,k}|2^{k+1}B_j|}
{\|\mathbf{1}_{2^{k+1}B_j}\|_{X^{\beta}}}\lf[\fint_{2^{k+1}B_j}\lf|f(y)
-P_{2^{k+1}B_j}^{(s)}(f)(y)\r|^{q_1}\,dy\r]^{\frac{1}{q_1}}\\
&\ls \|f\|_{\mathcal{L}_{X^{\beta},q_1,s,\beta d}(\rn)},
\end{align*}
where the implicit positive constants are independent of both $B$ and $f$.
Using this, \eqref{Iw-B-01}, \eqref{I-ball-01'}, and the definition
of $\|\cdot\|_{\mathcal{L}_{X,q_2,s,d}(\rn)}$, we finally conclude that, for
any
$f\in \mathcal{L}_{X^{\beta},q_1,s,\beta d}(\rn)$,
\begin{align*}
\lf\|\widetilde{I}_{\alpha}(f)\r\|_{\mathcal{L}_{X,q_2,s,d}(\rn)}
\ls \|f\|_{\mathcal{L}_{X^{\beta},q_1,s,\beta d}(\rn)}.
\end{align*}
This finishes the proof of \eqref{I-01},
and hence of Theorem \ref{main-Theorem2}.
\end{proof}

Next, we extend the range $\alpha\in(0,1)$ in
Theorem \ref{main-Theorem2} to the range $\alpha\in(0,n)$ as follows.

\begin{theorem}\label{main-corollary}
Let both $X$ and $p_-$ satisfy Assumption \ref{assump1},
$\beta\in(1,\fz)$, $q\in[1,\infty)$, $s\in\zz_+$, $d\in(0,\fz)$,
$\alpha\in (0,n)$, and $\widetilde{I}_{\alpha}$ be the same as
in Remark \ref{rem-I-B}(i).
Further assume that there exists a positive constant $\wz C$
such that, for any ball $B\in\mathbb{B}(\rn)$,
\begin{align}\label{3.22x}
|B|^{\frac{\alpha}{n}}\leq \wz C
\|\mathbf{1}_B\|_X^{\frac{\beta-1}{\beta}}.
\end{align}
\begin{enumerate}
\item[\rm(i)]	
If $s\in(\frac{n}{p_-}-n-1+\frac{\alpha}{n},\fz)\cap\zz_+$, then
$\widetilde{I}_{\alpha}$
is bounded from $\mathcal{L}_{X^{\beta},q,s}(\rn)$
to $\mathcal{L}_{X,q,s}(\rn)$, namely,
there exists a positive constant $C$ such that,
for any $f\in \mathcal{L}_{X^{\beta},q,s}(\rn)$,
\begin{align*}
\lf\|\widetilde{I}_{\alpha}(f)\r\|_{\mathcal{L}_{X,q,s}(\rn)}
\leq C\|f\|_{\mathcal{L}_{X^{\beta},q,s}(\rn)}.
\end{align*}

\item[\rm(ii)]
If $s\in(\frac{n}{ \min\{d,p_-\}}-n-1+\frac{\alpha}{n},\fz)\cap\zz_+$,
then $\widetilde{I}_{\alpha}$ is bounded from
$\mathcal{L}_{X^{\beta},q,s,{\beta d}}(\rn)$
to $\mathcal{L}_{X,q,s,d}(\rn)$, namely,
there exists a positive constant $C$ such that, for any
$f\in \mathcal{L}_{X^{\beta},q,s,{\beta d}}(\rn)$,
\begin{align*}
\lf\|\widetilde{I}_{\alpha}(f)
\r\|_{\mathcal{L}_{X,q,s,d}(\rn)}
\leq C \|f\|_{\mathcal{L}_{X^{\beta},q,s,{\beta d}}(\rn)}.
\end{align*}
\end{enumerate}
\end{theorem}

\begin{proof}
Let all the symbols be the same as in the present theorem.
We only prove (ii) because the proof of (i) is similar.

Next, we show (ii).
The case $\alpha\in(0,1)$ follows from Theorem \ref{main-Theorem2} with $q_1=q_2$ therein.
We now show the case $\alpha\in[1,n)$.
In this case, we have $\frac{\alpha}{n}\in(0,1)$.
In addition, for any $i\in\{1,\ldots,n\}$, let
$$\beta_i:=\frac{(n-i+1)\beta+i-1}{(n-i)\beta+i}.
$$
It is easy to show
$\beta=\Pi_{i=1}^{n}\beta_i$. Moreover, by
\eqref{3.22x}, we conclude that, for any $i\in\{1,\ldots,n\}$,
\begin{align*}
|B|^{\frac{\alpha}{n^2}}\lesssim\|\mathbf{1}_B
\|_{X^{\beta_1\cdots\beta_{i-1}}}
^{\frac{\beta_i-1}{\beta_i}}.
\end{align*}
Using this, the range of $s$, the definition of $\beta_i$
for any $i\in\{1,\ldots,n\}$, and the obtained result for the case $\alpha\in(0,1)$,
we find that, for any $i\in\{1,\ldots,n\}$,
$\widetilde{I}_{\frac{\alpha}{n}}$ is bounded from
$\mathcal{L}_{X^{\beta_1\cdots\beta_{i}},q,s,\beta_1\cdots\beta_{i}d}(\rn)$
to
$\mathcal{L}_{X^{\beta_1\cdots\beta_{i-1}},q,s,\beta_1\cdots\beta_{i-1}d}(\rn)$,
which further implies that $(\wz
I_\frac{\alpha}{n})^{n+1-i}(g)\in\mathcal{L}_{X^{\beta_1
\cdots\beta_{i-1}},q,s,\beta_1\cdots\beta_{i-1}d}(\rn)$.
This finishes the proof of (ii), and hence of Theorem \ref{main-corollary}.
\end{proof}

We now show the necessity of Theorem \ref{thm-main-I}.
To be precise, we prove the following theorem.

\begin{theorem}\label{thm-main-II}
Let all the symbols be the same as in Theorem \ref{thm-main-I}.
Further assume that $\widetilde{I}_{\alpha}$ is bounded from
$\mathcal{L}_{X^{\beta},1,s,{\beta d}}(\rn)$
to $\mathcal{L}_{X,1,s,d}(\rn)$, namely,
there exists a positive constant $C$ such that, for any
$f\in \mathcal{L}_{X^{\beta},1,s,{\beta d}}(\rn)$,
\begin{align*}
\lf\|\widetilde{I}_{\alpha}(f)
\r\|_{\mathcal{L}_{X,1,s,d}(\rn)}
\leq C \|f\|_{\mathcal{L}_{X^{\beta},1,s,{\beta d}}(\rn)}.
\end{align*}
Then there exists a positive constant $\wz C$
such that, for any ball $B\in\mathbb{B}(\rn)$,
\begin{align*}
|B|^{\frac{\alpha}{n}}\leq \wz C
\|\mathbf{1}_B\|_X^{\frac{\beta-1}{\beta}}.
\end{align*}
\end{theorem}

To prove Theorem \ref{thm-main-II}, we first recall the
concept of the Hardy space
$H_X(\rn)$ associated with $X$, originally studied in \cite{SHYY}.
Then we establish several technique lemmas.

In what follows, denote by $\mathcal{S}(\rn)$ the space of all
Schwartz functions equipped with the topology
determined by a well-known countable family of norms, and by
$\mathcal{S}'(\rn)$ its topological dual space equipped with the
weak-$*$ topology. For any $N\in\mathbb{N}$ and
$\phi\in\mathcal{S}(\rn)$, let
\begin{equation*}
p_N(\phi):=\sum_{\alpha\in\mathbb{Z}
^{n}_{+},|\alpha|\leq N}
\sup_{x\in\rn}(1+|x|)^{N+n}|
\partial^{\alpha}\phi(x)|
\end{equation*}
and
\begin{equation*}
\mathcal{F}_N(\rn):=\left\{\phi\in\mathcal{S}
(\rn):\ p_N(\phi)\in[0,1]\right\},
\end{equation*}
where, for any $\alpha:=(\alpha_1,\ldots,\alpha_n)\in\mathbb{Z}^n_+$,
$|\alpha|:=\alpha_1+\cdots+\alpha_n$ and
$\partial^\alpha:=(\frac{\partial}{\partial
x_1})^{\alpha_1}\cdots(\frac{\partial}{\partial
x_n})^{\alpha_n}$. Moreover, for any $r\in\rr$, we denote by $\lfloor r\rfloor$ (resp., $\lceil r\rceil$) the
\emph{maximal} (resp., \emph{minimal})
\emph{integer not greater} (resp., \emph{less}) \emph{than} $r$.

\begin{definition}\label{2d1}
Let $X$ be a ball quasi-Banach
function space and $N\in\nn$ be sufficiently
large. Then the \emph{Hardy
space} $H_X(\rn)$ is defined to be
the set of all the $f\in\mathcal{S}'(\rn)$
such that
$$
\left\|f\right\|_{H_X(\rn)}
:=\left\|\mathcal{M}_N(f)\right\|_{X}<\infty,
$$
where the \emph{non-tangential
grand maximal function}
$\mathcal{M}_{N}(f)$
of $f\in\mathcal{S}'(\rn)$
is defined by setting, for any $x\in\rn$,
\begin{align*}
\mathcal{M}_{N}(f)(x):=
\sup\left\{|f*\phi_{t}(y)|:\
\phi\in\mathcal{F}_{N}(\rn),\
t\in(0,\infty),\ |x-y|<t\right\}.
\end{align*}
\end{definition}

\begin{remark}
Let all the symbols be the same as in Definition \ref{2d1}.
Assume that there exists an $r\in(0,\fz)$ such that
the Hardy--Littlewood maximal operator $\mathcal{M}$ is bounded on $X^{\frac{1}{r}}$.
If $N\in[\lfloor \frac{n}{r}+1\rfloor,\fz)\cap\mathbb{N}$, then,
by \cite[Theorem 3.1]{SHYY}, we find that the Hardy space
$H_X(\rn)$ is independent of
the choice of $N$.
\end{remark}

We also recall the following definitions of both the
$(X,q,s)$-atom and the finite atomic Hardy space
$H_{\mathrm{fin}}^{X,q,s,d}(\rn)$ which are just, respectively,
\cite[Definition 3.5]{SHYY}
and \cite[Definition 1.9]{yyy20}.

\begin{definition}\label{atom}
Let $X$ be a ball quasi-Banach function space, $q\in(1,\infty]$, and
$s\in\zz_+$.
Then a measurable function $a$ on $\rn$ is called an
$(X,q,s)$-\emph{atom}
if there exists a ball $B\in\mathbb{B}(\rn)$ such that
\begin{enumerate}
\item[(i)] $\supp\,(a):=\{x\in\rn:\ a(x)\neq0\}\subset B$;
\item[(ii)]
$\|a\|_{L^q(\rn)}\le\frac{|B|^{\frac{1}{q}}}{\|\mathbf{1}_B\|_X}$;
\item[(iii)] $\int_{\rn} a(x)x^\gamma\,dx=0$ for any
$\gamma:=(\gamma_1,\ldots,\gamma_n)\in\zz_+^n$ with
$|\gamma|:=\gamma_1+\cdots+\gamma_n\le s$,
here and thereafter, for any $x:=(x_1,\ldots,x_n)\in\rn$,
$x^\gamma:=x_1^{\gamma_1}\cdots x_n^{\gamma_n}$.
\end{enumerate}
\end{definition}

\begin{definition}\label{finatom}
Let both $X$ and $p_-$ satisfy Assumption \ref{assump1}. Assume that
$r_0\in(0,\min\{1,p_-\})$ and $p_0\in(r_0,\infty)$
satisfy Assumption \ref{assump2}.
Let $s\in[\lfloor n(\frac{1}{\min\{1,p_-\}}-1)\rfloor,\infty)\cap\zz_+$,
$d\in(0,r_0]$, and $q\in(\max\{1,p_0\},\infty]$.
The \emph{finite atomic Hardy space}
$H_{\mathrm{fin}}^{X,q,s,d}({{\rr}^n})$,
associated with $X$, is defined to be the set of all finite
linear combinations of $(X,q,s)$-atoms. The quasi-norm
$\|\cdot\|_{H_{\mathrm{fin}}^{X,q,s,d}({{\rr}^n})}$ in
$H_{\mathrm{fin}}^{X,q,s,d}({{\rr}^n})$
is defined by setting, for any $f\in
H_{\mathrm{fin}}^{X,q,s,d}({{\rr}^n})$,
\begin{align*}
\|f\|_{H_{\mathrm{fin}}^{X,q,s,d}({{\rr}^n})}&:=\inf\left\{\left\|
\left[\sum_{j=1}^{N}
\left(\frac{{\lambda}_j}{\|{\mathbf{1}}_{B_j}\|_X}\right)^d
{\mathbf{1}}_{B_j}\right]^{\frac1d}\right\|_{X}
\right\},
\end{align*}
where the infimum is taken over all finite linear combinations of $(X,q,s)$-atoms of
$f$, namely, $N\in\nn$,
$f=\sum_{j=1}^{N}\lambda_ja_j$,
$\{\lambda_j\}_{j=1}^{N}\subset[0,\infty)$, and $\{a_j\}_{j=1}^{N}$
being
$(X,q,s)$-atoms supported, respectively,
in the balls $\{B_j\}_{j=1}^{N}\subset\mathbb{B}(\rn)$.
\end{definition}

We also recall the following several lemmas.
Recall that $X$ is said to have an \emph{absolutely continuous
quasi-norm} if,
for any $f\in X$ and any measurable subsets $\{E_j\}_{j\in\nn}\subset \rn$
with $E_{j+1}\subset E_j$ for any $j\in\nn$, and $\bigcap_{j\in\nn}
E_j=\emptyset$,
$\|f\mathbf{1}_{E_j}\|_{X}\downarrow 0$ as $j\to\fz$.
The following lemma is just \cite[Theorem 3.14]{zhyy2022}.

\begin{lemma}\label{2t1}
Let $X$, $d$, $q$, and $s$ be the same as in Definition \ref{finatom},
and $X$ have an absolutely continuous quasi-norm.
Then the dual space of $H_X({{\rr}^n})$, denoted by
$(H_X({{\rr}^n}))^*$,
is $\mathcal{L}_{X,q',s,d}({{\rr}^n})$ with
$\frac{1}{q}+\frac{1}{q'}=1$ in the following sense:
\begin{enumerate}
\item[{\rm (i)}] Let $g\in\mathcal{L}_{X,q',d,s}({{\rr}^n})$.
Then the linear functional
\begin{align}\label{2te1}
L_g:\ f\to \lf\langle L_g,f\r\rangle:=\int_{{{\rr}^n}}f(x)g(x)\,dx,
\end{align}
initially defined for any $f\in
H_{\mathrm{fin}}^{X,q,s,d}({{\rr}^n})$,
has a bounded extension to $H_X({{\rr}^n})$.

\item[{\rm (ii)}] Conversely, any continuous linear
functional on $H_X(\rn)$ arises as in \eqref{2te1}
with a unique $g\in\mathcal{L}_{X,q',s,d}({{\rr}^n})$.
\end{enumerate}
Moreover,
$\|g\|_{\mathcal{L}_{X,q',s,d}({{\rr}^n})}\sim\|L_g\|_{(H_X({{\rr}^n}))^*}$
with the positive equivalence constants independent of both $g$ and $L_g$.
\end{lemma}

\begin{remark}\label{rem-3.19}
Let $X$ be a concave ball quasi-Banach function space and
$d\in(0,1]$.
In this case, Lemma \ref{2t1} coincides with \cite[Theorem
1.12]{yyy20}
which shows that the dual space of $H_X(\rn)$ is
$\mathcal{L}_{X,q',s}({{\rr}^n})$ in this case.
\end{remark}

The following lemma is just \cite[Theorem 3.7]{cjy-01} which gives the
boundedness of $I_\alpha$ from $H_X(\rn)$ to $H_{X^\beta}(\rn)$.

\begin{lemma}\label{thm-Ia-02}
Let $X$, $p_-$, $r_0\in(0,\min\{\frac{1}{\beta},p_-\})$, and
$p_0\in(r_0,\infty)$ satisfy both
Assumptions \ref{assump1} and \ref{assump2} with $\beta\in(1,\infty)$.
Further assume that $X$ has an absolutely continuous quasi-norm.
Let $\alpha\in(0,n)$
and $I_\alpha$ be the same as in \eqref{cla-I}.
Then $I_\alpha$ can be extended to a unique bounded linear
operator, still denoted by $I_\alpha$, from $H_X(\rn)$ to
$H_{X^\beta}(\rn)$, namely,
there exists a positive constant $C$ such that, for any $f\in
H_X(\rn)$,
\begin{align*}
\lf\|I_{\alpha}(f)\r\|_{H_{X^{\beta}}(\rn)}\leq C\|f\|_{H_X(\rn)}
\end{align*}
if and only if there exists a positive constant $\wz C$ such that,
for any ball $B\in\mathbb{B}(\rn)$,
\begin{align*}
|B|^{\frac{\alpha}{n}}\leq \wz C
\|\mathbf{1}_B\|_X^{\frac{\beta-1}{\beta}}.
\end{align*}
\end{lemma}

The following lemma is just \cite[Lemma 2.20]{jtyyz3}.

\begin{lemma}\label{sum-g}
Let $q\in[1,\infty)$, $s\in\zz_+$, and $\theta\in(0,2^{-s})$.
Then there exists a positive constant $C$
such that, for any $f\in L^1_{\mathrm{loc}}(\rn)$
and any ball $B\in\mathbb{B}(\rn)$,
\begin{align*}
&\sum_{k=1}^{\infty}\theta^k\lf[\fint_{2^kB}\lf|f(x)
-P_{B}^{(s)}(f)(x)\r|^q\,dx\r]^{\frac{1}{q}}\\
&\quad\leq C\sum_{k=1}^{\fz}\theta^k\lf[\fint_{2^kB}\lf|f(x)
-P_{2^kB}^{(s)}(f)(x)\r|^q\,dx\r]^{\frac{1}{q}}.
\end{align*}
\end{lemma}

We also recall the following definition of $(X,q,s,\tau)$-molecules,
which is just \cite[Definition 3.8]{SHYY}.

\begin{definition}\label{def-mol}
Let $X$ be a ball quasi-Banach function space, $q\in(1,\infty]$,
$s\in\zz_+$, and $\tau\in(0,\fz)$. Then a measurable function $M$ on
$\rn$ is called an
\emph{$(X,q,s,\tau)$-molecule} centered at a ball
$B\in \mathbb{B}(\rn)$ if
\begin{enumerate}
\item[\rm(i)] for any $j\in\zz_+$,
\begin{align}\label{mole-01}
\lf\|M\mathbf{1}_{L_j}\r\|_{L^q(\rn)}
\leq 2^{-j\tau}\frac{|B|^{\frac{1}{q}}}{\|\mathbf{1}_{B}\|_{X}},
\end{align}
where $L_0:=B$ and, for any $j\in\nn$, $L_j:=2^{j}B\setminus
2^{j-1}B$;
\item[\rm(ii)] $\int_{\rn} M(x)x^\gamma\,dx=0$ for any
$\gamma\in\zz_+^n$ with $|\gamma|\le s$.
\end{enumerate}
\end{definition}

\begin{remark}\label{mole-atom}
Let all the symbols be the same as in Definition \ref{def-mol}.
It is easy to show that any $(X,q,s)$-atom is also an
$(X,q,s,\tau)$-molecule.
\end{remark}

By Lemma \ref{sum-g}, we have the following conclusion.

\begin{lemma}\label{lem-mol-01}
Let both $X$ and $p_-$ satisfy Assumption \ref{assump1},
$q\in[1,{\infty})$,
$\frac{1}{q}+\frac{1}{q'}=1$, $s\in\zz_+$, $d\in(0,\infty)$,
and
$\tau\in(n(\frac{1}{p_-}-\frac{1}{q}),\fz)
\cap(\frac{n}{q'}+s,\fz)$.
Then, for any $(X,q,s,\tau)$-molecule $M$
and any $g\in \mathcal{L}_{X,q',s,d}(\rn)$,
$Mg$ is integrable and, moreover, there exists a positive constant $C$,
independent of both $M$ and $g$, such that
\begin{align}\label{Mg-03}
\lf|\int_{\rn}M(x)g(x)\,dx\r|\leq C\|g\|_{\mathcal{L}_{X,q',s,d}(\rn)}.
\end{align}
\end{lemma}

\begin{proof}
Let all the symbols be the same as in the present lemma.
Let $M$ be an $(X,q,s,\tau)$-molecule centered at a ball $B\in
\mathbb{B}(\rn)$. It is easy to show that,
for any $g\in \mathcal{L}_{X,q',s,d}(\rn)$,
\begin{align}\label{Mg-01}
&\int_{\rn}|M(x)g(x)|\,dx\\
&\quad\leq\int_{\rn}\lf|M(x)P_{B}^{(s)}(g)(x)\r|\,dx
+\int_{\rn}\lf|M(x)\lf[g(x)-P_{B}^{(s)}(g)(x)\r]\r|\,dx\noz\\
&\quad=:\mathrm{A}_1+\mathrm{A}_2.\noz
\end{align}
Using Definition \ref{def-mol}(ii), we know that $\mathrm{A}_1<\fz$.
Then we show $\mathrm{A}_2<\fz$. To this end,
let $L_0:=B$ and, for any $j\in\nn$, $L_j:=2^{j}B\setminus2^{j-1}B$.
Notice that $-\tau+\frac{n}{q'}<s$ due to
$\tau\in(\frac{n}{q'}+s,\fz)$.
Then, from this, the vanishing moments of $M$, the H\"older
inequality, Lemma \ref{sum-g}
with $\theta:=2^{-\tau+\frac{n}{q'}}\in(0,2^{-s})$, \eqref{key-c},
\eqref{mole-01},
$\tau\in(n(\frac{1}{p_-}-\frac{1}{q}),\fz)$,
and the definition of $\|\cdot\|_{\mathcal{L}_{X,q',s,d}(\rn)}$,
we infer that there exists an $r\in(\frac{nq}{\tau q+n},p_-)$
such that, for any $g\in \mathcal{L}_{X,q',s,d}(\rn)$,
\begin{align}\label{Mg-02}
\mathrm{A}_2
&=\sum_{j\in\zz_+}\int_{L_j}\lf|M(x)
\lf[g(x)-P_{B}^{(s)}(g)(x)\r]\r|\,dx\\
&\leq \sum_{j\in\zz_+}\lf\|M\mathbf{1}_{L_j}
\r\|_{L^q(\rn)}\lf|2^jB\r|^{\frac{1}{q'}}
\lf[\fint_{2^jB}\lf|g(x)-P_{B}^{(s)}(g)(x)\r|^{q'}\,dx\r]^{\frac{1}{q'}}\noz\\
&\leq \sum_{j\in\zz_+}2^{-j(\tau-\frac{n}{q'})}
\frac{|B|}{\|\mathbf{1}_{B}\|_{X}}
\lf[\fint_{2^jB}\lf|g(x)-P_{B}^{(s)}(g)(x)\r|^{q'}\,dx\r]^{\frac{1}{q'}}\noz\\
&\ls \sum_{j\in\zz_+}2^{-j(\tau-\frac{n}{q'})}
\frac{|B|}{\|\mathbf{1}_{B}\|_{X}}
\lf[\fint_{2^jB}\lf|g(x)-P_{2^jB}^{(s)}(g)(x)\r|
^{q'}\,dx\r]^{\frac{1}{q'}}\noz\\
&\ls \sum_{j\in\zz_+}2^{-j(\tau+\frac{n}{q}-\frac{n}{r})}
\frac{|2^jB|}{\|\mathbf{1}_{2^jB}\|_{X}}
\lf[\fint_{2^jB}\lf|g(x)-P_{2^jB}^{(s)}(g)(x)\r|
^{q'}\,dx\r]^{\frac{1}{q'}}\noz\\
&\ls \sum_{j\in\zz_+}2^{-j(\tau+\frac{n}{q}-\frac{n}{r})}
\|g\|_{\mathcal{L}_{X,q',s,d}(\rn)}
\lesssim \|g\|_{\mathcal{L}_{X,q',s,d}(\rn)}.\noz
\end{align}	
Using this, $\mathrm{A}_1<\fz$, and \eqref{Mg-01},
we find that $Mg$ is integrable, which, together with
the vanishing moments of $M$ and \eqref{Mg-02}, further implies that
\begin{align*}
\lf|\int_{\rn}M(x)g(x)\,dx\r|
=\lf|\int_{\rn}M(x)
\lf[g(x)-P_{B}^{(s)}(g)(x)\r]\,dx\r|
\leq\mathrm{A}_2\ls \|g\|_{\mathcal{L}_{X,q',s,d}(\rn)}.
\end{align*}
This finishes the proof of \eqref{Mg-03}, and hence of Lemma \ref{lem-mol-01}.
\end{proof}

The following lemma is well known (see, for instance, \cite[p.\,54, Lemma 4.1 and p.\,86]{Lu} for its proof),
which plays a key role in the proof of Proposition \ref{HK-mole} below.

\begin{lemma}\label{lem-4.13x}
Let $s\in\zz_+$. Then
there exists a positive constant $C$ such that,
for any $f\in L^1_{\mathrm{loc}}(\rn)$, any ball $B\in\mathbb{B}(\rn)$,
and any $x\in 2B\setminus B$,
\begin{align}\label{pq2}
\lf|P_{2B\setminus B}^{(s)}(f)(x)\r|
\leq C\fint_{2B\setminus B}|f(x)|\,dx,
\end{align}	
where $P_{2B\setminus B}^{(s)}(f)$ denotes the
unique polynomial $P\in \mathcal{P}_s(\rn)$ such that,
for any $\gamma\in\zz_+^n$ with $|\gamma|\leq s$,
\begin{align*}
\int_{2B\setminus B}\lf[f(x)
-P_{2B\setminus B}^{(s)}(f)(x)\r]x^{\gamma}\,dx=0
\end{align*}
and $f_{2B\setminus B}$ is as in \eqref{fe} with $E$ replaced by $2B\setminus B$.
\end{lemma}	

We also need the following lemma.

\begin{lemma}\label{l4.6x}
Let $X$ be a ball quasi-Banach
function space on $\rn$, $p\in(0,\fz)$, and
$\theta_p:=\frac{\log 2}{\log (2\sigma_p)}$, where
\begin{align*}
\sigma_p:=\inf\lf\{C\in[1,\fz):\ \|f+g\|_{X^p}\leq C[\|f\|_{X^p}+\|g\|_{X^p}],\
\forall\,f,g\in X^p\r\}.
\end{align*}
Then, for any $b\in(0,\theta_p)$ and $\{f_j\}_{j\in\nn}\subset \mathscr{M}(\rn)$,
\begin{align}\label{rem-key-02}
\lf\|\sum_{j\in\nn}|f_j|\r\|_{X^p}
\leq 4^{\frac{1}{\theta_p}}\lf[\sum_{j\in\nn}\lf\|f_j\r\|_{X^p}^{b}\r]^{\frac{1}{b}}.
\end{align}
\end{lemma}

\begin{proof}
Let all the symbols be the same as in the present lemma.
From both the Aoki--Rolewicz theorem
(see \cite{ao,Ro1957}, \cite[Proposition 2.14]{syy}, or \cite[Exercise 1.4.6]{G1}
with a detailed hint on its proof) and Definition \ref{Debqfs}(iii), we deduce that,
for any $\{f_j\}_{j\in\nn}\subset \mathscr{M}(\rn)$,
\begin{align}\label{rem-key-01}
\lf\|\sum_{j\in\nn}|f_j|\r\|_{X^p}
\leq 4^{\frac{1}{\theta_p}}\lf[\sum_{j\in\nn}\lf\|f_j\r\|_{X^p}^{\theta_p}\r]^{\frac{1}{\theta_p}}.
\end{align}
Moreover,
using both \eqref{rem-key-01} and \cite[p.\,7, Proposition 1.5]{ss2011}, we find that,
for any $\{f_j\}_{j\in\nn}\subset \mathscr{M}(\rn)$ and $b\in(0,\theta_p)$,
\eqref{rem-key-02} holds true, which
completes the proof of Lemma \ref{l4.6x}.
\end{proof}

The following proposition gives a special atomic decomposition
for any $(X,q,s,\tau)$-molecule, which plays a vital role in the proof of Theorem \ref{thm-main-II}.

\begin{proposition}\label{HK-mole}
Let $X$, $p_-$, $q$, $s$, and $d$ be the same as in Definition \ref{finatom}, and
\begin{align}\label{Mole-01}
\tau\in \lf(n\lf[\frac{1}{p_-}-\frac{1}{q}\r],\fz\r)
\cap\lf(\frac{n}{q'}+s,\fz\r),
\end{align}
where $\frac{1}{q}+\frac{1}{q'}=1$.
Then, for any $(X,q,s,\tau)$-molecule $M$ centered at the ball $B\in\mathbb{B}(\rn)$,
there exists a sequence $\{c_j\}_{j\in\zz_+}\subset\cc$
and a sequence $\{a_{j}\}_{j\in\zz_+}$ of $(X,q,s)$-atoms supported, respectively,
in $\{2^j B\}_{j\in\zz_+}$ such that
\begin{align}\label{Mac-01}
M=\sum_{j\in\zz_+}c_ja_j
\end{align}
both in $H_X(\rn)$ and pointwisely on $\rn$.
\end{proposition}

\begin{proof}
We prove the present proposition
via borrowing some ideas
from the proof of \cite
[Thoerem 2.9]{tw80} (see also the proof of \cite[Proposition 3.15]{jtyyz3}).
Let all the symbols be the same as in the present proposition.
Let $M$ be an $(X,q,s,\tau)$-molecule
centered at the ball $B(z,r)$ with $z\in\rn$ and $r\in(0,\fz)$.
Without loss of generality, we may assume that $z:=\mathbf{0}$
and $M$ is a real-valued function.
Indeed, if $M$ is complex-valued, then
we only need to consider the real part and the imaginary part
of $M$, respectively.
Moreover, assume that \eqref{Mac-01} holds true for $B:=B(\mathbf{0},r)$.
Since $M(\cdot+z)$ is an $(X,q,s,\tau)$-molecule
centered at the ball $B(\mathbf{0},r)$, then, by the assumption that
\eqref{Mac-01} holds true for $B(\mathbf{0},r)$, we find that
there exists a sequence $\{c_j\}_{j\in\zz_+}\subset\cc$
and a sequence $\{a_{j}\}_{j\in\zz_+}$ of $(X,q,s)$-atoms supported, respectively,
in $\{2^jB(\mathbf{0},r)\}_{j\in\zz_+}$ such that
\begin{align}\label{sjtx}
M(\cdot+z)=\sum_{j\in\zz_+}c_ja_j(\cdot)
\end{align}
both in $H_X(\rn)$ and pointwisely on $\rn$.
Now, we claim that
\begin{align}\label{jgsj}
M(\cdot)=\sum_{j\in\zz_+}c_ja_j(\cdot-z)
\end{align}
both in $H_X(\rn)$ and pointwisely on $\rn$.
It is easy to prove that \eqref{jgsj} holds true pointwisely on $\rn$.
Then we show that \eqref{jgsj} also holds true in $H_X(\rn)$. To this end,
let $\psi\in\mathcal{S}(\rn)$ with $\int_{\rn}\psi(x)\,dx=1$, and
$\wz\psi(\cdot):=\psi(\cdot+z)$.
From \cite[Theorem 3.1]{SHYY} and \eqref{sjtx}, we infer that, for any $N\in\nn$,
\begin{align*}
&\lf\|M(\cdot)-\sum_{j=1}^{N}c_ja_j(\cdot-z)\r\|_{H_X(\rn)}\\
&\quad\sim\lf\|\sup_{t\in(0,\infty)}\lf|\wz\psi_t*
\lf[M(\cdot)-\sum_{j=1}^{N}c_ja_j(\cdot-z)\r]\r|\,\r\|_X\\
&\quad\sim\lf\|\sup_{t\in(0,\infty)}\lf|\psi_t*
\lf[M(\cdot+z)-\sum_{j=1}^{N}c_ja_j(\cdot)\r]\r|\,\r\|_X
\sim\lf\|M(\cdot+z)-\sum_{j=1}^{N}c_ja_j(\cdot)\r\|_{H_X(\rn)}\\
&\quad\to0
\end{align*}
as $N\to\infty$, where the implicit positive constants are independent of
both $M$ and $N$, but depend on both $\psi$ and $z$.
This implies that the above claim holds true.
In addition, for any $j\in\zz_+$, let
$$
\wz c_j:=\frac{c_j\|\mathbf{1}_{2^jB(z,r)}\|_X}
{\|\mathbf{1}_{2^jB(\mathbf{0},r)}\|_X}\quad\text{and}\quad
\wz a_j(\cdot):=\frac{a_j(\cdot-z)\|\mathbf{1}_{2^jB(\mathbf{0},r)}\|_X}
{\|\mathbf{1}_{2^jB(z,r)}\|_X}.
$$
Then $\{\wz a_j\}_{j\in\zz_+}$ is a sequence of $(X,q,s)$-atoms supported, respectively,
in $\{B(z,r)\}_{j\in\zz_+}$ and
$$
M=\sum_{j\in\zz_+}\wz c_j\wz a_j
$$
pointwisely on $\rn$, that is,
the desired conclusion of the present proposition holds true.
Thus, the above claim holds true and, in the remainder of the resent proof, we assume that $z:=\mathbf{0}$
and $M$ is a real-valued function.

To show the present proposition, let $L_0:=B(\mathbf{0},r)$ and,
for any $j\in\nn$, let $L_j:=2^{j}B(\mathbf{0},r)\setminus 2^{j-1}B(\mathbf{0},r)$.
Moreover, for any $j\in\zz_+$ and $x\in \rn$, we use
$P_{L_j}^{(s)}(M)$ to denote the
unique polynomial of $\mathcal{P}_s(\rn)$
such that, for any $\gamma\in\zz_+^n$ with
$|\gamma|\leq s$, and for any $j\in\zz_+$,
\begin{align}\label{HK-M-1}
\int_{L_j}\lf[M(x)-P_{L_j}^{(s)}(M)(x)\r]x^{\gamma}\,dx=0.
\end{align}
We first present some well-known
facts on $P_j:=P_{L_j}^{(s)}(M)$ with $j\in\zz_+$
(see, for instance, \cite[pp.\,76-77]{tw80} and also \cite[pp.\,85-88]{Lu}).
For any given $j\in\zz_+$,
let $\{\phi_\nu^{(j)}:\ \nu\in\zz_+^n\text{ and }|\nu|\leq s\}$
be the \emph{orthogonal polynomials} with weight $\frac{1}{|L_j|}$
by means of the Gram--Schmidt method from
$\{x^\nu:\ \nu\in\zz_+^n\text{ and }|\nu|\leq s\}$
restricted on $L_j$, namely,
$\{\phi_\nu^{(j)}:\ \nu\in\zz_+^n\text{ and }
|\nu|\leq s\}\subset \mathcal{P}_s(\rn)$
and, for any $\nu_1,\nu_2\in\zz_+^n$ with
$|\nu_1|\leq s$ and $|\nu_2|\leq s$,
\begin{align*}
\lf\langle \phi_{\nu_1}^{(j)}, \phi_{\nu_2}^{(j)}\r\rangle_{L_j}
:=\frac{1}{|L_j|}\int_{L_j}
\phi_{\nu_1}^{(j)}(x)\phi_{\nu_2}^{(j)}(x)\,dx
=\begin{cases}
\displaystyle
1,\ & \nu_1=\nu_2,\\
\displaystyle
0,\ & \nu_1\neq \nu_2.
\end{cases}
\end{align*}
It is easy to prove that, for any $x\in L_j$ with $j\in\zz_+$,
\begin{align}\label{P-j}
P_j(x)=\sum_{\{\nu\in\zz_+^n:\ |\nu|\leq s\}}
\lf\langle M,\phi_{\nu}^{(j)}
\r\rangle_{L_j}\phi_\nu^{(j)}(x).
\end{align}
Moreover, for any given $j\in\zz_+$, we use
$\{\psi_{\nu}^{(j)}:\ \nu\in\zz_+^n\text{ and }|\nu|\leq s\}$
to denote the \emph{dual basis} of
$\{x^{\nu}:\ \nu\in\zz_+^n\text{ and }|\nu|\leq s\}$
restricted on $L_j$ with respect to the weight $\frac{1}{|L_j|}$, namely,
$\psi_{\nu}^{(j)}\in \mathcal{P}_s(\rn)$ and,
for any $\nu_1,\nu_2\in\zz_+^n$ with $|\nu_1|\leq s$ and $|\nu_2|\leq s$,
\begin{align}\label{4.10x}
\lf\langle \psi_{\nu_1}^{(j)}, x^{\nu_2}\r\rangle_{L_j}
:=\frac{1}{|L_j|}\int_{L_j}\psi_{\nu_1}^{(j)}(x)x^{\nu_2}\,dx
=\begin{cases}
\displaystyle
1,\ & \nu_1=\nu_2,\\
\displaystyle
0,\ & \nu_1\neq \nu_2.
\end{cases}
\end{align}
By the definition of the dual basis, we conclude that,
if, for any $\nu\in\zz_+^n$ with $|\nu|\leq s$,
$j\in\zz_+$, and $x\in L_j$, $\phi_{\nu}^{(j)}(x)
=\sum_{\{\gamma\in\zz_+^n:\ |\gamma|\leq s\}}
m_{\nu,\gamma}^{(j)}x^{\gamma}$
with $\{m_{\nu,\gamma}^{(j)}\}\subset \rr$, then
$$
\psi_{\nu}^{(j)}=\sum_{\{\gamma\in\zz_+^n:\ |\gamma|\leq s\}}
m_{\gamma,\nu}^{(j)}\phi_{\gamma}^{(j)},
$$
which, together with \eqref{P-j}, further implies that, for any $j\in\zz_+$,
\begin{align}\label{3.8x}
P_j\mathbf{1}_{L_j}=\sum_{\{\nu\in\zz_+^n:\ |\nu|\leq s\}}
\lf\langle M,x^{\nu}\r\rangle_{L_j}\psi_{\nu}^{(j)}\mathbf{1}_{L_j}.
\end{align}

To prove the present proposition,
we first estimate $M\mathbf{1}_{2^lB(\mathbf{0},r)}$
for any $l\in\nn$. Indeed, for any $l\in\nn$, we have
\begin{align}\label{M-l'}
M\mathbf{1}_{2^lB(\mathbf{0},r)}
&=\sum_{j=1}^l \lf[M\mathbf{1}_{L_j}-P_j\mathbf{1}_{L_j}\r]
+\sum_{j=0}^{l}P_j\mathbf{1}_{L_j}\\
&=\sum_{j=1}^l \alpha_j
+\sum_{j=0}^{l}P_j\mathbf{1}_{L_j},\noz
\end{align}
where $\alpha_j:=(M-P_j)\mathbf{1}_{L_j}$.

Next, we show that, for any $j\in\{1,\ldots,l\}$ with $l\in\nn$, $\alpha_j$
is an $(X,q,s)$-atom multiplied by a positive constant.
Indeed, from \eqref{HK-M-1}, we deduce that, for any $j\in\zz_+$
and $\gamma\in\zz_+^n$ with $|\gamma|\leq s$,
\begin{align}\label{3.10xx}
\int_{\rn}\alpha_j(x)x^\gamma\,dx=0
\ \text{and}\ \supp\,(\alpha_j)\subset L_j\subset2^jB(\mathbf{0},r).
\end{align}
Thus, for any $j\in\zz_+$,
$\alpha_j$ satisfies both (i) and (iii) of Definition \ref{atom}.
Then we prove that, for any $j\in\zz_+$, $\alpha_j$ multiplied by a positive constant
satisfies Definition \ref{atom}(ii). Indeed, in what follows,
we choose a $u\in(\frac{nq}{\tau q+n},p_-)$ due to \eqref{Mole-01}.
Then, using \eqref{pq2}, the H\"older inequality,
Definition \ref{def-mol}(i), and \eqref{key-c},
we conclude that, for any $j\in\zz_+$,
\begin{align*}
\|\alpha_j\|_{L^q(\rn)}
&\leq \lf\|M\mathbf{1}_{L_j}\r\|_{L^q(\rn)}
+\lf[\int_{L_j}|P_j(x)|^q\,dx\r]^{\frac{1}{q}}\\
&\leq \lf\|M\mathbf{1}_{L_j}\r\|_{L^q(\rn)}
+|L_j|^{\frac{1}{q}}\fint_{L_j}|M(x)|\,dx\\
&\ls\lf\|M\mathbf{1}_{L_j}\r\|_{L^q(\rn)}
\ls2^{-j\tau}\frac{|B(\mathbf{0},r)|^{\frac{1}{q}}}
{\|\mathbf{1}_{B(\mathbf{0},r)}\|_{X}}\\
&\ls2^{-j(\tau+\frac{n}{q}-\frac{n}{u})}\frac{|2^jB(\mathbf{0},r)|^{\frac{1}{q}}}
{\|\mathbf{1}_{2^jB(\mathbf{0},r)}\|_{X}},
\end{align*}
which, combined with \eqref{3.10xx}, further implies that
there exists a positive constant $C_1$, independent of $j$, such that, for any $j\in\zz_+$,
$A_j:=\frac{\alpha_j}{\lambda_j}$ is an $(X,q,s)$-atom, where
\begin{align}\label{lam-j}
\lambda_j:=C_12^{-j(\tau+\frac{n}{q}-\frac{n}{u})}.
\end{align}
This is a desired conclusion for any $\alpha_j$ with $j\in\{1,\ldots,l\}$ and $l\in\nn$.

We now consider
$\sum_{j=0}^{l}P_j\mathbf{1}_{L_j}$ for any $l\in\nn$. To this end, for any
$\nu\in\zz_+^n$ with $|\nu|\leq s$, and for any $j\in\zz_+$, let
\begin{align}\label{4.12x}
\eta^{(j)}_{\nu}:=\int_{\rn\setminus 2^jB(\mathbf{0},r)}M(x)x^{\nu}
\,dx\ \text{and}\ \eta_{\nu}^{(-1)}
:=\int_{\rn}M(x)x^{\nu}\,dx=0.
\end{align}
Then, by \eqref{3.8x},
similarly to the proof of \cite[(3.16)]{jtyyz3}, we have, for any $l\in\nn$,
\begin{align}\label{sum-Pj}
\sum_{j=0}^{l}P_j\mathbf{1}_{L_j}
=\sum_{\{\nu\in\zz_+^n:\ |\nu|\leq s\}}
\sum_{j=0}^{l-1}\widetilde{\alpha}_\nu^{(j)}
-\sum_{\{\nu\in\zz_+^n:\ |\nu|\leq s\}}
\frac{\eta_{\nu}^{(l)}\psi_{\nu}^{(l)}\mathbf{1}_{L_l}}{|L_l|},
\end{align}
where, for any $\nu\in\zz_+^n$ with $|\nu|\leq s$, and for any $j\in\{0,\ldots,l-1\}$,
\begin{align}\label{4.32x}
\widetilde{\alpha}_\nu^{(j)}
:=\eta_{\nu}^{(j)}\lf[\frac{\psi_{\nu}^{(j+1)}
\mathbf{1}_{L_{j+1}}}{|L_{j+1}|}
-\frac{\psi_{\nu}^{(j)}\mathbf{1}_{L_j}}{|L_j|}\r].
\end{align}

Now, we claim that, for any $\nu\in\zz_+^n$
with $|\nu|\leq s$, and for any $j\in\zz_+$, $\widetilde{\alpha}_\nu^{(j)}$
is an $(X,q,s)$-atom multiplied by a positive constant
and, for any $\phi\in \mathcal{S}(\rn)$,
\begin{align}\label{eta-3'}
\lim_{l\to\fz}\int_{\rn}\frac{\eta_{\nu}^{(l)}\psi_{\nu}^{(l)}(x)
\mathbf{1}_{L_l}(x)}{|L_l|}\phi(x)\,dx
=0.
\end{align}
Indeed, using \eqref{4.10x}, we conclude that,
for any $\nu,\gamma\in\zz_+^n$ with $|\nu|\leq s$
and $|\gamma|\leq s$, and for any $j\in\zz_+$,
\begin{align}\label{4.14x}
\int_{\rn}\widetilde{\alpha}_\nu^{(j)}(x)x^{\gamma}\,dx=0\ \text{and}\
\supp\,(\widetilde{\alpha}_{\nu}^{(j)})\subset (L_{j+1}\cup L_j)\subset 2^{j+1}B(\mathbf{0},r).
\end{align}
Moreover, by \cite[Lemma 3.16]{tw80} (see also \cite[(7.2)]{Lu}),
we find that there exists a positive constant
$C_2$, independent of $\nu$, $j$, and $r$, such that,
for any $\nu\in\zz_+^n$ with $|\nu|\leq s$,
and for any $x\in L_j$ with $j\in\zz_+$,
\begin{align}\label{L-7.2}
\lf|\psi_{\nu}^{(j)}(x)\r|\leq \frac{C_2}{(2^{j-1}r)^{|\nu|}}.
\end{align}
By this, \eqref{4.12x}, the H\"older inequality,
Definition \ref{def-mol}(i), $\tau>\frac{n}{q'}+s$, and \eqref{key-c},
we conclude that there exists a $u\in(0,p_-)$ such that, for any
$\nu\in \zz_+^n$ with $|\nu|\leq s$, and for any $j\in\zz_+$,
\begin{align}\label{eta-1}
\lf|\eta_{\nu}^{(j)}\r|
&\leq\int_{\rn\setminus 2^jB(\mathbf{0},r)}\lf|M(x)x^{\nu}\r|\,dx
=\sum_{i=j+1}^{\fz}\int_{L_i}\lf|M(x)x^{\nu}\r|\,dx\\
&\le \sum_{i=j+1}^{\fz}\lf\|M\mathbf{1}_{L_i}\r\|_{L^q(\rn)}
|L_i|^{\frac{1}{q'}}\lf(2^{i}r\r)^{|\nu|}
\ls\sum_{i=j+1}^{\fz}2^{-i\tau}
\frac{|B(\mathbf{0},r)|^{\frac{1}{q}}}{\|\mathbf{1}_{B(\mathbf{0},r)}\|_X}
\lf|2^iB(\mathbf{0},r)\r|^{\frac{1}{q'}+\frac{|\nu|}{n}}\noz\\
&\sim\sum_{i=j+1}^{\fz}2^{-i(\tau-\frac{n}{q'}-|\nu|)}
\frac{|B(\mathbf{0},r)|^{1+\frac{|\nu|}{n}}}{\|\mathbf{1}_{B(\mathbf{0},r)}\|_X}
\sim2^{-j(\tau-\frac{n}{q'}-|\nu|)}
\frac{|B(\mathbf{0},r)|^{1+\frac{|\nu|}{n}}}{\|\mathbf{1}_{B(\mathbf{0},r)}\|_X}\noz\\
&\ls 2^{-j(\tau+\frac{n}{q}-\frac{n}{u})}
\frac{|2^{j+1}B(\mathbf{0},r)|^{1+\frac{|\nu|}{n}}}{\|\mathbf{1}_{2^{j+1}B(\mathbf{0},r)}\|_X}.\noz
\end{align}
From \eqref{4.32x}, \eqref{L-7.2}, and \eqref{eta-1}, it follows that,
for any $\nu\in\zz_+^n$ with $|\nu|\leq s$, and for any $j\in\zz_+$,
\begin{align*}
\lf\|\widetilde{\alpha}_\nu^{(j)}\r\|_{L^q(\rn)}
&\leq \lf\|\widetilde{\alpha}_\nu^{(j)}
\r\|_{L^{\fz}(\rn)}\lf|2^{j+1}B(\mathbf{0},r)\r|^{\frac{1}{q}}
\ls\frac{|\eta_{\nu}^{(j)}|
|2^{j+1}B(\mathbf{0},r)|^{\frac{1}{q}}}{(2^{j-1}r)^{|\nu|}|L_j|}\\
&\sim\frac{|\eta_{\nu}^{(j)}|}{(2^{j-1}r)^{|\nu|}}
\lf|2^{j+1}B(\mathbf{0},r)\r|^{-\frac{1}{q'}}
\sim\lf|\eta_{\nu}^{(j)}\r|
\lf|2^{j+1}B(\mathbf{0},r)\r|^{-\frac{1}{q'}-\frac{|\nu|}{n}}\\
&\ls2^{-j(\tau+\frac{n}{q}-\frac{n}{u})}
\frac{|2^{j+1}B(\mathbf{0},r)|^{\frac{1}{q}}}{\|\mathbf{1}_{2^{j+1}B(\mathbf{0},r)}\|_X}.
\end{align*}
By this and \eqref{4.14x}, we conclude that there exists a positive constant
$\widetilde{C}$, independent of $j$, such that, for any
$\nu\in\zz_+^n$ with $|\nu|\leq s$,
and for any $j\in\zz_+$, $\widetilde{A}_{\nu}^{(j)}
:=\frac{\widetilde{\alpha}_\nu^{(j)}}{\widetilde{\lambda}_{j}}$
is an $(X,q,s)$-atom,
where
\begin{align}\label{4.13x}
\widetilde{\lambda}_{j}:=\widetilde{C}2^{-j(\tau+\frac{n}{q}-\frac{n}{u})}.
\end{align}
This is a desired conclusion of $\widetilde{\alpha}_{\nu}^{(j)}$ for any
$j\in\zz_+$ and $\nu\in\zz_+^n$ with
$|\nu|\leq s$.

Next, we prove \eqref{eta-3'}.
Indeed, using \eqref{L-7.2},
\eqref{eta-1}, and $\tau>\frac{n}{q'}$,
we find that, for any given $\phi\in \mathcal{S}(\rn)$ and for any $l\in\nn$,
\begin{align*}\
&\lf|\int_{\rn}\frac{\eta_{\nu}^{(l)}\psi_{\nu}^{(l)}(x)
\mathbf{1}_{L_l}(x)}{|L_l|}\phi(x)\,dx\r|\\
&\quad
\lesssim2^{l(-\tau+\frac{n}{q'})}
\fint_{L_l}\lf|\phi(x)\r|\,dx
\lesssim 2^{l(-\tau+\frac{n}{q'})}\|\phi\|_{L^{\fz}(\rn)}\to 0
\end{align*}
as $l\to\fz$, where the implicit positive constants depend on $r$.
This shows that \eqref{eta-3'} holds true.

Finally, we show that
\begin{align}\label{M-Ml-JN**}
M=\sum_{j\in\zz_+}\lambda_j A_j
+\sum_{j\in\zz_+}\sum_{\{\nu\in\zz_+^n:\ |\nu|\leq s\}}
\widetilde{\lambda}_{j}\widetilde{A}_\nu^{(j)}
\end{align}
both in $H_X(\rn)$ and pointwisely on $\rn$.
Indeed, using both \eqref{M-l'} and \eqref{sum-Pj},
we conclude that, for any $l\in\nn$,
\begin{align}\label{3.19x}
M\mathbf{1}_{2^lB(\mathbf{0},r)}
&=\sum_{j=0}^{l}\lambda_j A_j+
\sum_{\{\nu\in\zz_+^n:\ |\nu|\leq s\}}
\sum_{j=0}^{l-1}
\widetilde{\lambda}_{j}\widetilde{A}_\nu^{(j)}\\
&\quad-\sum_{\{\nu\in\zz_+^n:\ |\nu|\leq s\}}
\frac{\eta_{\nu}^{(l)}\psi_{\nu}^{(l)}\mathbf{1}_{L_l}}{|L_l|}.\noz
\end{align}
By the definition of $\{L_j\}_{j\zz_+}$, we easily find that,
for any $x\in \rn$, there exists a unique $l_x\in\zz_+$ such that $x\in L_{l_x}$,
which, together with \eqref{3.19x}, $\supp\,(A_j)\subset L_j$,
and $\supp\,(\widetilde{A}_{\nu}^{(j)})\subset (L_{j+1}\cup L_j)$,
further implies that
\begin{align*}
M(x)&=M(x)\mathbf{1}_{2^{l_x+1}B(\mathbf{0},r)}(x)\\
&=\sum_{j=0}^{l_x+1}\lambda_j A_j(x)+
\sum_{\{\nu\in\zz_+^n:\ |\nu|\leq s\}}
\sum_{j=0}^{l_x}
\widetilde{\lambda}_{j}\widetilde{A}_\nu^{(j)}(x)\\
&\quad-\sum_{\{\nu\in\zz_+^n:\ |\nu|\leq s\}}
\frac{\eta_{\nu}^{(l_x+1)}\psi_{\nu}^{(l_x+1)}(x)\mathbf{1}_{L_{l_x+1}}(x)}{|L_l|}\\
&=\sum_{j=0}^{\fz}\lambda_j A_j(x)+
\sum_{\{\nu\in\zz_+^n:\ |\nu|\leq s\}}
\sum_{j=0}^{\fz}
\widetilde{\lambda}_{j}\widetilde{A}_\nu^{(j)}(x).
\end{align*}
This shows that \eqref{M-Ml-JN**} holds true pointwisely on $\rn$.
Besides, from \eqref{3.19x}, the assumption that $M\in L^q(\rn)$, \eqref{3.19x},
the Lebesgue dominated convergence theorem,
and \eqref{eta-3'}, we deduce that,
for any given $\phi\in \mathcal{S}(\rn)$ and for any $l\in\nn$,
\begin{align*}
&\lf|\lf\langle M-\sum_{j=0}^{l}\lambda_j A_j
-\sum_{j=0}^{l-1}\sum_{\{\nu\in\zz_+^n:\ |\nu|\leq s\}}
\widetilde{\lambda}_{j}\widetilde{A}_\nu^{(j)},\phi\r\rangle\r|\\
&\quad=\lf|\int_{\rn}\lf[M(x)-\sum_{j=1}^{l}\lambda_j A_j(x)
-\sum_{j=0}^{l-1}\sum_{\{\nu\in\zz_+^n:\ |\nu|\leq s\}}
\widetilde{\lambda}_{j}\widetilde{A}_\nu^{(j)}(x)\r]\phi(x)\,dx\r|\\
&\quad\leq\sum_{\{\nu\in\zz_+^n:\ |\nu|\leq s\}}
\lf|\int_{\rn}\frac{\eta_{\nu}^{(l)}\psi_{\nu}^{(l)}(x)
\mathbf{1}_{L_l}(x)}{|L_l|}\phi(x)\,dx\r|
+\lf|\int_{\rn\setminus 2^lB(\mathbf{0},r)}M(x)\phi(x)\,dx\r|\\
&\quad\to0
\end{align*}
as $l\to\fz$, and hence
\begin{align}\label{M-Ml-JN*}
M=\sum_{j\in\zz_+}\lambda_j A_j
+\sum_{j\in\zz_+}\sum_{\{\nu\in\zz_+^n:\ |\nu|\leq s\}}
\widetilde{\lambda}_{j}\widetilde{A}_\nu^{(j)}
\end{align}
in $\mathcal{S}'(\rn)$.
Using this,
\cite[Theorem 3.6]{SHYY}, \eqref{rem-key-02},  \eqref{M-Ml-JN*}, \eqref{lam-j}, \eqref{4.13x}, and
$\tau>n(\frac{1}{u}-\frac{1}{q})$, we conclude that there exists a $b\in(0,d)$
such that, for any $N\in\nn$,
\begin{align*}
&\lf\|M-\sum_{j=0}^N\lambda_j A_j
-\sum_{j=0}^N\sum_{\{{\nu}\in\zz_+^n:\ |{\nu}|\leq s\}}
\widetilde{\lambda}_{j}\widetilde{A}_{\nu}^{(j)}\r\|_{H_X(\rn)}\\
&\quad\ls\lf\|\lf\{\sum_{j=N+1}^{\fz}\lf[\frac{\lambda_j}
{\|\mathbf{1}_{2^jB(\mathbf{0},r)}\|_X}\r]^{b}
\mathbf{1}_{2^jB(\mathbf{0},r)}\r\}^{\frac{1}{{b}}}\r\|_{X}
+\lf\|\lf\{\sum_{j=N+1}^{\fz}\lf[\frac{\wz\lambda_j}
{\|\mathbf{1}_{2^{j+1}B(\mathbf{0},r)}\|_X}\r]^{b}
\mathbf{1}_{2^{j+1}B(\mathbf{0},r)}\r\}^{\frac{1}{{b}}}\r\|_{X}\\
&\quad\sim\lf\|\sum_{j=N+1}^{\fz}\lf[\frac{\lambda_j}
{\|\mathbf{1}_{2^jB(\mathbf{0},r)}\|_X}\r]^{b}\mathbf{1}_{2^jB(\mathbf{0},r)}
\r\|_{X^{\frac{1}{b}}}^{\frac{1}{{b}}}
+\lf\|\sum_{j=N+1}^{\fz}\lf[\frac{\wz\lambda_j}
{\|\mathbf{1}_{2^{j+1}B(\mathbf{0},r)}\|_X}\r]^{b}\mathbf{1}_{2^{j+1}B(\mathbf{0},r)}
\r\|_{X^{\frac{1}{{b}}}}^{\frac{1}{{b}}}\\
&\quad\ls\lf\{\sum_{j=N+1}^{\fz}\lf\|\lf[\frac{\lambda_j}
{\|\mathbf{1}_{2^jB(\mathbf{0},r)}\|_X}\r]^{b}\mathbf{1}_{2^jB(\mathbf{0},r)}
\r\|_{X^{\frac{1}{{b}}}}\r\}^{\frac{1}{b}}
+\lf\{\sum_{j=N+1}^{\fz}\lf\|\lf[\frac{\wz\lambda_j}
{\|\mathbf{1}_{2^{j+1}B(\mathbf{0},r)}\|_X}\r]^{b}\mathbf{1}_{2^{j+1}B(\mathbf{0},r)}
\r\|_{X^{\frac{1}{{b}}}}\r\}^{\frac{1}{{b}}}\\
&\quad\sim\lf[\sum_{j=N+1}^{\fz}2^{-j(\tau+\frac{n}{q}-\frac{n}{u}){b}}\r]^{\frac{1}{{b}}}
+\lf[\sum_{j=N+1}^{\fz}2^{-(j+1)(\tau+\frac{n}{q}-\frac{n}{u}){b}}\r]^{\frac{1}{{b}}}\to0
\end{align*}
as $N\to\fz$.
By this and \cite[Theorem 3.6]{SHYY}, we find that \eqref{M-Ml-JN**} holds true,
which further implies that \eqref{Mac-01} holds true.
This finishes the proof of Proposition \ref{HK-mole}.
\end{proof}

We also need the following lemma; since its proof is similar to that
of \cite[Lemma 3.14]{jtyyz3}, we only sketch some important steps.

\begin{lemma}\label{lem-518-01}
Let both $X$ and $p_-$ satisfy Assumption \ref{assump1}.
Let $q\in[1,\fz)$, $s\in\zz_+$, and $d\in (0,\fz)$.
Then there exists a positive constant $C$ such that, for any
given $u\in(0,p_-)$, and for any $g\in \mathcal{L}_{X,q,s,d}(\rn)$, any
ball $B(z,r)\in\mathbb{B}(\rn)$ with $z\in\rn$ and $r\in(0,\fz)$,
and any $k\in\zz_+$,
\begin{align*}
&\lf[\fint_{2^kB(z,r)}\lf|g(x)
-P_{B(z,r)}^{(s)}(g)(x)\r|^q\,dy\r]^{\frac{1}{q}}\\
&\quad\leq C k\lf[2^{ks}+2^{nk(\frac{1}{u}-1)}\r]
\frac{\|\mathbf{1}_{B(z,r)}\|_{X}}{|B(z,r)|}
\|g\|_{\mathcal{L}_{X,q,s,d}(\rn)}.
\end{align*}
\end{lemma}

\begin{proof}
Let all the symbols be the same as in the present lemma.
Also, let $u\in(0,p_-)$. Then,
using the proof of \cite[Lemma 3.14]{jtyyz3}, we conclude that,
for any $g\in \mathcal{L}_{X,q,s,d}(\rn)$, any
ball $B(z,r)\in\mathbb{B}(\rn)$ with both $z\in\rn$ and $r\in(0,\fz)$,
and any $k\in\zz_+$,
\begin{align*}
&\lf[\fint_{2^kB(z,r)}\lf|g(x)
-P_{B(z,r)}^{(s)}(g)(x)\r|^q\,dy\r]^{\frac{1}{q}}\\
&\quad\ls\sum_{j=1}^k 2^{(k-j+1)s}\lf[\fint_{2^jB(z,r)}
\lf|g(x)-P_{2^jB(z,r)}^{(s)}(g)(x)\r|^q\,dx\r]^{\frac{1}{q}},
\end{align*}
which, combined with the definition of $\|\cdot\|_{\mathcal{L}_{X,q,s,d}(\rn)}$,
$u\in(0,p_-)$, and \eqref{key-c}, further implies that
\begin{align*}
&\lf[\fint_{2^kB(z,r)}\lf|g(x)
-P_{B(z,r)}^{(s)}(g)(x)\r|^q\,dy\r]^{\frac{1}{q}}\\
&\quad\ls\sum_{j=1}^k 2^{(k-j+1)s}
\frac{\|\mathbf{1}_{2^jB(z,r)}\|_{X}}{|2^jB(z,r)|}
\|g\|_{\mathcal{L}_{X,q,s,d}(\rn)}\\
&\quad\ls\sum_{j=1}^k 2^{(k-j+1)s+jn(\frac{1}{u}-1)}
\frac{\|\mathbf{1}_{B(z,r)}\|_{X}}{|B(z,r)|}
\|g\|_{\mathcal{L}_{X,q,s,d}(\rn)}\\
&\quad\ls k\lf[2^{ks}+2^{nk(\frac{1}{u}-1)}\r]
\frac{\|\mathbf{1}_{B(z,r)}\|_{X}}{|B(z,r)|}
\|g\|_{\mathcal{L}_{X,q,s,d}(\rn)}.
\end{align*}
This finishes the proof of Lemma \ref {lem-518-01}.
\end{proof}

Using both Proposition \ref{HK-mole} and Lemma \ref{lem-518-01}, we obtain the following conclusion.

\begin{lemma}\label{thm-dual-T}
Let $X$, $p_-$, $d$, $q$, and $s$ be the same as in Definition \ref{finatom}.
Let
\begin{align}\label{tau-01}
\tau\in \lf(n\lf[\frac{1}{p_-}-\frac{1}{q}\r],\fz\r)
\cap\lf(\frac{n}{q'}+s,\fz\r),
\end{align}
where $\frac{1}{q}+\frac{1}{q'}=1$.
Further assume that $X$ has an absolutely continuous quasi-norm.
Then, for any $(X,q,s,\tau)$-molecule $M$ and any $g\in
\mathcal{L}_{X,q',s,d}(\rn)$,
\begin{align}\label{518-01}
\lf\langle L_g,M\r\rangle=\int_{\rn}M(x)g(x)\,dx,
\end{align}
where $L_g$ is the same as in \eqref{2te1}.
\end{lemma}

\begin{proof}
Let all the symbols be the same as in the present lemma.
Let $M$ be any given $(X,q,s,\tau)$-molecule centered at the ball $B(z,r)$
with both $z\in\rn$ and $r\in(0,\fz)$. Then, from \cite[Theorem 3.9]{SHYY}, we deduce that $M\in H_X(\rn)$.
Without loss of generality, we may assume that $z=\mathbf{0}$.
Indeed, if \eqref{518-01} holds true for $z=\mathbf{0}$, then, for
any $z\in\rn$, $M(\cdot+z)$ is an $(X,q,s,\tau)$-molecule centered at the ball $B(\mathbf{0},r)$.
By Proposition \ref{HK-mole}, we conclude that
there exists a sequence $\{c_j\}_{j\in\zz_+}\subset\cc$
and a sequence $\{a_{j}\}_{j\in\zz_+}$ of $(X,q,s)$-atoms supported, respectively,
in $\{2^j B(z,r)\}_{j\in\zz_+}$ such that
\begin{align}\label{sjb}
M=\sum_{j\in\zz_+}c_ja_j
\end{align}
both in $H_X(\rn)$ and pointwisely on $\rn$.
In addition, by an argument similar to that used in
the proof of \eqref{jgsj}, we conclude that
\begin{align*}
M(\cdot+z)=\sum_{j\in\zz_+}c_ja_j(\cdot+z)
\end{align*}
both in $H_X(\rn)$ and pointwisely on $\rn$.
Thus, using this, Lemma \ref{2t1}(i), \eqref{sjb},
the above assumption, and noticing that $g(\cdot+z)\in
\mathcal{L}_{X,q',s,d}(\rn)$ for any $g\in
\mathcal{L}_{X,q',s,d}(\rn)$, we find that, for any $g\in
\mathcal{L}_{X,q',s,d}(\rn)$,
\begin{align*}
\lf\langle L_g,M\r\rangle&=\sum_{j\in\zz_+}c_j
\lf\langle L_g,a_j\r\rangle
=\sum_{j\in\zz_+}c_j\int_{\rn}
g(x)a_j(x)\,dx\\
&=\sum_{j\in\zz_+}c_j\int_{\rn}
g(x+z)a_j(x+z)\,dx
=\sum_{j\in\zz_+}c_j
\lf\langle L_{g(\cdot+z)},a_j(\cdot+z)\r\rangle\\
&=\lf\langle L_{g(\cdot+z)},
\sum_{j\in\zz_+}c_ja_j(\cdot+z)\r\rangle
=\lf\langle L_{g(\cdot+z)},
M(\cdot+z)\r\rangle\\
&=\int_{\rn}g(x+z)M(x+z)dx
=\int_{\rn}g(x)M(x)dx,
\end{align*}
which further implies that \eqref{518-01} holds true for any $z\in\rn$.
Thus, in the remainder of this proof, we assume that $z=\mathbf{0}$.

Let $g\in\mathcal{L}_{X,q',s,d}(\rn)$, $L_0:=B(\mathbf{0},r)$, and
$L_j:=2^jB(\mathbf{0},r)\setminus 2^{j-1}B(\mathbf{0},r)$ for any $j\in\nn$.
Moreover, for any $\nu\in\zz_+^n$ with $|\nu|\leq s$, and for any $j\in\zz_+$,
let $\eta_{\nu}^{(j)}$, $\psi_{\nu}^{(j)}$, $\lambda_j$, $A_j$, $\wz \lambda_j$,  and
$\wz A_{\nu}^{(j)}$ be the same as in the proof of Proposition \ref{HK-mole}.
Then, using \eqref{M-Ml-JN**}, we have
\begin{align}\label{M-Ml-JN***}
M=\sum_{j\in\zz_+}\lambda_j A_j
+\sum_{j\in\zz_+}\sum_{\{\nu\in\zz_+^n:\ |\nu|\leq s\}}
\widetilde{\lambda}_{j}\widetilde{A}_\nu^{(j)}
\end{align}
both in $H_X(\rn)$ and pointwisely on $\rn$ and, using \eqref{3.19x}, we obtain,
for any $l\in\nn$,
\begin{align}\label{3.19xx}
M\mathbf{1}_{2^lB(\mathbf{0},r)}
&=\sum_{j=0}^{l}\lambda_j A_j+
\sum_{\{\nu\in\zz_+^n:\ |\nu|\leq s\}}
\sum_{j=0}^{l-1}
\widetilde{\lambda}_{j}\widetilde{A}_\nu^{(j)}\\
&\quad-\sum_{\{\nu\in\zz_+^n:\ |\nu|\leq s\}}
\frac{\eta_{\nu}^{(l)}\psi_{\nu}^{(l)}\mathbf{1}_{L_l}}{|L_l|}.\noz
\end{align}
Notice that, for any $\nu\in\zz_+^n$ with
$|\nu|\leq s$, and for any $j\in\zz_+$, both $A_j$
and $\wz A_\nu^{(j)}$ are $(X,q,s)$-atoms.
From this, Lemma \ref{2t1}, and \eqref{M-Ml-JN***},
we deduce that
\begin{align}\label{518-04}
\lf\langle L_g,M\r\rangle
&=\lf\langle L_g,\sum_{j\in\zz_+}\lambda_j A_j
+\sum_{j\in\zz_+}\sum_{\{\nu\in\zz_+^n:\ |\nu|\leq s\}}
\widetilde{\lambda}_{j}\widetilde{A}_\nu^{(j)}\r\rangle\\
&=\sum_{j\in\zz_+}\lf\langle L_g,\lambda_j A_j\r\rangle
+\sum_{j\in\zz_+}\sum_{\{\nu\in\zz_+^n:\ |\nu|\leq s\}}
\lf\langle L_g,\wz \lambda_{j}\widetilde{A}_\nu^{(j)}\r\rangle\noz\\
&=\sum_{j\in\zz_+}\int_{\rn}\lambda_j A_j(x)g(x)\,dx+
\sum_{j\in\zz_+}\sum_{\{\nu\in\zz_+^n:\ |\nu|\leq s\}}
\int_{\rn}\widetilde{\lambda}_{j}\widetilde{A}_\nu^{(j)}(x)g(x)\,dx.\noz
\end{align}
Thus, to finish the proof of \eqref{518-01}, it suffices to show that
\begin{align}\label{22-5-18-03}
\int_{\rn}M(x)g(x)\,dx&=\sum_{j\in\zz_+}\int_{\rn}\lambda_j A_j(x)g(x)\,dx\\
&\quad+\sum_{j\in\zz_+}\sum_{\{\nu\in\zz_+^n:\ |\nu|\leq s\}}
\int_{\rn}\widetilde{\lambda}_{j}\widetilde{A}_\nu^{(j)}(x)g(x)\,dx\noz.
\end{align}

Now, we prove \eqref{22-5-18-03}.
Indeed, by both Lemma \ref{lem-mol-01} and the Lebesgue dominated convergence
theorem, we find that
\begin{align}\label{22-5-18-02}
\int_{\rn}M(x)g(x)\,dx=\lim_{l\to\fz}\int_{2^{l}B(z,r)}M(x)g(x)\,dx.
\end{align}
Moreover, using \eqref{3.19xx}, we conclude that, for any $l\in\nn$,
\begin{align}\label{22-5-18-00}
&\int_{2^{l}B(\mathbf{0},r)}M(x)g(x)\,dx\\
&\quad=\sum_{j=0}^{l}\int_{\rn}\lambda_j A_j(x)g(x)\,dx+
\sum_{j=0}^{l-1}\sum_{\{\nu\in\zz_+^n:\ |\nu|\leq s\}}
\int_{\rn}\widetilde{\lambda}_{j}\widetilde{A}_\nu^{(j)}(x)g(x)\,dx\noz\\
&\qquad-\sum_{\{\nu\in\zz_+^n:\ |\nu|\leq s\}}\int_{L_l}
\frac{\eta_{\nu}^{(l)}\psi_{\nu}^{(l)}(x)\mathbf{1}_{L_l}(x)g(x)}{|L_l|}\,dx.\noz
\end{align}
To show \eqref{22-5-18-03},
we claim that
\begin{align}\label{22-5-18-01}
\lim_{l\to \fz}\sum_{\{\nu\in\zz_+^n:\ |\nu|\leq s\}}\int_{L_l}
\frac{\eta_{\nu}^{(l)}\psi_{\nu}^{(l)}(x)\mathbf{1}_{L_l}(x)g(x)}{|L_l|}\,dx=0.
\end{align}
From \eqref{L-7.2}, \eqref{eta-1}, the H\"older inequality, Lemma
\ref{lem-518-01}, \cite[Lemma 2.19]{jtyyz3}, Remark \ref{rem-ball-B2},
and \eqref{tau-01}, we deduce that, for any given $u\in(\frac{nq}{\tau q+n},p_-)$,
and for any $\nu\in\zz_+^n$ with $|\nu|\leq s$, and for any $l\in\zz_+$,
\begin{align*}\
&\lf|\int_{\rn}\frac{\eta_{\nu}^{(l)}\psi_{\nu}^{(l)}(x)
\mathbf{1}_{L_l}(x)}{|L_l|}g(x)\,dx\r|\\
&\quad\lesssim2^{-l(\tau-\frac{n}{q'})}
\fint_{2^lB(\mathbf{0},r)}\lf|g(x)\r|\,dx
\lesssim 2^{-l(\tau-\frac{n}{q'})}
\lf[\fint_{2^lB(\mathbf{0},r)}\lf|g(x)\r|^{q'}\,dx\r]^{\frac{1}{q'}}\\
&\quad\lesssim2^{-l(\tau-\frac{n}{q'})}
\lf\{\lf[\fint_{2^lB(\mathbf{0},r)}
\lf|g(x)-P_{B(\mathbf{0},r)}^{(s)}(g)(x)\r|^{q'}\,dx\r]^{\frac{1}{q'}}
+\lf\|P_{B(\mathbf{0},r)}^{(s)}(g)
\r\|_{L^{\fz}(2^lB(\mathbf{0},r))}\r\}\\
&\quad\ls2^{-l(\tau-\frac{n}{q'})}
l\lf[2^{ls}+2^{nl(\frac{1}{u}-1)}\r]
\|g\|_{\mathcal{L}_{X,q',s,d}(\rn)}
+2^{-l(\tau-\frac{n}{q'}-s)}
\lf\|P_{B(\mathbf{0},r)}^{(s)}(g)
\r\|_{L^{\fz}(B(\mathbf{0},r))}\\
&\quad \ls 2^{-l(\tau-\frac{n}{q'}-s)}l
\|g\|_{\mathcal{L}_{X,q',s,d}(\rn)}+2^{-l(\tau+\frac{n}{q}-\frac{n}{u})}l
\|g\|_{\mathcal{L}_{X,q',s,d}(\rn)}\\
&\qquad+2^{-l(\tau-\frac{n}{q'}-s)}
\fint_{B(\mathbf{0},r)}|g(x)|\,dx\\
&\quad\to 0
\end{align*}
as $l\to\fz$, where the implicit positive constants depend on $r$.
This implies that the above claim holds true.
Then, combining both \eqref{22-5-18-02} and \eqref{22-5-18-01}, and letting $l\to\fz$
in both sides of \eqref{22-5-18-00}, we conclude that \eqref{22-5-18-03} holds true.

All together, by \eqref{518-04} and \eqref{22-5-18-03},
we find that, for any $g\in \mathcal{L}_{X,q',s,d}(\rn)$,
\begin{align*}
\lf\langle L_g,M\r\rangle
&=\int_{\rn}M(x)g(x)\,dx.
\end{align*}
This finishes the proof of \eqref{518-01}, and hence of Lemma \ref{thm-dual-T}.
\end{proof}

The proof of the following lemma is a slight modification of the proof of \cite[Lemma 3.13]{cjy-01};
we omit the details here.

\begin{lemma}\label{lemma-atm}
Let $X$ be a ball quasi-Banach space and
$s\in[n-1,\infty)\cap\zz_+$.
Let $\beta\in(1,\infty)$, $\alpha\in(0,n)$, and $I_\alpha$ be the same as in
\eqref{cla-I}.
Assume that $p\in(1,\frac{n}{\alpha})$.
Then, for any $(X,p,s)$-atom $a$ supported in a ball $B\in\mathbb{B}(\rn)$, $I_\alpha(a)$
is an $(X^\beta,q,s-n+1,\tau)$-molecule,
centered at $B$, multiplied by a positive constant depending on $B$,
where $\frac{1}{q}:=\frac{1}{p}-\frac{\alpha}{n}$
and $\tau\in(0,n+s+1-\alpha-\frac{n}{q}]$.
\end{lemma}

To show Theorem \ref{thm-main-II}, we also need the following lemma.

\begin{lemma}\label{dual01}
Let all the symbols be the same as in Theorem \ref{thm-main-I}.
Further assume that $\widetilde{I}_{\alpha}$ is bounded from
$\mathcal{L}_{X^{\beta},1,s,{\beta d}}(\rn)$
to $\mathcal{L}_{X,1,s,d}(\rn)$, namely,
there exists a positive constant $C$ such that, for any
$f\in \mathcal{L}_{X^{\beta},1,s,{\beta d}}(\rn)$,
\begin{align*}
\lf\|\widetilde{I}_{\alpha}(f)
\r\|_{\mathcal{L}_{X,1,s,d}(\rn)}
\leq C \|f\|_{\mathcal{L}_{X^{\beta},1,s,{\beta d}}(\rn)}.
\end{align*}
Then, for any $g\in \mathcal{L}_{X^\beta,1,s,\beta d}(\rn)$ and any
$(X,\infty,s)$-atom $a$,
\begin{align*}
\lf\langle L_g,I_\alpha(a)\r\rangle=\lf\langle L_{\wz
I_\alpha(g)},a\r\rangle,
\end{align*}
where $L_g$ is the same as in \eqref{2te1}, $L_{\wz
I_\alpha(g)}$ as in \eqref{2te1} with $g$
replaced by $\wz I_\alpha(g)$,
and $\frac{1}{q}+\frac{1}{q'}=1$.
\end{lemma}

\begin{proof}
Let all the symbols be the same as in the present lemma.
Let $a$ be any given $(X,\infty,s)$-atom
supported in a ball $B\in\mathbb{B}(\rn)$
and let $p\in (1,\frac{n}{\alpha})\cap(\frac{1}{1/\max\{1,\beta p_0\}+\alpha/n},\frac{n}{\alpha})$. Then $a$ is also an $(X,p,s)$-atom.
Also, let $\frac{1}{q}:=\frac{1}{p}-\frac{\alpha}{n}$ and
\begin{align*}
\tau\in\lf(n\lf[\frac{1}{\min\{1,\beta
p_-\}}-\frac{1}{q}\r],\infty\r)
\cap\lf(\frac{n}{q'}+s-n+1,n+s+1-\alpha-\frac{n}{q}\r].
\end{align*}
Then, using the range of $p$, we find that $q\in(\max\{1,\beta p_0\},\fz)$.
By this, the range of $s$, and Lemma \ref{lemma-atm}, we find that
$I_\alpha(a)$ is an $(X^\beta,q,s-n+1,\tau)$-molecule centered, at the ball $B$,
multiplied by a positive constant depending on $B$.
Since $X^\beta$, $\beta p_-$, $\beta d$, $q$, $\tau$, and $s-n+1$
satisfy all the assumptions of Lemma \ref{thm-dual-T}, from Lemma \ref{thm-dual-T}
and the linearity of $L_g$,
we deduce that, for any $g\in \mathcal{L}_{X^\beta,q',s,\beta d}(\rn)$,
\begin{align}\label{dual-ta-01'}
\lf\langle L_g,
I_\alpha(a)\r\rangle=\int_{\rn}g(x)I_\alpha(a)(x)\,dx.
\end{align}
Using both the assumption that $\wz I_{\alpha}$
is bounded from $\mathcal{L}_{X^\beta,1,s,\beta d}(\rn)$
to $\mathcal{L}_{X,1,s,d}(\rn)$ and Remark \ref{sgjs}(i),
we find that, for any $g\in \mathcal{L}_{X^\beta,q',s,\beta d}(\rn)$,
$\wz I_\alpha(g)\in \mathcal{L}_{X,q',s,d}(\rn)$.
From this and Lemma \ref{2t1},
we infer that, for any $g\in \mathcal{L}_{X^\beta,q',s,\beta d}(\rn)$,
\begin{align}\label{dual-ta-03'}
\lf\langle L_{\wz I_\alpha(g)},a\r\rangle
=\int_{\rn}\wz I_\alpha(g)(x)a(x)\,dx.
\end{align}
Moreover, using Lemma \ref{lem-suit} with $\lambda:=s+1-\alpha$,
we conclude that, for any $g\in \mathcal{L}_{X^\beta,q',s,\beta
d}(\rn)$
and any  ball $B:=B(x_0,r_0)\in\mathbb{B}(\rn)$
with $x_0\in\rn$ and $r_0\in(0,\fz)$,
\begin{align}\label{dual-ta-011}
\int_{\rn\setminus B}\frac{|g(y)
-P_{B}^{(s)}(g)(y)|}{|x_0-y|^{n+s+1-\alpha}}\,dy<\fz.
\end{align}
By both \eqref{dual-ta-01'} and \eqref{dual-ta-011}, similarly to
the proof of
\cite[Lemma 3.10]{jtyyz2}, we find that, for any $g\in \mathcal{L}_{X^\beta,q',s,\beta d}(\rn)$,
\begin{align*}
\int_{\rn}g(x)I_\alpha(a)(x)\,dx=\int_{\rn}\wz
I_\alpha(g)(x)a(x)\,dx,
\end{align*}
which, together with both \eqref{dual-ta-01'} and \eqref{dual-ta-03'},
further implies that
\begin{align*}
\lf\langle L_g, I_\alpha(a)\r\rangle=\int_{\rn}I_\alpha(a)g(x)\,dx
=\int_{\rn}\wz I_\alpha(g)(x)a(x)\,dx=\lf\langle
L_{\wz I_\alpha(g)},a\r\rangle.
\end{align*}
Both this and Remark \ref{sgjs}(i) then finish the proof of Lemma \ref{dual01}.
\end{proof}

We also need the following lemma which plays an important role in the proof of
Theorem \ref{thm-main-II}.
In what follows, we use
$(\mathcal{L}_{X,q,s,d}(\rn))^*$ to
denote the dual space of $\mathcal{L}_{X,q,s,d}(\rn)$,
namely, the set of all continuous
linear functionals on
$\mathcal{L}_{X,q,s,d}(\rn)$ equipped with the weak-$*$ topology.

\begin{lemma}\label{lem-adj-01}
Let $X$, $q$, $s$, and $d$ be the same as in Definition \ref{finatom},
and let $f\in H_X(\rn)$. For any $g\in \mathcal{L}_{X,q',s,d}(\rn)$, define
\begin{align}\label{dual-func}
\langle f, g\rangle_*:=\lf\langle L_g, f\r\rangle,
\end{align}
where $L_g$ is the same as in \eqref{2te1}.
Then $\langle f, \cdot\rangle_*\in (\mathcal{L}_{X,q',s,d}(\rn))^*$ and
\begin{align}\label{3.62x}
\|\langle f, \cdot\rangle_*\|_{(\mathcal{L}_{X,q',s,d}(\rn))^*}\sim \|f\|_{H_{X}(\rn)}
\end{align}
with the positive equivalence constants independent of $f$.
\end{lemma}

\begin{proof}
Let all the symbols be the same as in the present lemma.
It is easy to show that
$\langle f, \cdot\rangle_*$ is a linear operator on $\mathcal{L}_{X,q',s,d}(\rn)$.
Moreover, from both \eqref{dual-func} and Lemma \ref{2t1}, we deduce that, for any $g\in
\mathcal{L}_{X,q',s,d}(\rn)$,
\begin{align*}
|\langle f, g\rangle_*|	
=\lf|\lf\langle L_g, f\r\rangle\r|
\ls\lf\|L_{g}\r\|_{(H_{X}(\rn))^{*}}\|f\|_{H_{X}(\rn)}
\sim\|g\|_{\mathcal{L}_{X,q',s,d}(\rn)}\|f\|_{H_{X}(\rn)}.
\end{align*}
Thus, $\langle f, \cdot\rangle_*\in (\mathcal{L}_{X,q',s,d}(\rn))^*$.

Next, we show \eqref{3.62x}. To this end, let
$$\Phi:\ H_X(\rn)\longrightarrow \lf(H_X(\rn)\r)^{**}$$
denote the canonical embedding.
Then, from the definition of $\Phi$, Lemma \ref{2t1}, and \eqref{dual-func}, we deduce that,
for any $f\in H_X(\rn)$,
\begin{align*}
\|f\|_{H_X(\rn)}
&=\|\Phi \lf(f\r)\|_{(H_X(\rn))^{**}}=\sup_{F\in (H_X(\rn))^*,\ \|F\|_{(H_X(\rn))^*=1}}
|\langle \Phi \lf(f\r),F\rangle|\\
&=\sup_{F\in (H_X(\rn))^*,\ \|F\|_{(H_X(\rn))^*=1}}
|\langle F,f\rangle|\sim\sup_{g\in \mathcal{L}_{X,q',s,d}(\rn),\
\|g\|_{\mathcal{L}_{X,q',s,d}(\rn)=1}}
|\langle L_g,f\rangle|\\
&\sim\sup_{g\in \mathcal{L}_{X,q',s,d}(\rn),\ \|g\|_{\mathcal{L}_{X,q',s,d}(\rn)=1}}
|\langle f,g\rangle_*|\sim\|\langle f,\cdot\rangle_*\|_{(\mathcal{L}_{X,q',s,d}(\rn))^*}.
\end{align*}
Thus, \eqref{3.62x} holds true,
which then completes the proof of Lemma \ref{lem-adj-01}.
\end{proof}	

Now, we prove Theorem \ref{thm-main-II}.

\begin{proof}[Proof of Theorem \ref{thm-main-II}]
Let all the symbols be the same as in the present theorem.
From the assumption that $\wz I_{\alpha}$
is bounded from $\mathcal{L}_{X^{\beta},1,s,d}(\rn)$
to $\mathcal{L}_{X,1,s,d}(\rn)$, we deduce that,
for any $g\in \mathcal{L}_{X^{\beta},1,s,d}(\rn)$,
$\wz I_{\alpha}(g)\in \mathcal{L}_{X,1,s,d}(\rn)$.
By this, \eqref{dual-func}, and Lemmas \ref{dual01} and \ref{2t1},
we conclude that, for any $(X,\fz,s)$-atom $a$
and any $g\in \mathcal{L}_{X^{\beta},1,s,d}(\rn)$,
\begin{align*}
|\langle I_{\alpha}(a),g\rangle_*|
&=|\langle L_g,I_{\alpha}(a)\rangle|
=\lf|\lf\langle L_{\wz I_{\alpha}(g)},a\r\rangle\r|
\ls\lf\|L_{\wz I_{\alpha}(g)}\r\|_{(H_{X}(\rn))^{*}}\|a\|_{H_{X}(\rn)}\\
&\sim\lf\|\wz I_{\alpha}(g)\r\|_{\mathcal{L}_{X,1,s,d}(\rn)}\|a\|_{H_{X}(\rn)}
\ls \|g\|_{\mathcal{L}_{X^{\beta},1,s,d}(\rn)}\|a\|_{H_{X}(\rn)},
\end{align*}
which, together with Lemma \ref{lem-adj-01}, further implies that,
for any $(X,\fz,s)$-atom $a$,
\begin{align*}
\|I_{\alpha}(a)\|_{H_{X^{\beta}}(\rn)}
\sim \lf\|\lf\langle I_{\alpha}(a),\cdot\r\rangle_*\r\|_{(\mathcal{L}_{X^{\beta},1,s,d}(\rn))^*}
\ls \|a\|_{H_{X}(\rn)}.
\end{align*}
Using this and repeating the proof of the necessity of \cite[Theorem 3.7]{cjy-01},
we infer that, for any ball $B\in\mathbb{B}(\rn)$,
\begin{align*}
|B|^{\frac{\alpha}{n}}\ls\|\mathbf{1}_B\|_X^{\frac{\beta-1}{\beta}}.
\end{align*}
This finishes the proof Theorem \ref{thm-main-II}.
\end{proof}

\subsection{Relations Between $I_{\alpha}$ and $\wz I_{\alpha}$\label{sec-I-2}}

In this subsection, we prove that $\wz{I}_\alpha$ is
just the adjoint operator of $I_\alpha$. To be precise, we have the following theorem.

\begin{theorem}\label{dual-I-wI}
Let both $X$ and $p_-$ satisfy Assumption \ref{assump1}.
Assume that $X$, $r_0\in(0,\min\{\frac{1}{\beta},p_-\})$
with $\beta\in(1,\fz)$, and $p_0\in(r_0,\infty)$
satisfy Assumption \ref{assump2}.
Let $d:=r_0$, $\alpha\in(0,1)$, integer
\begin{align*}
s\geq\max\lf\{\lf\lfloor\frac{n}
{\min\{1,p_-\}}-n\r\rfloor,
\lf\lfloor\frac{n}{\min\{1,\beta p_-\}}-1\r\rfloor,
\lf\lfloor\frac{n}{\beta\min\{1,d\}}-n+\alpha\r\rfloor\r\},
\end{align*}
$I_\alpha$ be the same
as in \eqref{cla-I}, and $\wz{I}_\alpha$ the same as in
Remark \ref{rem-I-B}(i).
Further assume that $X$ has an absolutely continuous quasi-norm.	
If there exists a positive constant $C$ such that,
for any ball $B\in\mathbb{B}(\rn)$,
\begin{align}\label{assum-03}
|B|^{\frac{\alpha}{n}}\leq C
\|\mathbf{1}_B\|_X^{\frac{\beta-1}{\beta}},
\end{align}
then, for any $g\in \mathcal{L}_{X^\beta,1,s,\beta d}(\rn)$
and $f\in H_X(\rn)$,
\begin{align*}
\lf\langle L_g,I_\alpha(f)\r\rangle=\lf\langle L_{\wz
I_\alpha(g)},f\r\rangle,
\end{align*}
where $L_g$ is the same as in \eqref{2te1}, $L_{\wz
I_\alpha(g)}$ as in \eqref{2te1} with $g$
replaced by $\wz I_\alpha(g)$,
and $\frac{1}{q}+\frac{1}{q'}=1$.
\end{theorem}

To this end, we first state the well-known semi-group property of
$I_{\alpha}$, which is just
\cite[Theorem F(c)]{sw1960}.
In what follows,
the \emph{Hardy space}
$H^{p}(\rn)$ is defined as in Definition
\ref{2d1} with $X:=L^{p}(\rn)$.

\begin{lemma}\label{Reclassic}
Let $\alpha\in(0,n)$ and $I_{\alpha}$ be the same as in \eqref{cla-I}.
Then $I_{\alpha}$ has the \emph{semigroup property}, namely,
for any $\alpha_1,\alpha_2\in(0,n)$ satisfying
$\alpha_1+\alpha_2\in(0,n)$, and for any $f\in H^p(\rn)$ with
$p\in[\frac{n-1}{n},\frac{n}{\alpha_1+\alpha_2})$,
there exists a positive constant $C$, independent of $f$, such that
$$
I_{\alpha_1+\alpha_2}(f)= C I_{\alpha_1}(I_{\alpha_2}(f)).
$$
\end{lemma}

Now, we show Theorem \ref{dual-I-wI}.

\begin{proof}[Proof of Theorem \ref{dual-I-wI}]
Let all the symbols be the same as in the present theorem.
We first prove the case $\alpha\in(0,1)$.
To this end, let $f\in H_X(\rn)$.
Since $X$ has an absolutely continuous quasi-norm, from this and
\cite[Remark 3.12]{SHYY}, we deduce that
$H_{\mathrm{fin}}^{X,\fz,s,d}(\rn)$ is dense in $H_X(\rn)$.
Thus, there exists a sequence
$\{f_k\}_{k\in\nn}\subset H_{\mathrm{fin}}^{X,\fz,s,d}(\rn)$
such that $f=\lim_{k\to\infty}f_k$ in $H_X(\rn)$.
Besides, by the definition of
$H_{\mathrm{fin}}^{X,\fz,s,d}(\rn)$, we conclude that, for any
$k\in\nn$,
there exists a sequence $\{a_j^{(k)}\}_{j=1}^{m_k}$ of $(X,\fz,s)$-atoms 	
and $\{\lambda_j^{(k)}\}_{j=1}^{m_k}\subset[0,\fz)$ such that
\begin{align}\label{Ig-01}
f_k=\sum_{j=1}^{m_k}\lambda_j^{(k)}a_j^{(k)}.
\end{align}
From Lemma \ref{thm-Ia-02}, we infer that $I_\alpha$
is bounded from $H_X(\rn)$ to $H_{X^\beta}(\rn)$, which, together
with \eqref{Ig-01}, further implies that
$I_\alpha(f)=\lim_{k\to\infty} I_\alpha(f_k)$ in $H_{X^\beta}(\rn)$.
Therefore, combining this, Theorem \ref{main-corollary},
Lemma \ref{2t1}, \eqref{Ig-01}, \eqref{assum-03}, and Lemma
\ref{dual01}, we conclude that, for any $g\in \mathcal{L}_{X^\beta,1,s,\beta d}(\rn)$,
\begin{align*}
\lf\langle L_g,I_\alpha(f)\r\rangle
&=\lim_{k\to\infty}\lf\langle L_g,I_\alpha(f_k)\r\rangle
=\lim_{k\to\infty}\sum_{j=1}^{m_k}\lambda_j^{(k)}
\lf\langle L_g,I_\alpha(a_j^{(k)})\r\rangle\\
&=\lim_{k\to\infty}\sum_{j=1}^{m_k}\lambda_j^{(k)}
\lf\langle L_{\wz {I_\alpha}(g)},a_j^{(k)}\r\rangle
=\lf\langle L_{\wz {I_\alpha}(g)},f\r\rangle.
\end{align*}
This finishes the proof of the case $\alpha\in(0,1)$.

Now, we show the case $\alpha\in[1,n)$. In this case, we have $\frac{\alpha}{n}\in(0,1)$.
In addition, for any $i\in\{1,\ldots,n\}$, let
$$\beta_i:=\frac{(n-i+1)\beta+i-1}{(n-i)\beta+i}.
$$
It is easy to show
$\beta=\Pi_{i=1}^{n}\beta_i$. Moreover, by
\eqref{assum-03}, we conclude that
\begin{align}\label{betai}
|B|^{\frac{\alpha}{n^2}}\lesssim\|\mathbf{1}_B
\|_{X^{\beta_1\cdots\beta_{i-1}}}
^{\frac{\beta_i-1}{\beta_i}}.
\end{align}
On one hand, using \eqref{betai} and the proof of Theorem \ref{main-corollary},
we find that, for any $i\in\{1,\ldots,n\}$,
$\widetilde{I}_{\frac{\alpha}{n}}$ is bounded from
$\mathcal{L}_{X^{\beta_1\cdots\beta_{i}},1,s,\beta_1\cdots\beta_{i}
d}(\rn)$
to
$\mathcal{L}_{X^{\beta_1\cdots\beta_{i-1}},1,s,\beta_1
\cdots\beta_{i-1}d}(\rn)$,
which further implies that $(\wz
I_\frac{\alpha}{n})^{n+1-i}(g)\in\mathcal{L}_{X^{\beta_1
\cdots\beta_{i-1}}
,1,s,\beta_1\cdots\beta_{i-1}d}(\rn)$.
On the other hand, by \eqref{betai} and Lemma \ref{thm-Ia-02},
we conclude that, for any $i\in\{1,\ldots,n\}$,
$I_{\frac{\alpha}{n}}$ is bounded from $H_{X^{\beta_1\cdots\beta_{i-1}}}(\rn)$ to
$H_{X^{\beta_1\cdots\beta_{i}}}(\rn)$,
which further implies that $(I_{\frac{\alpha}{n}})^{i}(f)\in
H_{X^{\beta_1\cdots\beta_{i}}}(\rn)$.
From Lemma \ref{Reclassic}, we deduce that
$I_\alpha(f)=(I_{\frac{\alpha}{n}})^n(f)$ in $\mathcal{S}'(\rn)$. Combining this and
$I_\alpha(f)\in H_{X^\beta}(\rn)$, we conclude that
$I_\alpha(f)=(I_{\frac{\alpha}{n}})^n(f)$ in $H_{X^\beta}(\rn)$.
By this, Theorem \ref{dual-I-wI}, and \eqref{de-wsIn},
we find that there exists a positive constant $C$, independent of both $f$ and $g$,
such that, for any $\alpha\in(0,n)$,
\begin{align*}
\lf\langle L_g,I_\alpha(f)\r\rangle
&=C\lf\langle L_g,\lf(I_{\frac{\alpha}{n}}\r)^{n}(f)\r\rangle
=C\lf\langle L_{\wz
I_\frac{\alpha}{n}(g)},\lf(I_{\frac{\alpha}{n}}\r)^{n-1}(f)\r\rangle\\
&=\cdots=C\lf\langle L_{(\wz I_\frac{\alpha}{n})^{n}(g)},f\r\rangle
=C\lf\langle L_{\wz I_\alpha(g)},f\r\rangle.
\end{align*}
This finishes the proof of Theorem \ref{dual01}.
\end{proof}

\section{Applications\label{Appli}}

The aim of this section is to apply all the above main results to
four concrete examples of ball quasi-Banach
function spaces, namely, Morrey spaces
(Subsection \ref{Morrey}), mixed-norm Lebesgue spaces (Subsection \ref{Mixed-Norm}),
local generalized Herz spaces (Subsection \ref{H-H}),
and mixed Herz spaces (Subsection \ref{M-H}).
Moreover, to the best of our knowledge, all these results are new.
These examples reveal that more applications of
the main results of this article to other function
spaces are predictable.

\subsection{Morrey Spaces\label{Morrey}}

Recall that the Morrey space
$M_{r}^{p}(\mathbb{R}^{n})$ with $0<r \leq p<\infty$
was introduced by Morrey \cite{MCB1938}
in order to study the regularity
of solutions to certain equations.
The Morrey space has many applications in the theory of elliptic partial
differential equations, potential theory, and harmonic analysis
(see, for instance, \cite{cf1987,sdh2020,tyy-2019}).

\begin{definition}\label{Def-Morrey}
Let $0<r \leq p<\infty$. The \emph{Morrey space}
$M_{r}^{p}(\mathbb{R}^{n})$ is
defined to be the set of all the measurable
functions $f$ on $\rn$ such that
$$
\|f\|_{M_{r}^{p}(\mathbb{R}^{n})}:=\sup
_{B \in \mathbb{B}(\rn)}|B|^{\frac{1}{p}-\frac{1}{r}}\|f\|_{L^{r}(B)}<\infty.
$$
\end{definition}

\begin{remark}\label{Rem-Morrey}
Let $0<r \leq p<\infty$. It is easy to show that
$M_{r}^{p}(\rn)$ is a ball quasi-Banach function space.
However, as was pointed out in \cite[p.\,87]{SHYY}, $M_{r}^{p}(\rn)$ may not be a
quasi-Banach function space.
\end{remark}

The following theorem is a corollary of both
Theorems \ref{main-Theorem2} and \ref{main-corollary}.

\begin{theorem}\label{ap-M}
Let $0<r \leq p<\frac{n}{\alpha}$. Then both Theorems
\ref{main-Theorem2} and \ref{main-corollary}
with $X:=M_{r}^{p}(\rn)$, $p_-:=r$, and $\beta:=\frac{n}{n-\alpha p}$ hold true.
\end{theorem}

\begin{proof}
Let all the symbols be the same as in the present
theorem. It has been point out in Remark \ref{Rem-Morrey}
that $M^p_r(\rn)$ is a ball quasi-Banach space.
In addition, by \cite[Lemma 2.5]{tx2005} (see also \cite[Lemma 7.2]{zyyw}),
we conclude that
$M_r^p(\rn)$ satisfies Assumption
\ref{assump1} with both $X:=M_r^p(\rn)$ and $p_-:=r$.
Moreover, it is easy to prove that, for any
$B\in\mathbb{B}(\rn)$,
$$
\|\mathbf{1}_B\|_{M_r^p(\rn)}=|B|^{\frac{1}{p}-\frac{1}{r}}
\|\mathbf{1}_B\|_{L^r(B)}=|B|^{\frac{1}{p}},
$$
which further implies that
$|B|^{\frac{\alpha}{n}}\lesssim\|\mathbf{1}_B\|_
{M^{p}_{r}(\rn)}^{\frac{\beta-1}{\beta}}$
if and only if
$\beta=\frac{n}{n-\alpha p}$.
Thus, all the assumptions of both
Theorems \ref{main-Theorem2} and \ref{main-corollary} with $X:=M_r^p(\rn)$ are satisfied.
Then, using both Theorems \ref{main-Theorem2} and \ref{main-corollary}
with $X:=M_r^p(\rn)$, we obtain the desired conclusions,
which completes the proof of
Theorem \ref{ap-M}.
\end{proof}

\begin{remark}
\begin{enumerate}
\item[\rm (i)]	
Let $0<r\leq p<\infty$.
Theorem \ref{thm-main-I} can not be applied
to the Morrey space $M^{p}_{r}(\rn)$ because
$M^{p}_{r}(\rn)$ does not have the absolutely continuous quasi-norm unless $r=p$,
namely, $M^{p}_{r}(\rn)=L^r(\rn)$.

\item[$\mathrm{(ii)}$] Let $\alpha\in(0,1)$, $I_{\alpha}$ be the
same as in \eqref{cla-I},
and $\wz I_{\alpha}$ the same as in Remark \ref{rem-I-B}(i) with $s=0$ therein.
It is well known that, if $1<r\leq p<\frac{n}{\alpha}$, then
$I_{\alpha}$ is bounded from $M_r^p(\rn)$ to $M_{rn/(n-\alpha p)}^{pn/(n-\alpha p)}(\rn)$ (see, for instance, \cite[Theorem 3.1]{ADR1975}).
As an application of Theorem \ref{main-corollary}(i)
with $X:=M^1_1(\rn)=L^1(\rn)$ (This is indeed a special case of Theorem \ref{ap-M}),
we can give a corresponding result of \cite[Theorem 3.1]{ADR1975}
at the critical case $p:=\frac{n}{\alpha}$, namely, if $\alpha\in(0,1)$
and $r\in [1,\fz)$, then $\wz I_{\alpha}$ is bounded from
$M_r^{\frac{n}{\alpha}}(\rn)$ to $\mathop{\mathrm{BMO}\,}(\rn).$
Indeed, in this case,
$M_{rn/(n-\alpha p)}^{pn/(n-\alpha p)}(\rn)=L^{\fz}(\rn)$ and hence its
reasonable replacement should be the well-known \emph{space
$\mathop{\mathrm{BMO}\,}(\rn)$}
which is defined to be the set of all the $f\in L_{\mathrm{loc}}^1(\rn)$ such that
$$
\|f\|_{\mathop{\mathrm{BMO}\,}(\rn)}:=\sup_{B\in\mathbb{B}(\rn)}
\fint_B\lf|f(x)-f_B\r|\,dx<\fz
$$
with $f_B$ as in \eqref{fe} via replacing $E$ by $B$.
Observe that, when $r\in[1,\infty)$ and
$\alpha\in(0,n)$, by the definitions of both
$M_r^{\frac{n}{\alpha}}(\rn)$ and
$\mathcal{L}_{L^{\frac{n}{n-\alpha}}(\rn),r,0}(\rn)$,
we easily obtain
$$M_r^{\frac{n}{\alpha}}(\rn)\subset
\mathcal{L}_{L^{\frac{n}{n-\alpha}}(\rn),r,0}(\rn).$$
Using this and
Theorem \ref{main-corollary}(i) with $X:=L^1(\rn)$, $q:=r\in[1,\fz)$,
$\alpha\in(0,1)$, and $\beta:=\frac{n}{n-\alpha}$, we further
conclude that, for any
$f\in M_r^{\frac{n}{\alpha}}(\rn)\subset\mathcal{L}_{L^{\frac{n}{n-\alpha}}(\rn),r,0}(\rn)$,
\begin{align*}
\lf\|\widetilde{I}_{\alpha}(f)\r\|_{\mathrm{BMO}\,(\rn)}
\sim\lf\|\widetilde{I}_{\alpha}(f)\r\|_{\mathcal{L}_{L^{1}(\rn),r,0}(\rn)}
\ls\|f\|_{\mathcal{L}_{L^{\frac{n}{n-\alpha}}(\rn),r,0}(\rn)}
\ls \|f\|_{M_r^{\frac{n}{\alpha}}(\rn)},
\end{align*}
where, in the first step, we used the well-known result that
$\mathrm{BMO}\,(\rn)=\mathcal{L}_{L^1(\rn),r,0}(\rn)$ for any $r\in[1,\fz)$
(see, for instance, \cite[p.\,125, Corollary 6.12]{Duo01}).
Thus, $\wz I_{\alpha}$ is bounded from
$M_r^{\frac{n}{\alpha}}(\rn)$ to $\mathop{\mathrm{BMO}\,}(\rn)$,
which completes the proof of the above claim. This critical case is
of independent interest, which might be known; however,
we did not find the exact reference.

\item[\rm (iii)]	
To the best of our knowledge, Theorem \ref{ap-M} is totally new.

\item[\rm (iv)]
It is worth pointing out that there exist
many studies about fractional integrals on Morrey-type
space and we refer the reader to
\cite{DGNSS,HSS2016,ho2017,SS2017,SST2009,Ho2021}.
\end{enumerate}
\end{remark}

\subsection{Mixed-Norm Lebesgue Spaces\label{Mixed-Norm}}

The mixed-norm Lebesgue space
$L^{\vec{p}}(\mathbb{R}^{n})$
was studied by Benedek and Panzone
\cite{BAP1961} in 1961, which can be
traced back to H\"ormander \cite{HL1960}.
We refer the reader to
\cite{CGN2017,CGG2017,CGN20172,GN2016,HLY2019,HLYY2019,HYacc}
for more studies on mixed-norm type spaces.

\begin{definition}\label{mixed}
Let $\vec{p}:=(p_{1}, \ldots, p_{n})
\in(0, \infty]^{n}$. The \emph{mixed-norm
Lebesgue space}
$L^{\vec{p}}(\mathbb{R}^{n})$
is defined to be the set of all the measurable
functions $f$ on $\rn$ such that
\begin{align*}
\|f\|_{L^{\vec{p}}(\mathbb{R}^{n})}:=\left\{\int_{\mathbb{R}}
\cdots\left[\int_{\mathbb{R}}\left|f(x_{1}, \ldots,
x_{n})\right|^{p_{1}} \,d x_{1}\right]^{\frac{p_{2}}{p_{1}}}
\cdots \,dx_{n}\right\}^{\frac{1}{p_{n}}}<\infty
\end{align*}
with the usual modifications made when $p_i=\infty$ for some $i\in\{1,\ldots,n\}$.
Moreover, let
\begin{align}\label{p-p+}
p_{-}:=\min\{p_{1},\ldots, p_{n}\}\text{ and }p_{+}:=\max\{p_{1},\ldots, p_{n}\}.
\end{align}
\end{definition}

\begin{remark}\label{mix-r}
Let $\vec{p}\in(0, \infty)^{n}$.	
From Definition \ref{mixed}, we easily deduce that
$L^{\vec{p}}(\mathbb{R}^{n})$
is a ball quasi-Banach space.
However, as was pointed in \cite[Remark 7.21]{zyyw}, $L^{\vec{p}}(\rn)$
may not be a quasi-Banach function space.
\end{remark}

The following theorem is a corollary of Theorems \ref{thm-main-I},
\ref{main-Theorem2}, \ref{main-corollary}, and \ref{dual-I-wI}.

\begin{theorem}\label{apply2}
Let $\vec{p}:=(p_{1}, \ldots, p_{n})
\in(0, \infty)^{n}$ satisfy
$\sum_{i=1}^{n}\frac{1}{p_i}\in(\alpha,\infty)$ and let both
$p_-$ and $p_+$ be the same as in \eqref{p-p+}.
\begin{enumerate}
\item[\rm (i)]
Theorems \ref{thm-main-I} with
$X:=L^{\vec{p}}(\rn)$ holds true if and only if $\beta:=\frac{\sum_{i=1}^{n}\frac{1}{p_i}}
{\sum_{i=1}^{n}\frac{1}{p_i}-\alpha}$;
	
\item[\rm (ii)]
Theorems \ref{main-Theorem2}, \ref{main-corollary}, and \ref{dual-I-wI} with
both $X:=L^{\vec{p}}(\rn)$ and $\beta:=\frac{\sum_{i=1}^{n}\frac{1}{p_i}}
{\sum_{i=1}^{n}\frac{1}{p_i}-\alpha}$ hold true.
\end{enumerate}
\end{theorem}

\begin{proof}
Let all the symbols be the same as in the
present theorem.
From Definition \ref{mixed}, we easily deduce that
$L^{\vec{p}}(\mathbb{R}^{n})$
is a ball quasi-Banach space with an absolutely continuous quasi-norm.
In addition, by \cite[Lemma 3.7]{HLY2019}, we conclude that
$L^{\vec{p}}(\rn)$ satisfies Assumption
\ref{assump1} with both $X:=L^{\vec{p}}(\rn)$
and $p_-$ the same as in Definition \ref{mixed}.
In addition, using both the dual theorem of $L^q(\rn)$
(see \cite[p.\,304, Theorem 1.a]{BAP1961}) and \cite[Lemma 3.5]{HLY2019}
(see also \cite[Lemma 4.3]{HYacc}), we conclude
that Assumption \ref{assump2} also holds true with $X:=L^{\vec{p}}(\rn)$,
$r_0\in(0,p_-)$, and $p_0\in(p_+,\fz)$.
Moreover, from the definition of $\|\cdot\|_{L^{\vec{p}}(\mathbb{R}^{n})}$,
we easily deduce that, for any
$B(x,r)\in\mathbb{B}(\rn)$ with $x\in\rn$
and $r\in(0,\infty)$,
\begin{align*}
\left\|\mathbf{1}_{B(x,r)}\right\|_{L^{\vec{p}}(\mathbb{R}^{n})}
=\left\|\mathbf{1}_{B(\mathbf{0},r)}\right\|_{L^{\vec{p}}(\mathbb{R}^{n})}
=r^{\sum_{i=1}^{n}\frac{1}{p_i}}
\left\|\mathbf{1}_{B(\mathbf{0},1)}\right\|
_{L^{\vec{p}}(\mathbb{R}^{n})},
\end{align*}
which further implies that
$|B|^{\frac{\alpha}{n}}\lesssim\|\mathbf{1}_{B}\|
_{L^{\vec{p}}(\mathbb{R}^{n})}^{\frac{\beta-1}{\beta}}$
if and only if $\beta=\frac{\sum_{i=1}^{n}\frac{1}{p_i}}
{\sum_{i=1}^{n}\frac{1}{p_i}-\alpha}$.
Thus, all the assumptions of Theorems \ref{thm-main-I}, \ref{main-Theorem2},
\ref{main-corollary}, and \ref{dual-I-wI} with $X:=L^{\vec{p}}(\rn)$ are satisfied.
Then, using Theorems \ref{thm-main-I},
\ref{main-Theorem2}, \ref{main-corollary}, and \ref{dual-I-wI} with $X:=L^{\vec{p}}(\rn)$,
we obtain the desired conclusions,
which complete the proof
of Theorem \ref{apply2}.
\end{proof}

\begin{remark}
\begin{enumerate}
\item[\rm (i)]	
Let $\alpha\in(0,1)$ and $\widetilde{I}_{\alpha}$ be the same as
in Remark \ref{rem-I-B}(i) with $s=0$ therein.	
Let $p\in[1,\fz)$, $\vec{p}:=(p_1,\ldots,p_n)\in (1,\fz)^n$,
and $p_-$ be the same as in \eqref{p-p+}.
Recall that the \emph{mixed Morrey space}
$\mathcal{M}_{\vec{p}}^{p}(\rn)$
was originally introduced by Nogayama \cite{No2019}
and further studied in \cite{No2019-2,noss2021},
which is defined to be the set
of all the measurable functions $f$ on $\rn$ such that
\begin{align*}
\|f\|_{\mathcal{M}_{\vec{p}}^{p}(\rn)}
:=\sup_{B\in\mathbb{B}(\rn)}|B|^{\frac{1}{p}-\sum_{i=1}^n\frac{1}{p_i}}
\|f\mathbf{1}_{B}\|_{L^{\vec{p}}(\rn)}<\fz.
\end{align*}
Moreover, in \cite[Theorem 1]{ws2022}, Wei and Sun proved that $\wz I_{\alpha}$
is bounded from $\mathcal{M}_{\vec{p}}^{\frac{n}{\alpha}}(\rn)$
to $\mathrm{BMO}\,(\rn)$
when $\alpha\in(0,\sum_{i=1}^n\frac{1}{p_i})$;
in \cite[Theorem 2]{ws2022}, Wei and Sun proved that $\wz I_{\alpha}$
is bounded from $\mathcal{M}_{\vec{p}}^{p}(\rn)$
to $\mathcal{L}_{L^{p'n/(p'\alpha+n)}(\rn),1,0}(\rn)$
when $\alpha\in(\frac{n}{p},1+\frac{n}{p})\cap(0,\sum_{i=1}^n\frac{1}{p_i})$,
where $\frac{1}{p}+\frac{1}{p'}=1$.
We now claim that \cite[Theorem 1]{ws2022} in the case when
$\alpha\in(0,1)\cap (0,\sum_{i=1}^n\frac{1}{p_i})$
and \cite[Theorem 2]{ws2022} in the case when $\alpha\in(\frac{n}{p},1)\cap(0,\sum_{i=1}^n\frac{1}{p_i})$
can be deduced
from Theorem \ref{main-corollary}(i)
with $X:=L^1(\rn)$ or $X:=L^{\frac{p'n}{p'\alpha+n}}(\rn)$
and $p\in(n,\infty)$ (This is indeed a special case of Theorem \ref{apply2}).
To prove this claim, we first show that
$\mathcal{M}_{\vec{p}}^{p}(\rn)\subset \mathcal{L}_{L^{p'}(\rn),1,0}(\rn)$ and,
for any $f\in \mathcal{M}_{\vec{p}}^{p}(\rn)\subset L_{\mathrm{loc}}^1(\rn)$,
\begin{align}\label{4.22x}
\|f\|_{\mathcal{L}_{L^{p'}(\rn),1,0}(\rn)}	
\ls\|f\|_{\mathcal{M}_{\vec{p}}^{p}(\rn)}.
\end{align}
Indeed, by both the H\"older inequality and the
definition of $\|\cdot\|_{L^{\vec{p}}(\rn)}$, we conclude that,
for any $f\in L_{\mathrm{loc}}^1(\rn)$ and any ball $B:=B(x_B,r_B)\in\mathbb{B}(\rn)$
with $x_B\in\rn$ and $r_B\in(0,\fz)$,
\begin{align*}
&\frac{|B|}{\|\mathbf{1}_B\|_{L^{p'}(\rn)}}\fint_{B}\lf|f(x)-P_{B}^{(s)}(f)(x)\r|\,dx\\	
&\quad\ls|B|^{\frac{1}{p}}\lf[\fint_{B}
\lf|f(x)-P_{B}^{(s)}(f)(x)\r|^{p_-}\,dx\r]^{\frac{1}{p_-}}
\ls|B|^{\frac{1}{p}}\lf[\fint_{B}|f(x)|^{p_-}\,dx\r]^{\frac{1}{p_-}}\\
&\quad\ls|B|^{\frac{1}{p}}\left\{\fint_{(x_n-r_B,x_n+r_B)}
\cdots\left[\fint_{(x_1-r_B,x_1+r_B)}\left|f(x_{1}, \ldots,
x_{n})\right|^{p_{-}} \,d x_{1}\right]^{\frac{p_{-}}{p_{-}}}
\cdots \,dx_{n}\right\}^{\frac{1}{p_{-}}}\\	
&\quad\ls|B|^{\frac{1}{p}}\left\{\fint_{(x_n-r_B,x_n+r_B)}
\cdots\left[\fint_{(x_1-r_B,x_1+r_B)}\left|f(x_{1}, \ldots,
x_{n})\right|^{p_{1}} \,d x_{1}\right]^{\frac{p_{2}}{p_{1}}}
\cdots \,dx_{n}\right\}^{\frac{1}{p_{n}}}\\
&\quad\ls \lf|\sqrt nB\r|^{\frac{1}{p}-\frac{1}{n}\sum_{i=1}^n\frac{1}{p_i}}
\lf\|f\mathbf{1}_{\sqrt nB}\r\|_{L^{\vec{p}}(\rn)}.
\end{align*}
This implies that \eqref{4.22x} holds true and hence
$\mathcal{M}_{\vec{p}}^{p}(\rn)\subset \mathcal{L}_{L^{p'}(\rn),1,0}(\rn)$.
In addition, assuming $\alpha\in(0,1)
\cap(0,\sum_{i=1}^n\frac{1}{p_i})$ and using
Theorem \ref{main-corollary}(i) with $X:=L^1(\rn)$, $q:=1$, and $\beta:=\frac{n}{n-\alpha}$,
we conclude that, for any $f\in \mathcal{L}_{L^{\frac{n}{n-\alpha}}(\rn),1,0}(\rn)$,
\begin{align*}
\lf\|\widetilde{I}_{\alpha}(f)\r\|_{\mathrm{BMO}\,(\rn)}
\ls\|f\|_{\mathcal{L}_{L^{\frac{n}{n-\alpha}}(\rn),1,0}(\rn)}.
\end{align*}
By this and \eqref{4.22x} with $p:=\frac{n}{\alpha}$,
we obtain \cite[Theorem 1]{ws2022} in the case when
$\alpha\in(0,1)\cap (0,\sum_{i=1}^n\frac{1}{p_i})$.
Moreover, assuming $p\in(n,\infty)$ and $\alpha\in(\frac{n}{p},1)\cap(0,\sum_{i=1}^n\frac{1}{p_i})$, and
using Theorem \ref{main-corollary}(i) with both
$X:=L^{\frac{p'n}{p'\alpha+n}}(\rn)$ and $\beta:=\frac{p'\alpha+n}{n}$,
we find that, for any $f\in \mathcal{L}_{L^{p'}(\rn),1,0}(\rn)$,
\begin{align*}
\lf\|\widetilde{I}_{\alpha}(f)\r\|_{\mathcal{L}_{L^{p'n/(p'\alpha+n)}(\rn),1,0}(\rn)}
\ls\|f\|_{\mathcal{L}_{L^{p'}(\rn),1,0}(\rn)},
\end{align*}
where $1/p+1/p'=1$. By this and \eqref{4.22x},
we obtain \cite[Theorem 2]{ws2022} in the case when
$\alpha\in(\frac{n}{p},1)\cap(0,\sum_{i=1}^n\frac{1}{p_i})$.
This shows the above claim. Thus, in this sense, even a special case of
Theorem \ref{apply2} is stronger than the corresponding
cases of \cite[Theorems 1 and 2]{ws2022}.

\item[\rm (ii)]	
To the best of our knowledge, Theorem \ref{apply2} is totally new.
\end{enumerate}
\end{remark}

\subsection{Local Generalized Herz Spaces\label{H-H}}

The local generalized Herz space was originally introduced
by Rafeiro and Samko \cite{RS2020},
which is a generalization of the classical
homogeneous Herz space and connects with the generalized Morrey type space.
Very recently, Li et al. \cite{LYH2022} developed a complete real-variable theory of Hardy spaces
associated with the local generalized Herz space.
We now present the concepts of both the function class
$M\left(\mathbb{R}_{+}\right)$ and the
local generalized Herz space
$\dot{\mathcal{K}}_{\omega, \mathbf{0}}^{p,
q}(\mathbb{R}^{n})$
(see \cite[Definitions 2.1 and 2.2]{RS2020}
and also \cite[Definitions 1.1.1 and 1.2.1]{LYH2022}).

\begin{definition}
Let $\mathbb{R}_+:=(0,\infty)$. The \emph{function class}
$M\left(\mathbb{R}_{+}\right)$ is defined to
be the set of all the positive functions
$\omega$ on $\mathbb{R}_{+}$ such that, for
any $0<\delta<N<\infty$,
$$
0<\inf _{t \in(\delta, N)} \omega(t) \leq
\sup _{t \in(\delta, N)} \omega(t)<\infty
$$
and there exist four constants $\alpha_{0}$,
$\beta_{0}, \alpha_{\infty},
\beta_{\infty} \in \mathbb{R}$ such that
\begin{enumerate}
\item[\rm (i)] for any $t \in(0,1]$, $\omega(t)
t^{-\alpha_{0}}$ is almost increasing and
$\omega(t) t^{-\beta_{0}}$ is almost
decreasing;

\item[\rm (ii)] for any $t \in[1, \infty)$, $\omega(t)
t^{-\alpha_{\infty}}$ is almost increasing
and $\omega(t) t^{-\beta_{\infty}}$ is
almost decreasing.
\end{enumerate}
\end{definition}

\begin{definition}\label{def-Lw}
Let $p,q \in(0, \infty)$ and $\omega \in
M\left(\mathbb{R}_{+}\right)$.
The \emph{local generalized Herz space}
$\dot{\mathcal{K}}_{\omega, \mathbf{0}}^{p,
q}(\mathbb{R}^{n})$ is defined to
be the set of all the measurable functions
$f$ on $\mathbb{R}^{n}$ such that
$$
\|f\|_{\dot{\mathcal{K}}_{\omega, \mathbf{0}}^{p,
q}\left(\mathbb{R}^{n}\right)}:=\left\{\sum_{k \in
\mathbb{Z}}\left[\omega\left(2^{k}\right)\right]^{q}\left\|f
\mathbf{1}_{B\left(\mathbf{0}, 2^{k}\right) \setminus
B\left(\mathbf{0},
2^{k-1}\right)}\right\|_{L^{p}\left(\mathbb{R}^{n}\right)}
^{q}\right\}^{\frac{1}{q}}
$$
is finite.
\end{definition}

\begin{remark}
Let all the symbols be the same as in Definition \ref{def-Lw}.
It has been proved in \cite[Theorems 1.2.20]{LYH2022} that
$\dot{\mathcal{K}}_{\omega,\mathbf{0}}^{p,
r}(\mathbb{R}^{n})$ is a ball quasi-Banach space.	
However, as was pointed out in \cite[Remark 4.15]{cjy-01},
$\dot{\mathcal{K}}_{\omega, \mathbf{0}}^{p,q}(\mathbb{R}^{n})$
may not be a quasi-Banach function space.
\end{remark}

\begin{definition}
Let $\omega$ be a positive
function on $\mathbb{R}_{+}$. Then the
\emph{Matuszewska-Orlicz indices} $m_{0}(\omega)$,
$M_{0}(\omega)$, $m_{\infty}(\omega)$, and
$M_{\infty}(\omega)$ of $\omega$ are
defined, respectively, by setting, for any
$h \in(0, \infty)$,
$$m_{0}(\omega):=\sup _{t \in(0,1)}
\frac{\ln (\varlimsup\limits_{h \to
0^{+}} \frac{\omega(h
t)}{\omega(h)})}{\ln t},\
M_{0}(\omega):=\inf _{t \in(0,1)}
\frac{\ln (
\varliminf\limits_{h\to0^{+}}
\frac{\omega(h t)}{\omega(h)})}{\ln
t},$$
$$m_{\infty}(\omega):=\sup _{t \in(1,
\infty)} \frac{\ln (\varliminf
\limits_{h \to\infty}
\frac{\omega(h t)}{\omega(h)})}{\ln
t},$$
and
$$
M_{\infty}(\omega):=\inf _{t \in(1, \infty)}
\frac{\ln (\varlimsup\limits_{h
\to \infty} \frac{\omega(h
t)}{\omega(h)})}{\ln t}.
$$
\end{definition}

The following theorem is a direct corollary of
Theorems \ref{thm-main-I}, \ref{main-Theorem2}, \ref{main-corollary}, and \ref{dual-I-wI}.

\begin{theorem}\label{apply6}
Let $p,q\in(0,\infty)$ and $\omega\in M(\mathbb{R}_+)$ satisfy
$m_0(\omega)\in(-\frac{n}{p},\infty)$
and $m_\infty(\omega)\in(-\frac{n}{p},\infty)$.	Let
\begin{align}\label{p-lyq}
p_-:=\min\lf\{p, \frac{n}{\max
\left\{M_{0}(\omega),M_{\infty}(\omega)\right\}+n/p}\r\}.
\end{align}
Then Theorems \ref{thm-main-I}, \ref{main-Theorem2},
\ref{main-corollary}, and \ref{dual-I-wI}
with both $X:=\dot{\mathcal{K}}_{\omega,\mathbf{0}}^{p,r}(\rn)$ and $p_-$
in \eqref{p-lyq} hold true.
\end{theorem}

\begin{proof}
Let all the symbols be the same as in
the present theorem. Let
$$
r_0 \in\lf(0, \min \lf\{\frac{1}{\beta},p_-, r\r\}\r)
$$
and
$$
p_0 \in\lf(\max\lf\{p, \frac{n}{\min \{m_{0}(\omega),m_{\infty}(\omega)\}+n / p}\r\}, \infty\r].
$$
By \cite[Theorems 1.2.20 and 1.4.1]{LYH2022},
we conclude that
$\dot{\mathcal{K}}_{\omega,\mathbf{0}}^{p,
r}(\mathbb{R}^{n})$ is a ball quasi-Banach space with an absolutely continuous quasi-norm.
To prove the desired conclusion,
it suffices to show that
$\dot{\mathcal{K}}_{\omega,\mathbf{0}}^{p,r}(\rn)$ satisfies
all the assumptions in Theorems \ref{thm-main-I}, \ref{main-corollary}, and \ref{dual-I-wI}
with $X:=\dot{\mathcal{K}}_{\omega,\mathbf{0}}^{p,r}(\rn)$.
Indeed, using \cite[Lemma 4.3.8]{LYH2022}, we find that both
$\dot{\mathcal{K}}_{\omega,
\mathbf{0}}^{p,r}(\rn)$ and $p_-$ satisfy Assumption \ref{assump1} with
$X:=\dot{\mathcal{K}}_{\omega,\mathbf{0}}^{p,r}(\rn)$.
Thus, all the assumptions of Theorems \ref{thm-main-I},
\ref{main-Theorem2}, \ref{main-corollary}, and \ref{dual-I-wI}
with $X:=\dot{\mathcal{K}}_{\omega,\mathbf{0}}^{p,r}(\rn)$ are satisfied.
Then, using Theorems \ref{thm-main-I}, \ref{main-Theorem2},
\ref{main-corollary}, and \ref{dual-I-wI}
with $X:=\dot{\mathcal{K}}_{\omega,\mathbf{0}}^{p,r}(\rn)$,
we obtain the desired conclusions,
which completes the proof
of Theorem \ref{apply6}.
\end{proof}

\begin{remark}
To the best of our knowledge, Theorem \ref{apply6}
is totally new.
\end{remark}

\subsection{Mixed Herz Spaces\label{M-H}}

We now recall the following definition of mixed Herz spaces, which is just
\cite[Definition 2.3]{zyz2022}.

\begin{definition}\label{mhz}
Let $\vec{p}:=(p_{1},\ldots,p_{n}),\vec{q}:=(q_{1},\ldots,q_{n})
\in(0,\infty]^{n}$, $\vec{\alpha}:=
(\alpha_{1},\ldots,
\alpha_{n})\in\rn$, and $R_{k_i}:=(-2^{k_i},2^{k_i})\setminus
(-2^{k_i-1},2^{k_i-1})$ for any $i\in\{1,\ldots,n\}$ and
$k_i\in\zz$.
The \emph{mixed Herz space}
$\dot{E}^{\vec{\alpha},\vec{p}}_{\vec{q}}(\rn)$ is
defined to be the
set of all the functions
$f\in \mathscr{M}(\rn)$ such
that
\begin{align*}
\|f\|_{\dot{E}^{\vec{\alpha},\vec{p}}_{\vec{q}}
(\rn)}:&=\lf\{\sum_{k_{n} \in
\zz}2^{k_{n}
p_{n}\alpha_{n}}
\lf[\int_{R_{k_{n}}}\cdots\Bigg\{\sum_{k_{1}
\in \zz}
2^{k_{1}p_{1}\alpha_{1}}\r.\r.\\
&\lf.\lf.\lf.\quad\times\lf[\int_{R_{k_{1}}}|f(x_{1},
\ldots,x_{n})|
^{q_{1}}\,dx_{1} \r]^{\f{p_{1}}{q_{1}}}
\r\}^{\f{q_{2}}
{p_{1}}}\cdots
\,dx_{n}  \r]^{\f{p_{n}}{q_{n}}}\r\}
^{\f{1}{p_n}}\\
&=:\,\left\|\cdots\|f\|_{\dot{K}^{\alpha_{1},p_{1}}_{q_{1}}
(\rr)}\cdots\right\|_{\dot{K}^{\alpha_{n},p_{n}}_{q_{n}}(\rr)}
<\infty
\end{align*}
with the usual modifications made when $p_{i}
=\infty$ or
$q_{j}=\infty$ for some $i,j\in\{1,\ldots,n \}$,
where $\|\cdots\|f\|_{\dot{K}^{\alpha_{1},p_{1}}_{q_{1}}(\rr)}
\cdots\|_{\dot{K}^{\alpha_{n},p_{n}}_{q_{n}}(\rr)}$ denotes the
norm obtained after taking successively
the $\dot{K}^{\alpha_{1},p_{1}}_{q_{1}}(\rr)$-norm
to $x_{1}$, the $\dot{K}^{\alpha_{2},p_{2}}_{q_{2}}
(\rr)$-norm to $x_{2}$,
$\ldots$, and
the $\dot{K}^{\alpha_{n},p_{n}}_{q_{n}}(\rr)$-norm to $x_{n}$.
\end{definition}

Recall that, to study the Lebesgue points of functions in mixed-norm Lebesgue spaces,
Huang et al. \cite{HWYY2021} introduced a special case of the mixed Herz space and, later, Zhao
et al. \cite{zyz2022} generalized it to the above case.
Also, both the dual theorem and the Riesz--Thorin
interpolation theorem on the above mixed Herz space have been fully studied in \cite{zyz2022}.

\begin{remark}
Let $\vec{p}:=(p_{1},\ldots,p_{n}),\vec{q}:=(q_{1},\ldots,q_{n})
\in(0,\infty]^{n}$ and $\vec{\alpha}:=(\alpha_{1},\ldots,\alpha_{n})\in\rn$.
By \cite[Proposition 2.22]{zyz2022}, we conclude that
$\dot{E}^{\vec{\alpha},\vec{p}}_{\vec{q}}(\rn)$
is a ball quasi-Banach space if and only if,
for any $i\in\{1,\ldots,n\}$, $\alpha_i\in(-\frac{1}{q_i},\fz)$.
By \cite[Propositions 2.8 and 2.22]{zyz2022}, we conclude that
$\dot{E}^{\vec{\alpha},\vec{p}}_{\vec{q}}(\rn)$
is a ball quasi-Banach space.
However, from \cite[Remark 2.4]{zyz2022}, we deduce that, when $\vec{p}=\vec{q}$ and
$\vec\alpha=\mathbf{0}$, the mixed Herz
space $\dot{E}^{\vec{\alpha},\vec{p}}_{\vec{q}}(\rn)$
coincides with the mixed Lebesgue
space $L^{\vec{p}}(\rn)$ defined in Definition \ref{mixed}.
Using this and Remark \ref{mix-r}, we find that $L^{\vec{p}}(\rn)$
may not be a quasi-Banach function space, and hence
$\dot{E}^{\vec{\alpha},\vec{p}}_{\vec{q}}(\rn)$
may not be a quasi-Banach function space.
\end{remark}

The following theorem is a corollary of Theorems \ref{thm-main-I}, \ref{main-Theorem2},
\ref{main-corollary}, and \ref{dual-I-wI}.

\begin{theorem}\label{apply8}
Let $\vec{p}:=(p_{1},\ldots,p_{n}),\vec{q}:=(q_{1},\ldots,q_{n})
\in(0,\infty)^{n}$ and $\vec{\alpha}:=
(\alpha_{1},\ldots,
\alpha_{n})\in\rn$ satisfy
$\sum_{i=1}^{n}(\frac{1}{p_i}+\alpha_i)\in(\alpha,\infty)$
and $\alpha_i\in(-\frac{1}{q_i},\infty)$ for any
$i\in\{1,\ldots,n\}$.
Let $\beta\in(1,\fz)$ and
$$p_-:=\min\lf\{p_1,\ldots,p_n,
q_1,\ldots,q_n,\lf(\alpha_{1}+\frac{1}{q_{1}}\r)^{-1},
\ldots,\lf(\alpha_{n}+\frac{1}{q_{n}}\r)^{-1} \r \}.$$
\begin{enumerate}
\item[\rm (i)]
Theorem \ref{thm-main-I} with
$X:=\dot{E}^{\vec{\alpha},\vec{p}}_{\vec{q}}(\rn)$ holds true
if and only if $\beta:=\frac{\sum_{i=1}^{n}\frac{1}{p_i}}
{\sum_{i=1}^{n}\frac{1}{p_i}-\alpha}$;

\item[\rm (ii)]
Theorems \ref{main-Theorem2}, \ref{main-corollary},
and \ref{dual-I-wI} with both
$X:=\dot{E}^{\vec{\alpha},\vec{p}}_{\vec{q}}(\rn)$
and $\beta:=\frac{\sum_{i=1}^{n}\frac{1}{p_i}}
{\sum_{i=1}^{n}\frac{1}{p_i}-\alpha}$ hold true.	
\end{enumerate}
\end{theorem}

\begin{proof}
Let all the symbols be the same as in the present
theorem. Let
$$r_0\in\lf(0,\min\lf\{\frac{1}{\beta},p_-\r\}\r),$$
and
$$p_0\in\lf(\max\lf\{p_1,\ldots,p_n,
q_1,\ldots,q_n,\lf(\alpha_{1}
+\frac{1}{q_{1}}\r)^{-1}
,\ldots,\lf(\alpha_{n}+\frac
{1}{q_{n}}\r)^{-1}\r\},\infty \r).$$
By \cite[Propositions 2.8 and 2.22]{zyz2022}, we conclude that
$\dot{E}^{\vec{\alpha},\vec{p}}_{\vec{q}}(\rn)$
is a ball quasi-Banach space with an absolutely continuous quasi-norm.
Using both \cite[Lemma 5.3(i)]{zyz2022} and its proof, we find that both
$\dot{E}^{\vec{\alpha},\vec{p}}_{\vec{q}}(\rn)$ and $p_-$ satisfy Assumption \ref{assump1} with
$X:=\dot{E}^{\vec{\alpha},\vec{p}}_{\vec{q}}(\rn)$.
In addition, from both \cite[Lemma 5.3(ii)]{zyz2022}
and its proof, we infer that Assumption \ref{assump2}
with $X:=\dot{E}^{\vec{\alpha},\vec{p}}_{\vec{q}}(\rn)$ also holds true.
Moreover, using \cite[(4.9.12)]{LYH2022}
with $\omega(t):=t^{\alpha_i}$ for any $t\in(0,\infty)$ and $i\in\{1,\ldots,n\}$, we have
\begin{align*}
\lf\|\mathbf{1}_B\r\|_{\dot{E}^{\vec{\alpha},\vec{p}}_{\vec{q}}(\rn)}
&\gtrsim\|\mathbf{1}_{Q(x,r)}\|_{\dot{E}^{\vec{\alpha},\vec{p}}
_{\vec{q}}(\rn)}\\
&\sim\prod_{i=1}^n\|\mathbf{1}_{(x_i-r,x_i+r)}\|_{\dot{K}
^{\alpha_{i},p_{i}}_{q_{i}}
(\rr)}
\gtrsim r^{\sum_{i=1}^{n}(\frac{1}{p_i}+\alpha_i)},
\end{align*}
where $Q(x,r)$ denotes the
cube with edges parallel to the coordinate axes, center $x\in\rn$,
and edge length $r\in(0,\fz)$.
This further implies that, for any ball $B\in\mathbb{B}(\rn)$, $|B|^{\frac{\alpha}{n}}
\lesssim\|\mathbf{1}_B\|_{\dot{E}^{\vec{\alpha},\vec{p}}_{\vec{q}}(\rn)}^{\frac{\beta-1}{\beta}}$
if and only if
$\beta=\frac{\sum_{i=1}^{n}(\frac{1}{p_i}+\alpha_i)}
{\sum_{i=1}^{n}(\frac{1}{p_i}+\alpha_i)-\alpha}$.
Thus, all the assumptions of
Theorems \ref{thm-main-I}, \ref{main-Theorem2}, \ref{main-corollary}, and \ref{dual-I-wI}
with $X:=\dot{E}^{\vec{\alpha},\vec{p}}_{\vec{q}}(\rn)$ are satisfied.
Then, using Theorems \ref{thm-main-I}, \ref{main-Theorem2}, \ref{main-corollary}, and \ref{dual-I-wI}
with $X:=\dot{E}^{\vec{\alpha},\vec{p}}_{\vec{q}}(\rn)$,
we obtain the desired conclusions,
which completes the proof
of Theorem \ref{apply8}.
\end{proof}

\begin{remark}
To the best of our knowledge, Theorem \ref{apply8}
is totally new.
\end{remark}

\bigskip

\noindent Yiqun Chen, Hongchao Jia and Dachun Yang (Corresponding
author)

\smallskip

\noindent  Laboratory of Mathematics and Complex Systems
(Ministry of Education of China),
School of Mathematical Sciences, Beijing Normal University,
Beijing 100875, The People's Republic of China

\smallskip

\noindent {\it E-mails}: \texttt{yiqunchen@mail.bnu.edu.cn} (Y. Chen)

\noindent\phantom{{\it E-mails:}} \texttt{hcjia@mail.bnu.edu.cn} (H. Jia)

\noindent\phantom{{\it E-mails:}} \texttt{dcyang@bnu.edu.cn} (D. Yang)


\begin{thebibliography}{99}
\bibitem{syy2}
J. Sun, D. Yang and W. Yuan,
Molecular characterization of weak Hardy spaces associated with
ball quasi-Banach function spaces on spaces of homogeneous type
with its application to Littelwood--Paley function characterization,
Submitted.

\vspace{-0.3cm}
\end{thebibliography}
\end{document}